\documentclass[a4paper,10pt]{amsart}
\usepackage[english]{babel}
\usepackage[utf8]{inputenc}
\usepackage[T1]{fontenc}

\usepackage[margin=1in]{geometry}
\usepackage{amssymb}
\usepackage{mathrsfs,relsize}
\usepackage{hyperref}
\usepackage[usenames,dvipsnames]{xcolor}
\hypersetup{colorlinks,%
citecolor=OliveGreen,%
filecolor=magenta,%
linkcolor=RoyalBlue,%
urlcolor=cyan}
\usepackage{enumitem}	
\usepackage{array}		
\usepackage{mathtools}	
\usepackage{tikz, float}
\usetikzlibrary{shadings,shapes}

\newcommand{\scr}[1]{\mathscr{#1}}
\newcommand{\frk}[1]{\mathfrak{#1}}
\newcommand{\bb}[1]{\mathbb{#1}}
\newcommand{\cal}[1]{\mathcal{#1}}

\newcommand{\N}{\mathbb{N}}	
\newcommand{\Z}{\mathbb{Z}}	
\newcommand{\R}{\mathbb{R}}	
\newcommand{\Co}{\mathscr{C}}	
\newcommand{\Id}{\mathrm{Id}}	
\newcommand{\Span}{\mathrm{span}}	
\newcommand{\Panel}{\mathrm{Panel}}

\newcommand{\dd}{\,\mathrm{d}}	
\newcommand{\de}{\partial}		
\newcommand{\inv}{^{-1}}
\DeclareMathOperator*{\Li}{\mathtt{Li}}
\DeclareMathOperator*{\Ls}{\mathtt{Ls}}
\newcommand{\Klim}{\operatorname*{K-lim}}	

\newcommand{\into}{\hookrightarrow}		

\newcommand{\IFF}{\Leftrightarrow}	
\newcommand{\ddx}{\framebox[\width]{$\Rightarrow$} }
\newcommand{\ssx}{\framebox[\width]{$\Leftarrow$} }

\newcommand{\HH}{\bb H}

\newcommand{\BU}{\mathtt{BU}}

\DeclarePairedDelimiter{\norm}{\lVert}{\rVert}
\newcommand{\G}{\mathbb{G}}
\newcommand{\vv}{\mathsf{v}}
\newcommand{\flip}{\Theta}

\usepackage{dsfont}
\newcommand{\one}{{\mathds 1\!}}	

\newcommand{\pD}{\text{\scalebox{-1}[1]{$\mathrm{P}$}\!\!$\mathrm{D}$}} 
\usepackage[normalem]{ulem}

\newcommand{\cvx}{\mathtt{cvx}}

\renewcommand{\G}{\mathbb{G}}
\newcommand{\g}{\mathfrak{g}}

\DeclareMathOperator{\Hom}{Hom} 
\newcommand{\Ball}{B}

\theoremstyle{plain}
\newtheorem{proposition}{Proposition}[section]
\newtheorem{theorem}[proposition]{Theorem}
\newtheorem{lemma}[proposition]{Lemma}

\newtheorem{corollary}[proposition]{Corollary}

\theoremstyle{definition}

\newtheorem{remark}[proposition]{Remark}

\theoremstyle{remark}

\title{
Sub-Finsler horofunction boundaries of the Heisenberg group}
\date{\today}

\author[Fisher]{Nate Fisher}
\address[Fisher]{Department of Mathematics, 503 Boston Ave, Medford, MA}
\email{nathan.fisher@tufts.edu}

\author[Nicolussi Golo]{Sebastiano Nicolussi Golo}
\address[Nicolussi Golo]{Department of Mathematics and Statistics, 40014 University of Jyväskylä, Finland}
\email{sebastiano2.72@gmail.com}

\thanks{S.N.G~has been supported
by the University of Padova STARS Project ``Sub-Riemannian Geometry and Geometric Measure Theory Issues: Old and New'';
by the INdAM – GNAMPA Project 2019 ``Rectifiability in Carnot groups'';
and by the Marie Curie Actions-Initial Training Network ``Metric Analysis For Emergent Technologies (MAnET)'' (n.~607643).}

\begin{document}

\begin{abstract}
We give a complete analytic and geometric description of the horofunction boundary for polygonal sub-Finsler metrics---that is, those that arise as asymptotic cones of word metrics---on the Heisenberg group.  
We develop theory for the more general case of horofunction boundaries in homogeneous groups by connecting 
horofunctions to Pansu derivatives of the distance function.
\end{abstract}
\maketitle

\setcounter{tocdepth}{2}
\tableofcontents


\section{Introduction}

\subsection{Describing the horofunction boundary} 

The study of boundaries of metric spaces has a rich history and has been fundamental in building bridges between the fields of algebra, topology, geometry, and dynamical systems. Understanding the boundary was essential in the proof of Mostow's rigidity theorem for closed hyperbolic manifolds, and boundaries have also been used to classify isometries of metric spaces, to understand algebraic splittings of groups, and to study the asymptotic behavior of random walks.

The simplest and most classical setting for horofunctions is in the study of isometries of the hyperbolic plane. There, the isometry group splits and induces a geodesic flow and a horocycle flow on the tangent bundle; horocycles, or orbits of the horocycle flow, are level sets of horofunctions. The notion has since been abstracted by Busemann, generalized by Gromov, and used by Rieffel, Karlsson--Ledrappier, and many others to derive results in various fields. The horofunction boundary is obtained by embedding a metric space $X$ into the space of continuous real-valued functions on $X$ via the metric, as we will define below.

In this paper, we develop tools to study the horofunction boundary of homogeneous groups, in particular the real Heisenberg group $\HH$. The horofunction boundary of the Heisenberg group has been the subject of study in several publications. Klein and Nicas described the boundary of $\HH$ for the Kor\'anyi and sub-Riemannian metrics \cite{KN-koranyi, KN-cc}, while several others have studied the boundaries of discrete word metrics in the integer Heisenberg group \cite{walsh-orbits, bader-finkel}. In this paper, we aim to understand the horofunction boundary of the real Heisenberg group $\HH$ for a family of polygonal sub-Finsler metrics which arise as the asymptotic cones of the integer Heisenberg group for different word metrics \cite{pansu-thesis}. 

While horofunction boundaries are not (yet) used as widely as visual boundaries or Poisson boundaries, they admit a theory which is useful across several fields including geometry, analysis, and dynamical systems. Whether it is classifying Busemann function, giving explicit formulas for the horofunctions, describing the topology of the boundary, or studying the action of isometries on the boundary, what it means to understand or to describe a horofunction boundary varies significantly between works.

In this paper, as is done in for the $\ell^\infty$ metric on $\R^n$ in \cite{df-stars}, we hope to combine these analytic, topological, and dynamical descriptions while also introducing a more geometric approach. 
In particular, we want to associate a ``direction'' to every horofunction as well as a geometric condition for a sequence of points to induce a horofunction.
In some settings, the horofunction boundary is made up entirely of limit points induced by geodesic rays---or in other words, every horofunction is a Busemann function. 
It is known that in CAT(0) spaces \cite{BH} as well as in polyhedral normed vector spaces \cite{karlsson-horoballs}, the horofunction boundary is composed only of Busemann functions.
This connection between horofunctions and geodesic rays provides a natural notion of directionality to the horofunction boundary, which is not present in settings of mixed curvature, as described in \cite{MR3673666}.  For the model we develop in homogeneous metrics,
sequences converging to a horofunction can often be dilated back to a well-defined point on the unit sphere, which we can then regard as a direction.  
In these sub-Finsler metrics, there are many directions with no infinite geodesics at all, so this provides one of the motivating senses in which the horofunction boundary is a better choice to capture the geometry and dynamics in nilpotent groups.

\subsection{Outline of paper}

For any homogeneous group, we convert the problem of describing the horofunction boundary to a study of directional derivatives, i.e., Pansu derivatives, of the distance function. It suffices to understand Pansu derivatives on the unit sphere.  Therefore, in any homogeneous group where the unit sphere is understood, our method allows a description of the horofunction boundary.

Pansu-differentiable points on the sphere  (i.e., points $p$ at which distance to the origin has a well defined Pansu derivative) can be thought of as {\em directions} of horofunctions.  
Not all horofunctions are directional; the rest are {\em blow-ups} of non-differentiable points.  
Background on homogeneous groups, Pansu derivatives, and horofunctions is provided in \S\ref{sec:prelim}.
We use Kuratowski limits---a notion of set convergence in a metric space---to define the blow-up of a function in \S\ref{sec:kuratowski}.  

In the remainder of the paper, we focus on the Heisenberg group $\HH$. For sub-Riemannian metrics on $\HH$, Klein--Nicas showed that the horofunction boundary is a topological disk \cite{KN-cc}.  In Theorem~\ref{vertSeq} of \S\ref{sec:disk} we show that an analogous disk belongs to the boundary for the larger class of sub-Finsler metrics, but is a proper subset in many cases.  

Our main theorem (Theorem~\ref{thm:main} in \S\ref{sec:subfinsler}) describes the horoboundary of polygonal sub-Finsler metrics on $\HH$ in terms of blow-ups.
From this, we are able to give explicit expressions for the  horofunctions, to describe the topology of the boundary, and to identify Busemann points.

This description is extremely explicit and allows us to visualize the horofunction boundary and to understand it geometrically.
We get a correspondence between ``directions'' on the sphere and functions in the boundary, as indicated in Figure \ref{duality}. This description allows us to realize the horofunction boundary as a kind of dual to the unit sphere, generalizing previous observations for normed vector spaces and for the sub-Riemannian metric on $\HH$ \cite{ji-schilling, df-stars, karlsson-horoballs, walsh-norm, KN-cc}.

\begin{figure}[h]
\centering
\begin{tikzpicture}

\begin{scope}[yshift=4.8cm]
\draw[ultra thick, orange] (0,0) circle (1.1cm);
\end{scope}

\begin{scope}[xshift=3.5cm,yshift=4.8cm]
\draw[ultra thick, orange] (0,0) circle (1.1cm);
\end{scope}

\begin{scope}[xshift=7.5cm,yshift=4.8cm]
\draw[thick, blue] (1,0)--(1,1)--(0,1)--(-1,0)--(-1,-1)--(0,-1)--cycle;
\foreach \x in {(1,0), (1,1), (0,1), (-1,0), (-1,-1), (0,-1)}
\draw [fill=red, draw = none] \x circle (0.08);
\end{scope}

\begin{scope}[xshift=11cm,yshift=4.8cm]
\draw[thick, red] (1,0)--(0,1)--(-1,1)--(-1,0)--(0,-1)--(1,-1)--cycle;
\foreach \x in {(1,0), (-1,1), (0,1), (-1,0), (1,-1), (0,-1)}
\draw [fill=blue, draw = none] \x circle (0.08);
\end{scope}

\node at (0,.3) {\includegraphics[width=1.5in]{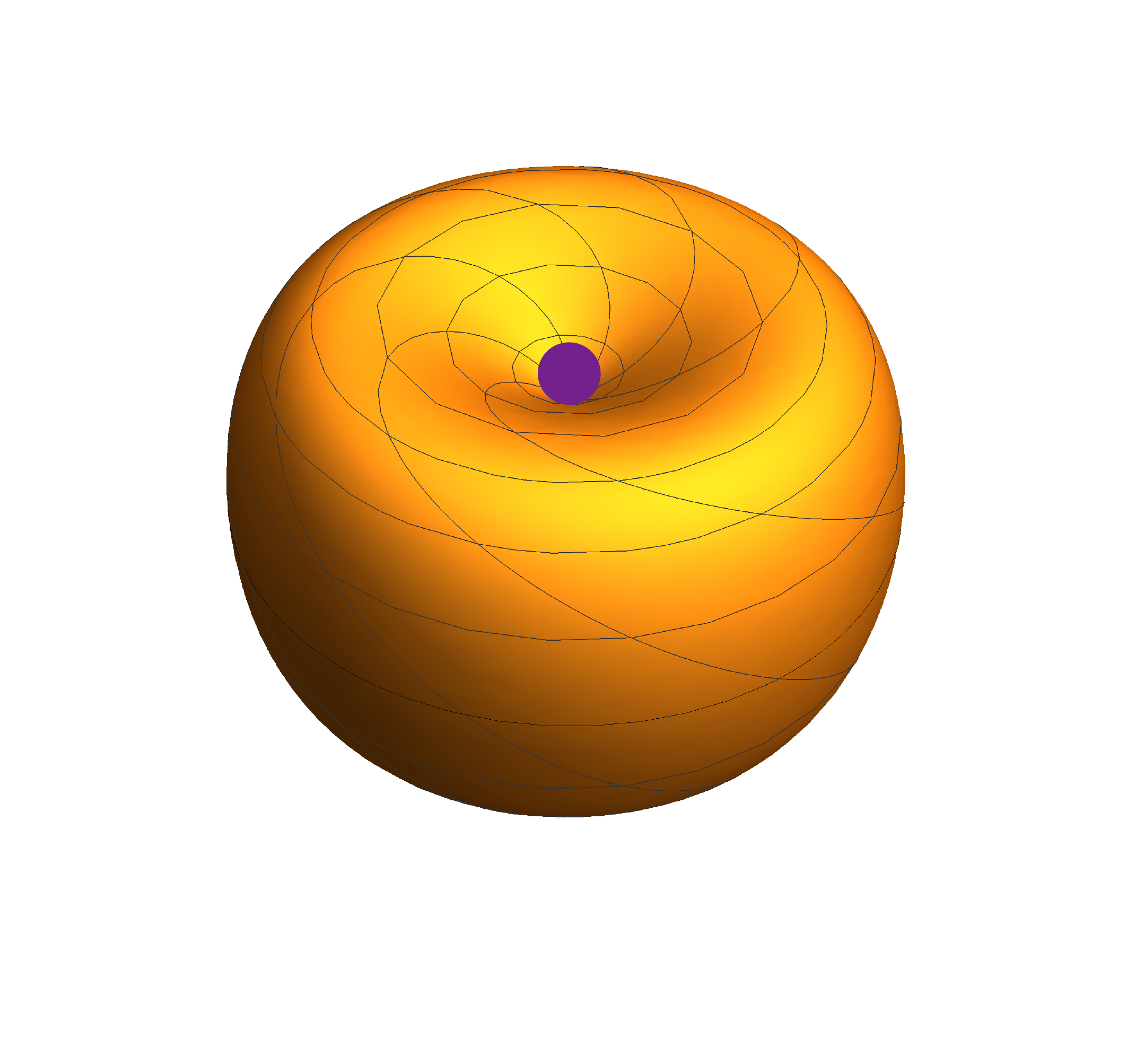}};
\node at (3.5,.5) {\includegraphics[width=1.3in]{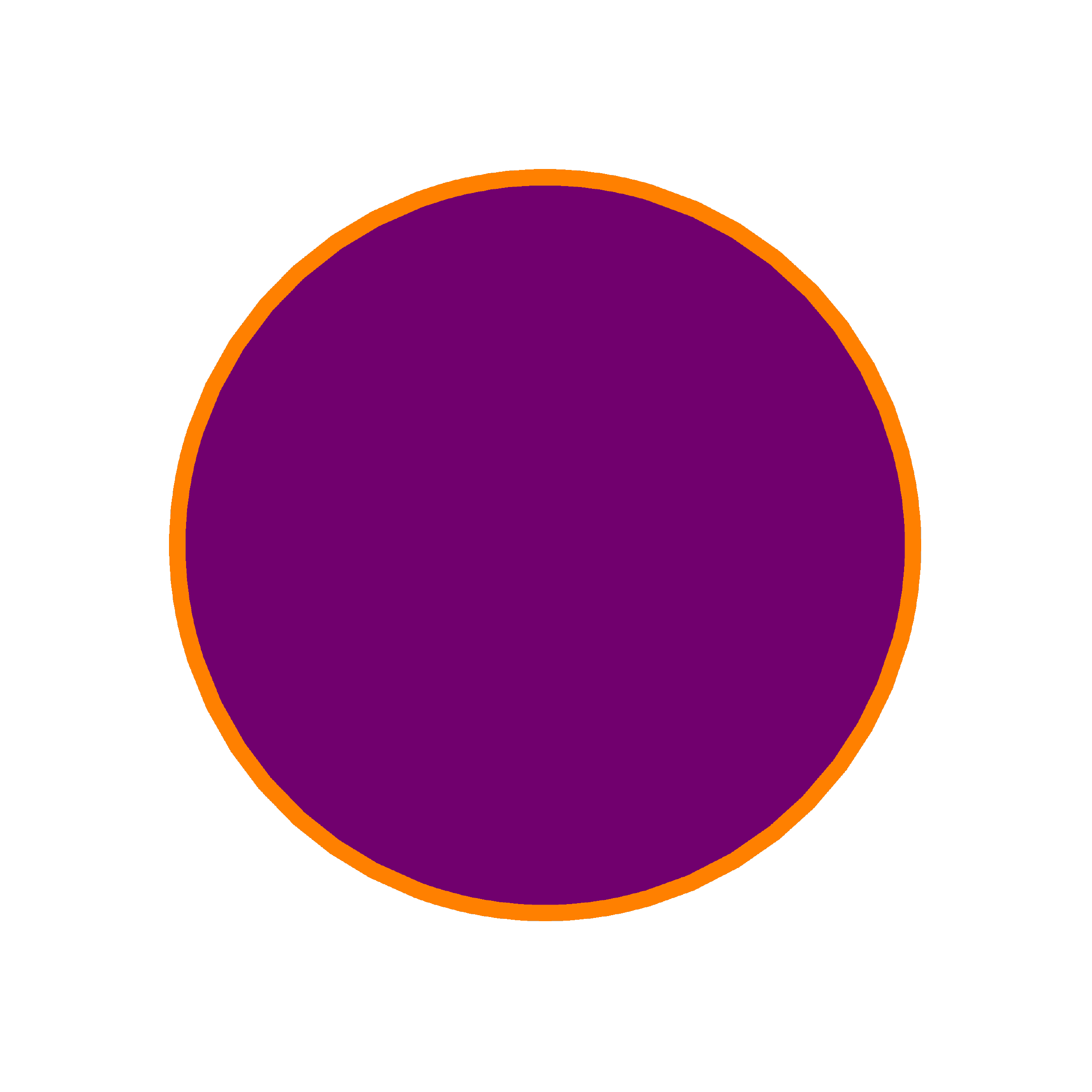}};

\node at (7.5,.3) {\includegraphics[width=1.3in]{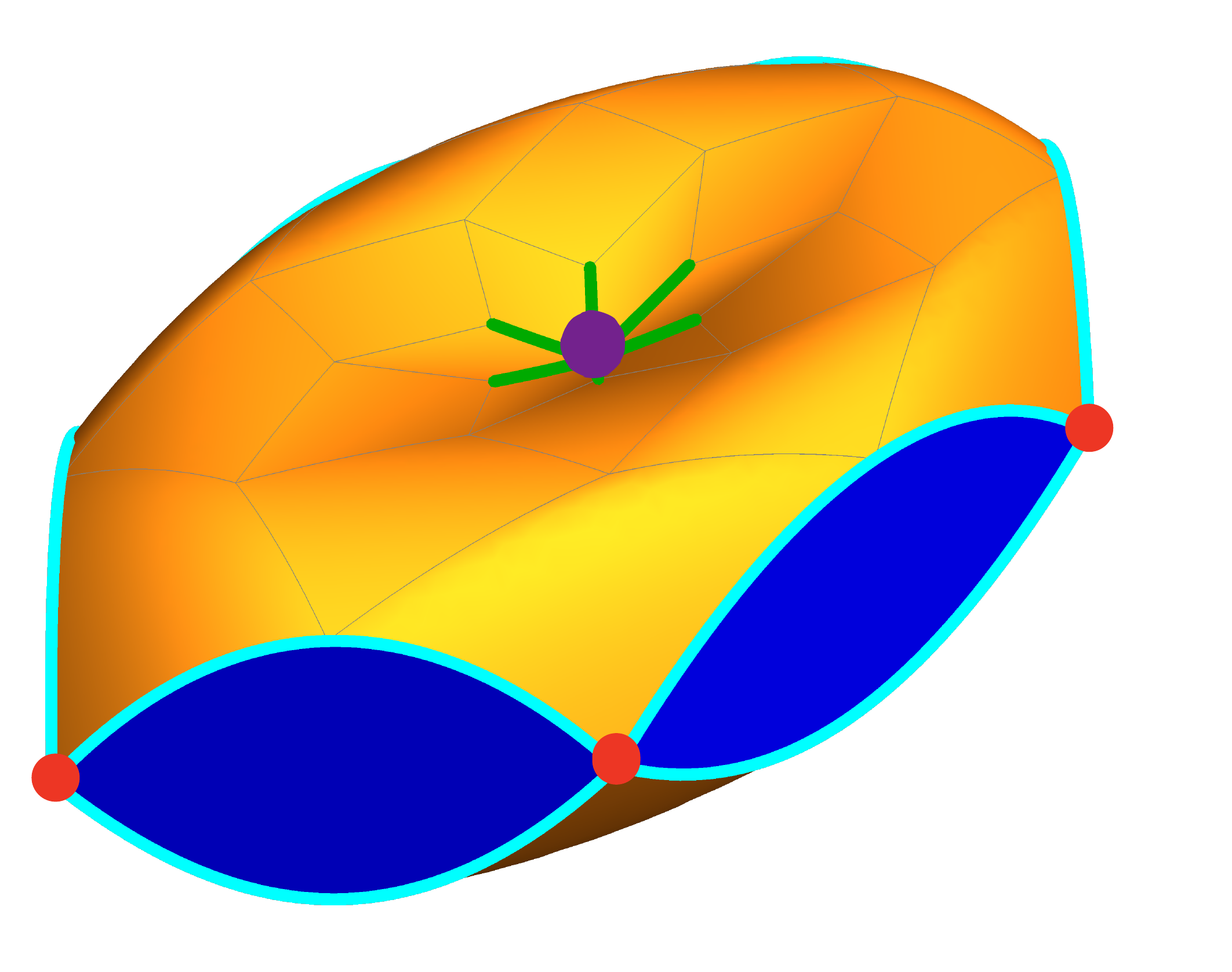}};
\node at  (11,.3) {\includegraphics[width=1.15in]{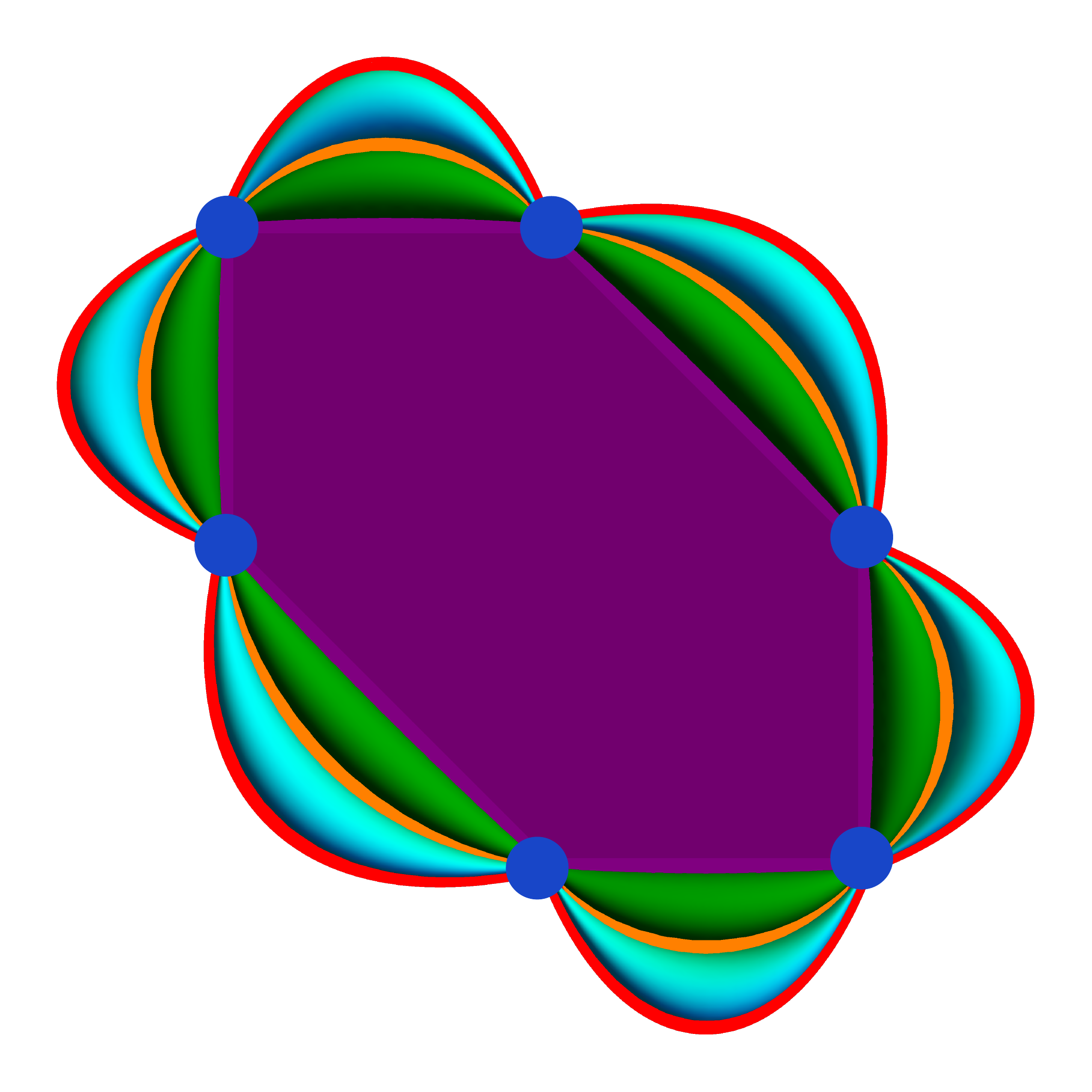}};

\draw (-2,-1.7) rectangle (13, 2.8);
\draw (-2,2.8) rectangle (13, 7.2);
\draw (5.5, -1.7) -- (5.5, 7.2);

\node[draw] at (1.75, 6.8) {$(\R^2, d_{\sf Eucl})$};
\node[draw] at (9.25, 6.8) {$(\R^2, d_{\sf hex})$};
\node[draw] at (1.75, 2.4) {$(\HH, d_{\sf subRiem})$};
\node[draw] at (9.25, 2.4) {$(\HH, d_{\sf subFins})$};

\foreach \x in {(0,-1.3), (0, 3.2), (7.5, -1.3), (7.5, 3.2)}
\node at \x {Sphere};
\foreach \x in {(3.5,-1.3), (3.5, 3.2), (11, -1.3), (11, 3.2)}
\node at \x {Boundary};

\end{tikzpicture}

\caption{The duality between unit spheres and horofunction boundaries for various metric spaces, where colors indicate a correspondence between directions on the spheres and points in the boundary. Note that in both the round and hexagonal cases, the 2D spheres and boundaries embed in the Heisenberg spheres (along the equators) and boundaries.}
\label{duality}
\end{figure}

Finally, using our description of the boundary, we  also study the group action on the boundary in \S\ref{sec:dynamics}, generalizing results of Walsh and Bader--Finkelshtein \cite{walsh-orbits, bader-finkel}.

\subsection*{Acknowledgements} The authors would like to thank Moon Duchin for suggesting the problem and bringing us together to work on this project. We also appreciate the fruitful discussions we have had with Enrico Le Donne, Sunrose Shrestha, and Anders Karlsson.
Finally, we thank Linus Kramer for pointing out to us a common mistake in the definition of horoboundary that we had repeated, see Section~\ref{sec5f907e6b}.

\section{Preliminaries on homogeneous groups and horofunctions}\label{sec:prelim}

We begin with a brief introduction to graded Lie groups, homogeneous metrics, Pansu derivatives, and horofunctions. For a survey on graded Lie groups and homogeneous metrics, we refer the interested reader to \cite{ledonne-primer}.

\subsection{Graded Lie groups}
Let $V$ be a real vector space with finite dimension and  $[\cdot,\cdot]:V\times V\to V$ be the Lie bracket of a Lie algebra $\frk g=(V,[\cdot,\cdot])$.
We say that $\g$ is \emph{graded} if  subspaces $V_1,\dots,V_s$ are fixed so that
$V=V_1\oplus\dots\oplus V_s$ 
and $[V_i,V_j]:=\Span\{[v,w]:v\in V_i,\ w\in V_j\}\subset V_{i+j}$ for all $i,j\in\{1,\dots,s\}$,
where $V_k=\{0\}$ if $k>s$. 
Graded Lie algebras are nilpotent.
A graded Lie algebra is \emph{stratified of step $s$} if equality  $[V_1,V_j]=V_{j+1}$ holds and $V_s\neq\{0\}$.
Our main object of study are stratified Lie algebras, but we will often work with subspaces that are only graded Lie algebras.

On the vector space $V$ we define a group operation via the Baker--Campbell--Hausdorff formula
\begin{align*}
pq &:=
\sum_{n=1}^\infty\frac{(-1)^{n-1}}{n} \sum_{\{s_j+r_j>0:j=1\dots n\}}
	\frac{ [p^{r_1}q^{s_1}p^{r_2}q^{s_2}  \cdots p^{r_n}q^{s_n}] }
	{\sum_{j=1}^n (r_j+s_j) \prod_{i=1}^n r_i!s_i!  } \\
	&=p+q+\frac12[p,q]+\dots ,
\end{align*}
where 
\[
[p^{r_1}q^{s_1}p^{r_2}q^{s_2}  \cdots p^{r_n}q^{s_n}]
= \underbrace{[p,[p,\dots,}_{r_1\text{ times}}
\underbrace{[q,[q,\dots,}_{s_1\text{ times}}
\underbrace{[p,\dots}_{\dots}]\dots]]\dots]] .
\]
The sum in the formula above is  finite  because $\frk g$ is nilpotent.
The resulting Lie group, which we denote by $\G$, is nilpotent and simply connected; we will call it \emph{graded group} or  \emph{stratified group}, depending on the type of grading of the Lie algebra.
The identification $\G=V=\frk g$ corresponds to the identification between Lie algebra and Lie group via the exponential map $\exp:\g\to\G$.
Notice that $p^{-1}=-p$ for every $p\in\G$ and that $0$ is the neutral element of~$\G$.

If $\frk g'$ is another graded Lie algebra with underlying vector space $V'$ and Lie group $\G'$, then, with the same identifications as above, a map $V\to V'$ is a Lie algebra morphism if and only if it is a Lie group morphism, and all such maps are linear.
In particular, we denote by $\Hom_h(\G;\G')$ the space of all \emph{homogeneous morphisms} from $\G$ to $\G'$, that is, all linear maps $V\to V'$ that are Lie algebra morphisms (equivalently, Lie group morphisms) and that map $V_j$ to $V_j'$.
If $\g$ is stratified, then homogeneous morphisms are uniquely determined by their restriction to $V_1$.

For $\lambda>0$, define the \emph{dilations} as the maps $\delta_\lambda:V\to V$ such that $\delta_\lambda v=\lambda^j v$ for $v\in V_j$.
Notice that $\delta_\lambda\delta_\mu=\delta_{\lambda\mu}$
and that $\delta_\lambda\in\Hom_h(\G;\G)$, for all $\lambda,\mu>0$.
Notice also that a Lie group morphism $F:\G\to\G'$ is homogeneous if and only if $F\circ\delta_\lambda = \delta_\lambda'\circ F$ for all $\lambda>0$, where $\delta'_\lambda$ denotes the dilations in $\G'$.
We say that a subset $M$ of $V$ is \emph{homogeneous} if $\delta_\lambda(M)=M$ for all $\lambda>0$.

A \emph{homogeneous distance} on $\G$ is a distance function $d$ that is left-invariant and 1-homogeneous with respect to dilations, i.e., 
\begin{enumerate}
\item[(i)]
$d(gx,gy)=d(x,y)$ for all $g,x,y\in\G$;
\item[(ii)]
$d(\delta_\lambda x,\delta_\lambda y) = \lambda d(x,y)$ for all $x,y\in\G$ and all $\lambda>0$.
\end{enumerate}
When a stratified group $\G$ is endowed with a homogeneous distance $d$, we call the metric Lie group $(\G,d)$ a \emph{Carnot group}.
Homogeneous distances induce the  topology of  $\G$, see \cite[Proposition 2.26]{MR3943324}, and are biLipschitz equivalent to each other.
Every homogeneous distance defines a \emph{homogeneous norm}  $d_e(\cdot): \G \to [0, \infty), d_e(p) = d(e, p)$, where $e$ is the neutral element of $\G$.
We denote by $|\cdot|$ the Euclidean norm in $\R^\ell$. 

\subsection{Pansu derivatives}\label{sec:Pansudifferentiability}
Let $\G$ and $\G'$ be two Carnot groups with homogeneous metrics $d$ and $d'$, respectively, and let $\Omega\subset\G$ open.
A function $f:\Omega\to\G'$ is \emph{Pansu differentiable at $p\in\Omega$} if 
there is $L\in\Hom_h(\G;\G')$ 
such that
\[
\lim_{x\to p}
	\frac{d'( f(p)^{-1}f(x) , L(p^{-1}x) )}{d(p,x)}
= 0 .
\]
The map $L$ is called \emph{Pansu derivative} of $f$ at $p$ and it is denoted by $\pD f(p)$ or $\pD f|_p$.
A map $f:\Omega\to\G'$ is  \emph{of class $C^1_H$}
if $f$ is Pansu differentiable at all points of $\Omega$ and the Pansu derivative $p\mapsto \pD F|_p$ is continuous.
We denote by $C^1_H(\Omega;\G')$ the space of all  maps from $\Omega$ to $\G'$ of class $C^1_H$.

A function $f:\Omega\to\G'$ is \emph{strictly Pansu differentiable at $p\in\Omega$} if 
there is $L\in\Hom_h(\G;\G')$ 
such that
\[
\lim_{\epsilon\to0}\ \sup \left\{
	\frac{d'( f(y)^{-1}f(x) , L(y^{-1}x) )}{d(x,y)}
	: x,y\in\Ball_d(p,\epsilon),\ x\neq y \right\}
= 0,
\]
where $B_d(p, \epsilon)$ is the open $\epsilon$-ball centered at $p$.
Clearly, in this case $f$ is Pansu differentiable at $p$ and $L=\pD F|_p$.
Moreover, as shown in \cite[Proposition 2.4 and Lemma 2.5]{2020arXiv200402520J},
a function $f:\Omega\to\G'$ is of class $C^1_H$ on $\Omega$ if and only if $f$ is strictly Pansu differentiable at all points in $\Omega$.
If $f\in C^1_H(\Omega;\G')$, then $f:(\Omega,d)\to(\G',d')$ is locally Lipschitz.


\subsection{Sub-Finsler metrics}
Let $\G$ be a stratified group and $\|\cdot\|$ a norm on the first layer $V_1 \subset T_e\G$ of the stratification. Using left-translations, we extend the norm $\|\cdot\|$ to the sub-bundle $\Delta\subset T\G$ of left-translates of $V_1$. We call a curve in $\G$ {\em admissible} if it is tangent to $\Delta$ almost everywhere, and using the norm $\|\cdot\|$ we can measure the length of any admissible curve. A classical result tells us that in a stratified group, where $V_1$ bracket-generates the whole Lie algebra, any two points in $\G$ are connected by an admissible curve. We then define a Carnot-Carath\'eodory length metric by 
\[
d(p,q) = \inf_{\gamma} \left\{ \int_a^b \norm{\gamma'(t)}dt\right\},
\]
where the infimum is taken over all admissible $\gamma$ connecting $p$ to $q$.

\begin{proposition}[Eikonal equation]\label{prop5f72ae06}
	If $d$ is a homogeneous distance on $\G$, then $d_e: x \mapsto d(e,x)$ is Pansu differentiable almost everywhere.
	Moreover, if $d$ is sub-Finsler with norm $\|\cdot\|$, then 
	\begin{equation}\label{eq03261818}
	\|\pD d_e |_p\| := 
	\sup\{\vert \pD d_e|_pv\vert : d(e,v)\le 1\}
	=	1
	\qquad\text{for a.e.~}p\in\G .
	\end{equation}
\end{proposition}
\begin{proof}
	Since $d_e$ is 1-Lipschitz, then it is Pansu differentiable almost everywhere by the Pansu--Rademacher Theorem~\cite[Theorem~2]{MR979599} and $\|\pD d_e\| \le 1$.
	To prove~\eqref{eq03261818}, let $p\in\G$ be a point at which $d_e$ is Pansu differentiable, and let $\gamma:[0,T]\to\G$ be a length minimizing curve parametrized by arc-length from $e$ to $p$.
	Since, for $t\neq T$, we have
	\[
	d\left( e , \delta_{1/|t-T|}\left( \gamma(T)^{-1}\gamma(t) \right) \right) = 1 ,
	\]
	then there is a sequence $t_n\to T$ so that $\lim_{n\to\infty}\delta_{1/|t_n-T|}\left( \gamma(T)^{-1}\gamma(t_n) \right) = v$ exists.
	It follows that
	\[
	1 = \lim_{n\to\infty}\frac{|d_e(\gamma(t_n))-d_e(\gamma(T))|}{|t_n-T|}
	= | \pD d_e|_p[v]| ,
	\]
	and we conclude that $\|\pD d_e |_p\|=1$.
\end{proof}


\subsection{The Heisenberg group}\label{sec5f742960}
The Heisenberg group $\HH$ is the simply connected Lie group whose Lie algebra $\mathfrak{h}$ is generated by three vectors $X$, $Y$, and $Z$, with the only nontrivial  Lie bracket $[X,Y] =Z$. 
The stratification is given by  $V_1 = \Span\{X,Y\}$ and $V_2 = \Span\{Z\}$. 
Via the exponential map and the above basis for $\mathfrak{h}$, the Heisenberg group can be coordinatized as $\R^3$ with the following group multiplication:
\[
(x, y, z)(x', y', z') = (x + x', y +y', z + z' + \frac12(xy' - x'y)).
\]
Under this group operation, the generating vectors in the Lie algebra correspond to the left-invariant vector fields
\[
X = \de_x - \frac12 y \de_z, \quad Y = \de_y + \frac12 x \de z, \quad Z = \de _z.
\]
It will sometimes be convenient to coordinatize $\HH$ as $\R^2 \times \R$, in which case the group operation can be written
\[
(v,t)(w,s) = (v + w, t + s + \frac12\omega(v,w)),
\]
where $\omega$ is the standard symplectic form on the plane, $\omega((x,y),(x',y'))=xy'-x'y$.

Denote by $\Delta$ the {\em horizontal distribution}, the sub-bundle (or plane field) generated by the vector fields $X$ and $Y$. 
A curve is admissible if its derivative belongs to $\Delta$.

Let $\pi:\HH\to\R^2$ be the projection of a point to its horizontal components, $\pi(x,y,z) = (x,y)$, which is a group morphism.

Given a path $\gamma: [0,T] \to \R^2$ and an initial height $z_0$, there exists a unique lift to an admissible path $\hat\gamma$ in $\HH$ such that $\hat\gamma$ has height $z_0$ at time zero and $\pi(\hat\gamma) = \gamma$. 
Using Green's theorem and applying an elementary observation, we have that the third component of $\hat\gamma$,
\[
\gamma_3(t) 
= z_0 + \frac12 \int_0^t (\gamma_1\gamma_2' - \gamma_2\gamma_1')(s) \dd s,
\]
is given by the sum of $z_0$ and the balayage area of $\gamma$, i.e., the signed area enclosed by $\gamma$.

Let $d$ be the sub-Finsler metric on $\HH$ induced by a norm $\|\cdot\|$ on $\R^2$ with unit disk $Q$. 
The length in $(\HH,d)$ of an admissible curve $\hat\gamma$ is equal to the length in $(\R^2,\|\cdot\|)$ of the projected curve $\pi(\hat\gamma)$.
A well-known result is that geodesics in sub-Finsler metrics are lifts of solutions to the Dido problem with respect to $\|\cdot\|$; that is, geodesics are lifts of arcs which trace the perimeter of the isoperimetrix $I$ for the given norm.

\subsection{Horoboundary of a metric space}\label{sec5f907e6b}
Let $(X,d)$ be a metric space and $\Co(X)$ the space of continuous functions $X\to\R$ endowed with the topology of the uniform convergence on compact sets.
The map $\iota:X\into\Co(X)$, $(\iota(x))(y):=d(x,y)$, is an embedding, i.e., a homeomorphism onto its image.

Let $\Co(X)/\R$ be the topological quotient of $\Co(X)$ with kernel the constant functions, i.e., for every $f,g\in\Co(X)$ we set the equivalence relation $f\sim g\IFF f-g$ is constant.
For $o\in X$, we set
\[
\Co(X)_{o} := \{f\in\Co(X):f(o)=0\} .
\]
Then the restriction $\Co(X)_{o}\to \Co(X)/\R$ of the quotient map is an isomorphism of topological vector spaces, for each $o\in X$.
Indeed, one easily checks that it is both injective and surjective, and that its inverse map is $[f]\mapsto f-f(o)$, where $[f]\in\Co(X)/\R$ is the class of equivalence of $f\in\Co(X)$, is continuous.

The map $\hat\iota:X\into C(X)\to\Co(X)/\R$ is injective.
Indeed, if $x,x'\in X$ are such that $\iota(x)(z)-\iota(x')(z)$ is constant for all $z \in X$, then taking $z= x$ and then $z=x'$ in turn tells us that $c = d(x, x') = -d(x',x)$. Hence $c = 0$ and $x = x'$.

Moreover, $\hat\iota$ is continuous, but it does not need to be an embedding,
as we learned from~\cite[Proposition 4.5]{kramer}.
In the following lemma, which is a generalization of ~\cite[Remark 4.3]{kramer}, we show that $\hat\iota$ is an embedding under mild conditions on the distance function.

\begin{lemma}\label{lem5f8d4568}
	Let $(X,d)$ be a proper metric space with the following property:
	\begin{equation}\label{eq5f8d3c85}
	\forall p\in X\, \exists 0<r<s\, \forall x\in X\setminus B(p,s)\, \exists z\in B(p,s)\setminus B(p,r)\text{ s.t. }d(x,z)\le d(x,p) .
	\end{equation}
	Then the map $\hat\iota:X\into\Co(X)/\R$ is an embedding.
	
	In particular, any proper metric space with path connected balls satisfy~\eqref{eq5f8d3c85}
	with $r=1$ and $s=2$.
	And so do homogeneous distances on graded groups.
\end{lemma}
\begin{proof}
	We need to show that $\hat\iota$ maps closed sets to closed subsets of $\hat\iota(X)$.
	Let $A\subset X$ closed and $p\in X\setminus A$: we claim that $\hat\iota(p)\notin\text{cl}(\hat\iota(A))$.
	
	Using the isomorphism $\Co(X)_{p}\to \Co(X)/\R$, we can prove the claim for the map $\hat\iota_p:X\to \Co(X)_{p}$, $\hat\iota_p(x)=d(x,\cdot)-d(x,p)$.
	Let $0<r<s$ as in~\eqref{eq5f8d3c85} for this $p$,
	and let $\epsilon>0$ be such that $B(p,2\epsilon)\cap A=\emptyset$.
	We show that, for every $x\in A$,
	\begin{equation}\label{eq5f8d40e8}
	\sup_{z\in \bar B(p,s)} |\hat\iota_p(p)(z) - \hat\iota_p(x)(z)| \ge \min\{4\epsilon,r\}.
	\end{equation}
	Fix $x\in A$.
	First, if $d(p,x)\le s$, then 
	\[
	\hat\iota_p(p)(x) - \hat\iota_p(x)(x)
	= 2d(p,x) \ge 4\epsilon .
	\]
	Second, if $d(p,x)>s$, then let $z$ as in~\eqref{eq5f8d3c85}, so that
	\[
	\hat\iota_p(p)(z) - \hat\iota_p(x)(z)
	= d(p,z) - (d(x,z)-d(x,p))
	\ge d(p,z)
	\ge r .
	\]
	Thus~\eqref{eq5f8d40e8}.
	We conclude from~\eqref{eq5f8d40e8} and the compactness of $\bar B(p,s)$ that $\hat\iota_p(p)\notin\text{cl}(\hat\iota_p(A))$ and thus
	$\hat\iota(p)\notin\text{cl}(\hat\iota(A))$, where $\text{cl}(\cdot)$ denotes the topological closure.
	Hence, $\hat\iota(A)=\text{cl}(\hat\iota(A)\cap\hat\iota(X))$, i.e., $\hat\iota(A)$ is a closed subset of $\hat\iota(X)$.
	This completes the proof of the first part of the lemma.
	
	For the second part, let $(X,d)$ be a proper metric space with path connected balls.
	Set $r=1$ and $s=2$, let $p\in X$ and $x\in X$ with $d(p,x)\ge s$.
	Since $\bar B(x,d(p,x))$ is path connected, there is a continuous curve $\gamma:[0,1]\to \bar B(x,d(p,x))$ with $\gamma(0)=x$ and $\gamma(1)=p$.
	Since $a(t):=t\mapsto d(p,\gamma(t))$ is continuous, $a(0)\ge s$ and $a(1)=0$, then there is $t_0$ with $d(p,\gamma(t_0))\in [r,s]$.
	We conclude that~\eqref{eq5f8d3c85} holds with $z=\gamma(t_0)$.

	Notice that homogeneous distances on graded groups satisfy the above connectedness condition.
\end{proof}

Define the \emph{horoboundary of $(X,d)$} as
\[
\de_h X := cl(\hat\iota(X)) \setminus \hat\iota(X) \subset\Co(X)/\R ,
\]
where $cl(\hat\iota(X))$ is the topological closure.
Another description of the horoboundary is possible,

as we identify $\de_hX$ with a subset of $\Co(X)_{o}$ for some $o\in X$.
More explicitly: $f\in\Co(X)_{o}$ belongs to $\de_hX$ if and only if there is a sequence $p_n\in X$ such that $p_n\to\infty$ (i.e., for every compact $K\subset X$ there is $N\in\N$ such that $p_n\notin K$ for all $n>N$) and the sequence of functions $f_n\in\Co(X)_{o}$,
\begin{equation}\label{eq03251521}
f_n(x) := d(p_n,x)-d(p_n,o),
\end{equation}
converge uniformly on compact sets to $f$.

If $\gamma: [0,\infty) \to X$ is a geodesic ray, one can check that $\lim_{t\to\infty}\hat\iota(\gamma(t))$ exists, and the geodesic ray converges to a horofunction. Indeed, one can check that for each $x$ in a compact set $K$, $\{d(\gamma(t),x) - d(\gamma(t),\gamma(0))\}$ is non-increasing and bounded below.
These horofunctions which are the limits of geodesic rays, {\em Busemann functions}, have been widely studied and inspired the definition of general horofunctions.

\subsection{Horofunctions and the Pansu derivative}
On homogeneous groups, we observe a fundamental connection between horofunctions and Pansu derivatives of the function $d_e:x\mapsto d(e,x)$.

Let $d$ be a homogeneous metric on $\G$ with unit ball $B$ and unit sphere $\de B$.
Again, we denote by $e$ the neutral element of $\G$ and by $d_e$ the function $x\mapsto d(e,x)$.

\begin{lemma}\label{lem:pansuderiv}
	Let $d$ be a homogeneous metric on $\G$.
	If $f\in\de_h(\G,d)$, then there is a sequence $(p_n,\epsilon_n)\in\de B \times(0,+\infty)$ such that $p_n\to p\in \de B$, $\epsilon_n\to0$ and
	\begin{equation}\label{eq03240944}
	f(x) = \lim_{n\to\infty} \frac{d_e(p_n\delta_{\epsilon_n}x)-d_e(p_n)}{\epsilon_n}, 
	\quad\text{locally uniformly in $x\in\G$.}
	\end{equation}
		
	On the other hand, if $(p_n,\epsilon_n)\in\G\times(0,+\infty)$ such that $p_n\to p\in\de B$, $\epsilon_n\to0$ and $f:\G\to\R$ is the locally uniform limit \eqref{eq03240944}, then $f\in\de_h(\G,d)$.
	
	The horofunction $f$ is limit of the sequence of points
	\begin{equation}\label{eq03240951}
	q_n = \delta_{1/\epsilon_n} p_n^{-1} .
	\end{equation}
	Moreover, if $d_e$ is strictly Pansu differentiable at $p$, then $f=\pD d_e|_p$;
	if $p_n\equiv p$ and $d_e$ is Pansu differentiable at $p$, then $f=\pD d_e|_p$.
\end{lemma}
\begin{proof}
	A simple computation shows that,
	if $p_n,q_n\in\G$ and $\epsilon_n\in(0,+\infty)$ satisfy~\eqref{eq03240951}, then 
	\[
	d(q_n,x) - d(q_n,e)
	= \frac{d_e(p_n\delta_{\epsilon_n}x)-d_e(p_n)}{\epsilon_n} .
	\]
	Therefore, if $q_n\to f\in \de_h(\G,d)$, then we take $\epsilon_n:=d(e,q_n)^{-1}$, which converges to $0$, and $p_n=\delta_{\epsilon_n}q_n^{-1} \in \de B$.
	Then \eqref{eq03240944} holds and, up to passing to a subsequence, $p_n$ converges to a point $p\in \de B$. 
	
	The opposite direction is also clear.
\end{proof}

\subsection{Horofunctions on vertical fibers}

From the basic ingredients above, we can deduce that all horofunctions are constant on vertical fibers, when a Lipschitz property holds for $d_e:x \mapsto d(e,x)$.

Notice that, by \cite[Proposition 3.3 and Theorem A.1]{MR3739202}, the Lipschitz property~\ref{eq03241020} is satisfied for all homogeneous distances on $\G$, whenever $\g$ is strongly bracket generating, that is, the stratification $V_1\oplus V_2$ of $\g$ is such that, for every $v\in V_1\setminus\{0\}$, $[v,V_1]=V_2$.
The Heisenberg group $\HH$ is an example of such groups.

\begin{proposition}[Vertical invariance of horofunctions]\label{prop5f749836}
	Suppose that $\G$ is a Carnot group and $d$ a homogeneous distance satisfying 
	\begin{equation}\label{eq03241020}
	\text{there is $L>0$ such that }
	|d_e(x)-d_e(y)| \le L \rho(x,y)
	\text{ for all }
	x,y\in B(e,2)\setminus B(e,1/2) ,
	\end{equation}
	for some Riemannian distance $\rho$ on $\G$.
	Then, horofunctions of $(\G,d)$ are constant along the cosets of the center $[\G,\G]$.
	In particular, for every $f\in\de_h(\G,d)$ there is 
	$\hat f\in C(\G/[\G,\G])$ such that $f=\hat f\circ\pi$.
\end{proposition}
\begin{proof}
	Let $\rho$ be a left-invariant Riemannian metric on $\G$.
	Recall that, by the Ball-Box Theorem~\cite{mitchell-thesis,mitchell-cc, gershkovich,NSW}, if $\zeta\in[\G,\G]$ then $\lim_{\epsilon\to0^+}\frac{\rho(e,\delta_\epsilon\zeta)}{\epsilon}=0$.
	Now, fix $f\in\de_h(\G,d)$, and let $p_n\in\de B$ and $\epsilon_n\to0$ as in Lemma~\ref{lem:pansuderiv}.
	Then, for every $\zeta\in[\G,\G]$ and $x\in\G$, 
	\begin{align*}
	f(x\zeta)-f(x) 
	&= \lim_{n\to\infty} \frac{d_e(p_n\delta_{\epsilon_n}(x\zeta)) - d_e(p_n\delta_{\epsilon_n}x)}{\epsilon_n} \\
	&= \lim_{n\to\infty} \frac{d((p_n\delta_{\epsilon_n}x)^{-1},\delta_{\epsilon_n}\zeta) - d((p_n\delta_{\epsilon_n}x)^{-1},e)}{\epsilon_n}\\
	&\le L \limsup_{n\to\infty} \frac{\rho(e,\delta_{\epsilon_n}\zeta)}{\epsilon_n}
	=0 .\qedhere
	\end{align*}
\end{proof}

\begin{remark}
	We give an example where horofunctions are not constant along the center.
	Endow the stratified group $\G=\HH\times\R$ with a homogeneous distance of the form
	\[
	d((0,0,0;0) , (x,y,z;t) ) = |x|+|y|+c\sqrt{|z|}+|t|
	\]
	with $c>0$ chosen so that $d$ satisfies the triangular inequality.
	Using the notation of the above proof, take
	\begin{align*}
	p_n &= (0,0,1/n;1-c/\sqrt{n}), &
	x &= 0, &
	\zeta &= (0,0,1;0) , &
	\epsilon_n &= 1/\sqrt{n} .
	\end{align*}
	Then, $d(e,p_n)=1$ for all $n$ and
	\[
	\frac{d_e(p_n\delta_{\epsilon_n}(x\zeta)) - d_e(p_n\delta_{\epsilon_n}x)}{\epsilon_n}
	= c (\sqrt2-1) \neq 0
	\]
	for all $n$.
	Finally, a subsequence of $q_n=\delta_{1/\epsilon_n}p_n$ converges to a horofunction $f$ which satisfies $f(\zeta)-f(0)\neq0$, i.e., it is not constant along $[\G,\G]$.
\end{remark}


\section{Blow-ups of sets and functions in homogeneous groups}\label{sec:kuratowski}

As we observed in Lemma \ref{lem:pansuderiv}, in homogeneous groups there is a connection between horofunctions in the boundary and directional derivatives along the unit sphere. Wherever the unit sphere is smooth, this directional derivative is the Pansu derivative. While the unit sphere is Pansu differentiable almost everywhere, the nonsmooth points must be studied using a different strategy. In this section, we overview the Kuratowski convergence of closed sets, sometimes credited to Kuratowski--Painlev\'e, and we use it define the blow-up of functions.

\subsection{Kuratowski limits in metric spaces}
\newcommand{\CL}{\mathtt{CL}}
Let $(X,d)$ be a locally compact metric space and let $\CL(X)$ be the family of all closed subsets of $X$.
If $x\in X$ and $C\subset X$, we set $d(x,C):=\inf\{d(x,y):y\in C\}$.
The {\em Kuratowski limit inferior} of a sequence $\{C_n\}_{n\in\N}\subset\CL(X)$ is defined to be
\begin{align*}
\Li_{n\to \infty} C_n 
&:= \left\{ q \in X : \limsup_{n\to \infty} d(q, C_n) = 0 \right\} \\
&= \left\{q\in X: \forall n\in\N\;\exists x_n\in C_n\text{ s.t.~}\lim_{n\to\infty}x_n=q \right\},
\end{align*}
while the {\em Kuratowski limit superior} is defined to be
\begin{align*}
\Ls_{n\to \infty} C_n 
&:= \left\{ q \in X : \liminf_{n\to \infty} d(q, C_n) = 0\right\} \\
&= \left\{q\in X: \exists N\subset\N\text{ infinite }\forall k\in N\;\exists x_{k}\in C_{k}\text{ s.t.~}\lim_{k\to\infty} x_{k}=q \right\},
\end{align*}
It is clear that  $\Li_n C_n \subseteq \Ls_n C_n$ and that they are both closed.

If $\Li C_n = \Ls C_n = C$, then we say that the $C$ is the {\em Kuratowski limit} of $\{C_n\}_n$ and we write
\[
C=\Klim_{n\to\infty} C_n .
\]

If, for all $n\in\N$, $\Omega_n\subset X$ are closed sets and $f_n:\Omega_n\to\R$ continuous functions, then we say that, for some $\Omega\subset X$ closed and $f:\Omega\to\R$ continuous,
\[
\Klim_{n\to\infty}(\Omega_n,f_n) = (\Omega,f)
\]
if $\Omega=\Klim_n\Omega_n$ and if, for every $x\in\Omega$ and every sequence $\{x_n\}_{n\in\N}$ with $x_n\in\Omega_n$ and $x_n\to x$, we have $f(x)=\lim_nf_n(x_n)$.
Notice that this is equivalent to say that 
\[
\Klim_{n\to\infty} \{(x,f_n(x)):x\in\Omega_n\} = \{(x,f(x)):x\in\Omega\} .
\]

If $C^1_n,\dots,C^J_n$ are sequences of closed sets,

then one easily checks that 
\[
\Ls_{n\to\infty} \bigcup_{j=1}^J C^j_n 
\subset \bigcup_{j=1}^J \Ls_{n\to\infty}  C^j_n ,
\quad\text{and}\quad
\bigcup_{j=1}^J \Li_{n\to\infty}  C^j_n 
\subset \Li_{n\to\infty} \bigcup_{j=1}^J   C^j_n .
\]
Therefore, if the limit $\Klim_{n\to\infty} C^j_n$ exists for each $j$, then we have
\begin{equation}\label{eq03271215}
\Klim_{n\to\infty} \bigcup_{j=1}^JC^j_n
= \bigcup_{j=1}^J\Klim_{n\to\infty} C^j_n .
\end{equation}

It is a classical result of Zarankiewicz that under mild conditions, $\CL(X)$ is sequentially compact with respect to Kuratowski convergence.

\begin{theorem}[Zarankiewicz \cite{zarankiewicz}]\label{thm5f74be88}
	If $(X,d)$ is a separable metric space, then the family of closed sets is sequentially compact with respect to the Kuratowski convergence, that is, if $\{C_n\}_{n\in\N}$ is a sequence of closed sets, then there is $N\subset\N$ infinite and $C\subset X$ closed such that $\Klim_{N\ni n\to\infty} C_n = C$.
\end{theorem}

For $\epsilon\ge0$ and $C\subset X$, let 
\[
\cal N_\epsilon(C):=\{x:d(x,C)\le\epsilon\},
\text{ and }
\cal N_{-\epsilon}(C) :=\{x:d(x,X\setminus C)>\epsilon\} .
\]
Notice that $\cal N_{-\epsilon}(C) = X\setminus \cal N_\epsilon(X\setminus C)$. 

A set $C\subset X$ is a \emph{regular closed set} if it is the closure of its interior.
If $C$ is a closed set, then $\overline{X\setminus C}$ is regular closed.
If $C$ is a regular closed set, then 
\[
C = \bigcap_{\epsilon>0} \cal N_\epsilon(C)
= \overline{\bigcup_{\epsilon>0} \cal N_{-\epsilon}(C)}
\]

\begin{lemma}\label{lem04031714}
	Assume $X$ to be locally compact.
	Let $f_n:X\to\R$ be a sequence of continuous functions locally uniformly converging to $f_\infty:X\to\R$.
	Then
	\begin{equation}\label{eq04031645}
	\{f_\infty<0\}
	\subset \Li_{n\to\infty}\{f_n\le 0\}
	\subset \Ls_{n\to\infty}\{f_n\le 0\}
	\subset \{f_\infty\le 0\} .
	\end{equation}
	In particular, if $\overline{\{f_\infty<0\}}=\{f_\infty\le 0\}$, then
	\[
	\Klim_{n\to\infty} \{f_n\le 0\} = \{f_\infty\le 0\} .
	\]
\end{lemma}
\begin{proof}
	For the first inclusion in~\eqref{eq04031645}, let $p\in X$ with $f_\infty(p)<-\epsilon<0$ for some $\epsilon\le 0$.
	Then there is $r>0$ such that $\bar B(p,r)$ is compact and $f_\infty(x)<-\epsilon$ for all $x\in \bar B(p,r)$.
	By the uniform convergence on compact sets, there exist $N\in\N$ such that $f_n(p)<-\epsilon/2<0$ for all $n>N$.
	Therefore, $p\in \Li_{n\to\infty}\{f_n\le 0\}$.
	For the third inclusion in~\eqref{eq04031645}, consider a sequence $\{p_n\}_{n\in\N}\subset X$ with $p_n\to p$ and $f_n(p_n)\le 0$.
	Then, by the uniform convergence on compact sets, we have $\lim_n f_n(p_n) = f(p)$ and thus $f(p)\le 0$.
	The last statement is a direct consequence of the fact that Kuratowski superior and inferior limits are both closed.
\end{proof}

A family $\scr F\subset \R^X$ is \emph{strictly monotone} if for every $p\in X$ there exists $\gamma_p:[-1,1]\to X$ continuous with $\gamma_p(0)=p$ such that $t\mapsto f(\gamma_p(t))$ is strictly increasing for every $f\in\scr F$.

\begin{lemma}\label{lem04031738}
	If $\scr F\subset C(X)$ is strictly monotone and finite,

	then
	\[
	\overline{\{\max\scr F<0\}} = \{\max\scr F\le 0\} .
	\]
\end{lemma}
\begin{proof}
	Let $p\in\{\max\scr F\le 0\}$ with $\max\scr F(p)=0$.
	Let $\gamma_p:[-1,1]\to X$ be continuous with $\gamma_p(0)=p$ such that $t\mapsto f(\gamma_p(t))$ is strictly increasing for every $f\in\scr F$.
	It follows that, for every $f\in\scr F$ and $t<0$, we have
	$f(\gamma(t)) <f(\gamma(0)) \le \max\scr F(p) =0$.
	Then $p_n=\gamma(-1/n)$ is a sequence of points converging to $p$ with $\max\scr F(p_n)<0$.
	We conclude that $p\in \overline{\{\max\scr F<0\}}$.
\end{proof}

\begin{lemma}\label{lem5e8c350a}
	Assume that $X$ is locally compact.
	For each $j$ integer between $1$ and $J\in\N$, let $\{f^j_n\}_{n\in\N}$ be a sequence of continuous functions $f^j_n:X\to\R$ converging uniformly on compact sets to $f^j_\infty:X\to\R$.
	Then the sequence of continuous functions $g_n:=\max\{f^j_n\}_j$ converges uniformly on compact sets to $g_\infty:=\max\{f^j_\infty\}_j$.
	
	Moreover, if $\{f^j_\infty\}_{j=1}^J$ is strictly monotone,
	then
	\begin{equation}\label{eq5f719ad9}
	\Klim_{n\to\infty} \bigcap_{j=1}^J \{f^j_n\le0\} = \bigcap_{j=1}^J \{f^j_\infty\le0\} = \{g_\infty \leq 0\} .
	\end{equation}
\end{lemma}
\begin{proof}
	We give a proof only for $J=2$:
	the general case can then be proved by induction.

	So, we assume $J=2$.
	Let $K\Subset X$, $\epsilon>0$, and let $(f_1\vee f_2)(x) = \max\{f_1(x),f_2(x)\}$.
	Then there is $N\in\N$ such that $|f^j_n(x)-f^j_\infty(x)|<\epsilon$ for all $x\in K$ and $j$.
	We claim that $|(f^1_n\vee f^2_n)(x) - (f^1_\infty\vee f^2_\infty)(x)| < 3\epsilon$ for all $x\in K$.
	To prove the claim we need to check four cases, which by symmetry reduce to the following two:
	In the first case, $(f^1_n\vee f^2_n)(x)=f^1_n(x)$ and $(f^1_\infty\vee f^2_\infty)(x)=f^1_\infty(x)$.
	Then clearly $|(f^1_n\vee f^2_n)(x) - (f^1_\infty\vee f^2_\infty)(x)| <\epsilon$.
	In the second case, $(f^1_n\vee f^2_n)(x)=f^1_n(x)$ and $(f^1_\infty\vee f^2_\infty)(x)=f^2_\infty(x)$.
	Notice that 
	\begin{align*}
	0
	&\le f^2_\infty(x) - f^1_\infty(x) \\
	&\le f^2_\infty(x) - f^2_n (x) + f^2_n(x) - f^1_n(x) + f^1_n(x) - f^1_\infty(x) \\
	&\le f^2_\infty(x) - f^2_n (x) + f^1_n(x) - f^1_\infty(x)
	\le 2\epsilon .
	\end{align*}
	Therefore, 
	\[
	|(f^1_n\vee f^2_n)(x) - (f^1_\infty\vee f^2_\infty)(x)|
	= |f^1_n(x)-f^2_\infty(x)|
	\le |f^1_n(x)-f^1_\infty(x)| + |f^1_\infty(x)-f^2_\infty(x)|
	\le 3\epsilon .
	\]
	This proves the claim and the first part of the lemma.
	
	For the equalities in~\eqref{eq5f719ad9},
	notice that $\overline{\{g_\infty<0\}}=\{g_\infty\le 0\}$ by the strict monotonicity and Lemma~\ref{lem04031738}.
	Thus, we conclude~\eqref{eq5f719ad9} from Lemma~\ref{lem04031714}.

\end{proof}

\subsection{Blow-ups of sets in homogeneous groups}
Let $\G$ be a homogeneous group with a homogeneous distance $d$.
If $\Omega\subset\G$ is closed, $\{p_n\}_{n\in\N}\subset\G$ and $\{\epsilon_n\}_{n\in\N}\subset(0,+\infty)$ are sequences, we define the \emph{blow-up set}
\[
\BU(\Omega,\{p_n\}_n,\{\epsilon_n\}_n) 
:= \Klim_{n\to \infty} \delta_{1/\epsilon_n}(p_n\inv \Omega) ,
\]
if it exists.
We sometimes use also the \emph{intermediate blow-up sets} 
\begin{align*}
\BU^-(\Omega,\{p_n\}_n,\{\epsilon_n\}_n) 
&:= \Li_{n\to \infty} \delta_{1/\epsilon_n}(p_n\inv \Omega) ,\\
\BU^+(\Omega,\{p_n\}_n,\{\epsilon_n\}_n) 
&:= \Ls_{n\to \infty} \delta_{1/\epsilon_n}(p_n\inv \Omega) ,
\end{align*}
which are always well defined and
\begin{equation}\label{eq03241524}
\BU^-(\Omega,\{p_n\}_n,\{\epsilon_n\}_n) \subset \BU^+(\Omega,\{p_n\}_n,\{\epsilon_n\}_n) .
\end{equation}

\begin{proposition}\label{prop03241537}

	Let $\Omega\subset\G$ be a nonempty closed set, $\{p_n\}_{n\in\N}\subset\G$ and $\{\epsilon_n\}_{n\in\N}\subset(0,+\infty)$ sequences with $\epsilon_n\to 0$.
	\begin{enumerate}
	\item
	$\BU^+(\Omega,\{p_n\}_n,\{\epsilon_n\}_n)\neq\emptyset$, 
	if and only if $\liminf_{n\to\infty} \frac{d(p_n,\Omega)}{\epsilon_n} < \infty$.
	\item
	If $\BU^-(\Omega,\{p_n\}_n,\{\epsilon_n\}_n)\neq\G$,
	then $\limsup_{n\to\infty} \frac{d(p_n,\G\setminus\Omega)}{\epsilon_n} < \infty$.
	\end{enumerate}
	In particular, in case $p_n\to p$ then we have:
	\begin{enumerate}[label=(\arabic*')]
	\item
	If $p\notin\Omega$, then $\BU(\Omega,\{p_n\}_n,\{\epsilon_n\}_n)=\emptyset$.
	\item
	If $p\in \Omega^\circ$, then $\BU(\Omega,\{p_n\}_n,\{\epsilon_n\}_n)=\G$.
	\end{enumerate}
\end{proposition}
\begin{proof}
	(1) \ddx
	Let $q \in \BU^+(\Omega, \{p_n\}_n, \{\epsilon_n\}_n)$. 
	Then there exists $N\subset\N$ infinite and a sequence $\{x_{k}\}_{k\in N}\subset\Omega$ such that $q = \lim_{k\to\infty}\delta_{1/\epsilon_{k}}(p_{k}\inv x_{k})$.
	Therefore,
	\[
	\liminf_{n\to\infty} \frac{d(p_n,\Omega)}{\epsilon_n}
	\le \liminf_{N\ni k\to\infty} \frac{d(p_k,x_k)}{\epsilon_k}
	= \liminf_{N\ni k\to\infty}  d(e,\delta_{1/\epsilon_{k}}(p_{k}\inv x_{k}))
	= d(e,q) .
	\]
	\ssx
	Let $N\subset\N$ infinite and a sequence $\{x_{k}\}_{k\in N}\subset\Omega$
	such that $\lim_{N\ni k\to\infty} \frac{d(p_n,x_k )}{\epsilon_n} < \infty$.
	Since $\frac{d(p_n,x_k )}{\epsilon_n} = d(e,\delta_{\epsilon_{k}\inv}(p_{k}\inv x_{k}))$, we can assume, up to passing to a subsequence, that the limit $\lim_{N\ni k\to\infty}\delta_{1/\epsilon_{k}}(p_{k}\inv x_{k})$ esists.
	Thus, $\BU^+(\Omega,\{p_n\}_n,\{\epsilon_n\}_n)\neq\emptyset$.
		
	(2) 
	Let $q\in\G\setminus\BU^-(\Omega,\{p_n\}_n,\{\epsilon_n\}_n)$ and define $x_n:=p_n\delta_{\epsilon_n}q$.
	Since $q\notin \BU^-(\Omega,\{p_n\}_n,\{\epsilon_n\}_n)$, there is $N\subset\N$ infinite such that $x_k\notin\Omega$ for all $k\in N$. Therefore,
	\[
	\limsup_{n\to\infty} \frac{d(p_n,\G\setminus\Omega)}{\epsilon_n}
	\le \limsup_{n\to\infty} \frac{d(p_n,p_n\delta_{\epsilon_n}q)}{\epsilon_n}
	= d(e,q) .
	\]
\end{proof}

\begin{proposition}\label{prop5e8c49ee}
	Let $\Omega\subset\G$ be a nonempty closed set and $p\in\de\Omega$.

	Suppose that there exists a neighborhood $U$ of $p$ and a finite family of continuous functions $F_j:U\to\R$ with $j\in J$ finite such that $\Omega\cap U = \bigcap_{j\in J}\{F_j\le0\}$ and $F_j(p) = 0$ for all $j$.
	Suppose also that each $F_j$ is strictly Pansu differentiable at $p$ and that
	\begin{equation}\label{eq5e8c33f1}
	0 \notin \cvx\{\pD F_j|_p\}_{j\in J} .
	\end{equation}
	
	Let $p_n\to p$ and $\epsilon_n\to 0^+$, and assume that $\BU(\Omega,\{p_n\}_n,\{\epsilon_n\}_n)$ exists.
	Then
	\[
	\BU(\Omega,\{p_n\}_n,\{\epsilon_n\}_n)
	= \{x\in\G: \pD F_j|_p(x)\le t_j,\ j\in J\}
	\]
	with $t_j\in\R\cup\{-\infty,+\infty\}$ defined as follows:
	\begin{enumerate}
	\item
	if $\lim_n\frac{d(p_n,\{F_j\le0\})}{\epsilon_n} = +\infty$, then $t_j=-\infty$;
	\item
	if $\lim_n\frac{d(p_n,\G\setminus \{F_j\le0\})}{\epsilon_n} = +\infty$, then $t_j=+\infty$;
	\item
	otherwise, there are $q_n^j\in\{F_j=0\}$ such that, up to a subsequence, $\lim_n\delta_{1/\epsilon_n}((q_n^j)\inv p_n) =: v_j$, and we set $t_j=-\pD F_j|_p(v_j)$.	
	\end{enumerate}
\end{proposition}
\begin{proof}
	Let $p_n\to p$ and $\epsilon_n\to 0^+$, assume that $\BU(\Omega,\{p_n\}_n,\{\epsilon_n\}_n)$ exists.
	
	If there is any $j$ such that 
	$\lim_n\frac{d(p_n,\{F_j\le0\})}{\epsilon_n} = +\infty$
	then $\BU(\Omega,\{p_n\}_n,\{\epsilon_n\}_n)=\emptyset$ by Proposition~\ref{prop03241537}.
	
	Again by Proposition~\ref{prop03241537}, for any $j$ such that $\lim_n\frac{d(p_n,\G \setminus \{F_j\le0\})}{\epsilon_n} = +\infty$, we know $\BU(\{F_j\le0\},\{p_n\}_n,\{\epsilon_n\}_n)=\G = \{\pD F_j |_p\le +\infty\}$.
	
	Let $\hat J$ be the set of indices which do not fall into the first two cases. For all $j \in \hat J$, there are $q_n^j\in\{F_j=0\}$ with $d(p_n,q_n^j) = d(p_n,\{F_j=0\})$ and $\limsup_n\frac{d(p_n,q_n^j)}{\epsilon_n} <\infty$.
	Up to a subsequence, we can assume that the limit $\lim_n\delta_{1/\epsilon_n}(q_n\inv p_n) = v_j$ exists.
	Define 
	\begin{align*}
	f_n^j(x) 
	&:= \frac{F_j(p_n\delta_{\epsilon_n}x)}{\epsilon_n},
	\end{align*}
	and note that near $\delta_{1/\epsilon_n}(p_n\inv p)$, the locus $\bigcap\nolimits_j\{f_n^j \leq 0\}$ is a local description of the translated and dilated set $\delta_{1/\epsilon_n}(p_n\inv \Omega)$ for all $n>0$. We then observe that
	\[
	f_n^j(x) = \frac{F_j(p_n\delta_{\epsilon_n}x)-F(p_n)}{\epsilon_n} + \frac{F(p_n)-F(q_n)}{\epsilon_n} .
	\]
	By the strict Pansu differentiability of $F_j$ at $p$, the functions $f_n^j$ converge uniformly on compact sets to $f_\infty^j(x) := \pD F_j|_p(x) + \pD F_j|_p(v_j)$.
	
	Condition~\eqref{eq5e8c33f1} implies that there is $w\in V_1$ such that $\pD F_j|_p(w)>0$ for all $j$.
	Define $\gamma(t) = p\exp(tw)$.
	Then 
	\[
	\frac{\dd}{\dd t}f_\infty^j(\gamma(t)) = \pD F_j|_p(w) ,
	\] 
	which is strictly positive for $t$ in a neighborhood of $0$, for all $j\in \hat J$.
	Therefore, the family of functions $\{f_\infty^j\}_{j\in \hat J}$ is strictly monotone and we conclude by Lemma~\ref{lem5e8c350a} that
	\[
	\BU(\Omega,\{p_n\}_n,\{\epsilon_n\}_n)= \bigcap_{j\in \hat J} \{f_\infty^j\le0\} .
	\]
\end{proof}

\begin{proposition}\label{prop5e8c4f94}
	Let $\Omega\subset\G$ be a nonempty closed set and $p\in\de\Omega$.
	Suppose that there exists a neighborhood $U$ of $p$ and a finite family of continuous functions $F_j:U\to\R$ with $j\in J$ such that $\Omega\cap U = \bigcap_{j\in J}\{F_j\le0\}$ and $F_j(p)=0$.
	Suppose also that each $F_j$ is 
	smooth
	and that $\{\pD F_j|_p\}_{j\in J}$ are linearly independent.
	
	Then, for every $(t_j)_j\in(\R\cup \{+\infty\})^J$ and every $\epsilon_n\to 0^+$, there are $p_n\to p$  such that 
	\begin{equation}\label{eq5f75bda1}
	\BU(\Omega,\{p_n\}_n,\{\epsilon_n\}_n)
	= \{x\in\G: \pD F_j|_p(x)\le t_j,\ j\in J\} .
	\end{equation}
\end{proposition}
\begin{proof}
	Define the function 
	\[
	R(\epsilon) 
	= \max_{j\in J} \sup\left\{ \left| \frac{F_j(p\delta_\eta w)-F_j(p)}{\eta} - \pD F_j|_p(w) \right| : w\in B(0,1),\ 0<\eta\le\epsilon \right\} .
	\]
	Since $J$ is finite and $F_j$ are all smooth, we have $R(\epsilon)=O(\epsilon)$,
	see \cite[Theorem~1.42]{MR657581}.

	Fix $(t_j)_j\in(\R\cup \{+\infty\})^J$ and $\epsilon_n\to 0^+$.
	Since $\{\pD F_j|_p\}_{j\in J}$ are linearly independent, 
	there are $w_1,w_2\in\G$ such that 
	\begin{align*}
	\pD F_j|_p(w_1) &=- t_j \text{ if }t_j<+\infty ,
	& \pD F_j|_p(w_1) &= 1 \text{ if }t_j=+\infty ; \\
	\pD F_j|_p(w_2) &= 0 \text{ if }t_j<+\infty ,
	& \pD F_j|_p(w_2) &= -1 \text{ if }t_j=+\infty .
	\end{align*}
	Now, we define $p_n := p (\delta_{\epsilon_n^{3/4}}w_2) (\delta_{\epsilon_n}w_1)$.
	As in the proof of Proposition~\ref{prop5e8c49ee}, we define $f_n^j(x) = \frac{F_j(p_n\delta_{\epsilon_n}x)}{\epsilon_n}$ and recall that $\bigcap_j\{f_n^j\leq 0\}$ gives a local description of $\delta_{1/\epsilon_n}(p_n\inv \Omega)$. In the limit, our choice of $p_n$ will allow us to express $\BU(\Omega, \{p_n\}_n, \{\epsilon_n\}_n) = \bigcap_j\{f^j_\infty \leq 0\}$ as in equation~(\ref{eq5f75bda1}). Indeed, by our choice of $p_n$, for any $x\in\G$, it follows that
	\begin{align*}
	f_n^j(x) = \frac{F_j(p_n\delta_{\epsilon_n}x)}{\epsilon_n}
	&= \frac{F_j( p (\delta_{\epsilon_n^{3/4}}w_2) (\delta_{\epsilon_n}(w_1x)) )-F_j(p)}{\epsilon_n} \\
	&= \frac{F_j( p (\delta_{\epsilon_n^{3/4}}w_2) (\delta_{\epsilon_n}(w_1x)) )-F_j(p (\delta_{\epsilon_n^{3/4}}w_2))}{\epsilon_n}
		+ \frac{F_j( p (\delta_{\epsilon_n^{3/4}}w_2) )-F_j(p)}{\epsilon_n} ,
	\end{align*}
	where
	\[
	\lim_{n\to\infty} \frac{F_j( p (\delta_{\epsilon_n^{3/4}}w_2) (\delta_{\epsilon_n}(w_1x)) )-F_j(p (\delta_{\epsilon_n^{3/4}}w_2))}{\epsilon_n}
	= \pD F_j|_p (w_1x)
	= \begin{cases}
	\pD F_j|_p(x) - t_j &\text{ if }t_j<+\infty \\
	1 &\text{ if }t_j=+\infty 
	\end{cases}
	\]
	and, if $t_j<\infty$, 
	\[
	\lim_{n\to\infty}  \left|\frac{F_j( p (\delta_{\epsilon_n^{3/4}}w_2) )-F_j(p)}{\epsilon_n}\right|
	= \lim_{n\to\infty}  \left| \frac{F_j( p (\delta_{\epsilon_n^{3/4}}w_2) )-F_j(p)}{\epsilon_n^{3/4}} \right| \frac1{\epsilon_n^{1/4}} \\
	\le \lim_{n\to\infty} \frac{  R(\epsilon_n^{3/4}) }{\epsilon_n^{1/4}} 
	= 0 ,
	\]
	while, if $t_j=\infty$, then
	\[
	\lim_{n\to\infty}  \frac{F_j( p (\delta_{\epsilon_n^{3/4}}w_2) )-F_j(p)}{\epsilon_n}
	= \lim_{n\to\infty}  \frac{F_j( p (\delta_{\epsilon_n^{3/4}}w_2) )-F_j(p)}{\epsilon_n^{3/4}}  \frac1{\epsilon_n^{1/4}} \\
	= \pD F_j|_p(w_2) \lim_{n\to\infty}  \frac1{\epsilon_n^{1/4}}
	= -\infty .
\]
	Finally, using the same strategy as in the second part of the proof of Proposition~\ref{prop5e8c49ee}, 
	we conclude that~\eqref{eq5f75bda1} holds.
\end{proof}

\subsection{Blow-ups of functions in homogeneous groups}
For a continuous function $f:\Omega\to\R$, we define
\[
\BU((\Omega,f),\{p_n\}_n,\{\epsilon_n\}_n) 
:= \Klim_{n\to \infty} \left( \delta_{1/\epsilon_n}(p_n\inv \Omega) , \frac{ f(p_n\delta_{\epsilon_n}\cdot)-f(p_n) }{ \epsilon_n } \right) .
\]

\begin{proposition}\label{prop5e8c979c}
	Let $\Omega\subset\G$ be a nonempty closed set, $\{p_n\}_{n\in\N}\subset\G$ and $\{\epsilon_n\}_{n\in\N}\subset(0,+\infty)$ sequences with $p_n\to p\in\Omega$ and $\epsilon_n\to 0$.
	Suppose that $\Omega_0:=\BU(\Omega,\{p_n\}_n,\{\epsilon_n\}_n)$ exists.
	Let $f:\G\to\R$ be a continuous function that is strictly Pansu differentiable at $p$.
	Then
	\[
	\BU((\Omega,f),\{p_n\}_n,\{\epsilon_n\}_n) 
	= ( \Omega_0 , \pD f(p)|_{\Omega_0} ) .
	\]
\end{proposition}
\begin{proof}
	Let $f_n(x) := \frac{ f(p_n\delta_{\epsilon_n}x)-f(p_n) }{ \epsilon_n }$.
	If $x_n\in \delta_{1/\epsilon_n}(p_n\inv\Omega)$ are such that $x_n\to x \in \Omega_0$, then $f_n(x_n)\to \pD f|_p[x]$, by the strict Pansu differentiability of $f$ at $p$.
\end{proof}

If $Q$ is a closed set, we say that a function $f:Q\to\R$ is smooth if there exists a smooth extension of $f$ in a neighborhood of $Q$.
In particular, the derivative of $f$ at points $p\in\de Q$ is well defined.

\begin{theorem}\label{piecetogether}
	Let $\Omega\subset\G$ be a closed set such that there is a family $\cal Q$ of regular closed sets  with disjoint interiors such that $\Omega=\bigcup_{Q\in\cal Q}  Q$.
	For each $Q\in\cal Q$, let $f_Q:\G\to\R$ smooth such that the function $f:\Omega\to\R$ defined by
	\[
	f(x) := \chi(x) \sum_{Q\in\cal Q} f_Q(x) \one_{ Q}(x) 
	\]
	is Lipschitz continuous, where $\chi(x) := \left( \sum_{Q\in\cal Q} \one_{ Q}(x) \right)^{-1}$.
	
	Let $\{p_n\}_{n\in\N}\subset\G$ and $\{\epsilon_n\}_{n\in\N}\subset(0,+\infty)$ sequences with $p_n\to p\in\Omega^\circ$ and $\epsilon_n\to 0$.
	Assume that $R_Q:=\BU(Q,\{p_n\}_n,\{\epsilon_n\}_n)$ exists for every $Q\in\cal Q$.

	Then 
	\[
	\G = \bigcup_{Q\in\cal Q} R_Q
	\]
	and $\BU((\Omega,f),\{p_n\}_n,\{\epsilon_n\}_n) = (\G,g)$ exists, where
	\begin{equation}\label{eq03271750}
	g(x) = \tilde\chi(x) \sum_{Q\in\cal Q} \left( \pD f_Q|_p(x) + c_Q \right) \one_{R_Q}(x),
	\end{equation}
	with $\tilde\chi(x) := \left( \sum_{Q\in\cal Q} \one_{R_Q}(x) \right)^{-1}$
	and $c_Q\in\R$.
\end{theorem}
Notice that the constants $c_Q$ can be determined by the continuity of $g$ and $g(e)=0$.
If there are more than one choice of such constants, the resulting function is still the same:
indeed, if $g$ and $g'$ are two functions as in~\eqref{eq03271750} with different constants, then $g-g'$ is a piecewise constant and continuous function that is $0$ in $e$, and thus $g=g'$.
Moreover, we remark that we don't need the limit sets $R_Q$ to have disjoint interiors.
\begin{proof}
	The fact that $\G = \bigcup_{Q\in\cal Q} R_Q$ follows from $p\in\Omega^\circ$ and~\eqref{eq03271215}.
	Next, set $g_n:\delta_{1/\epsilon_n}(p_n\inv\Omega)\to\R$, $g_n(x) := \frac{f(p_n\delta_{\epsilon_n}x)-f(p_n)}{\epsilon_n}$.
	The family of functions $\{g_n\}_{n\in\N}$ is uniformly Lipschitz and $g_n(e)=0$ for all $n$.
	Thus, the set 
	\[
	\cal N :=\{N\subset\N\text{ infinite}:\{g_n\}_{n\in N}\text{ converge}\}
	\]
	is nonempty and for every $N\subset\N$ infinite there is $N'\in\cal N$ with $N'\subset N$.
	For every $N\in\cal N$, define $g^N:=\lim_{N\ni n\to\infty} g_n$.
	We aim to prove that $g^N=g$ for all $N\in\cal N$.
	
	Let $x\in R_Q$ for some $Q\in\cal Q$.
	Then there exist $y_n\in Q$ such that $x_n:=\delta_{1/\epsilon_n}(p_n\inv y_n)\to x$.
	Therefore, $g_n(x_n)\to g(x)$, where
	\begin{align*}
	g_n(x_n) 
	&= \frac{f(y_n)-f(p_n)}{\epsilon_n} \\
	&= \frac{f_Q(y_n)-f_Q(p_n)}{\epsilon_n} + \frac{f_Q(p_n)-f(p_n)}{\epsilon_n} .
	\end{align*}
	Since $f_Q$ is smooth at $p$, we have $\lim_n \frac{f_Q(y_n)-f_Q(p_n)}{\epsilon_n} = \pD f_Q|_p[x]$.
	Therefore, if $N\in\cal N$, then the limit $c_Q^N:=\lim_n \frac{f_Q(p_n)-f(p_n)}{\epsilon_n}$ exists and it is equal to $g^N(x)-\pD f_Q|_p[x]$.
	Moreover,
	\[
	g^N(x) = \tilde\chi(x) \sum_{Q\in\cal Q} \left( \pD f_Q|_p(x) + c^N_Q \right) \one_{R_Q}(x) .
	\]
	Finally, $g^N$ is continuous and $g^N(e)=0$.

	So, for any pair $N,N'\in\cal N$, the difference $g^N-g^{N'}$ is a piecewise constant and continuous function that takes the value 0 at $e$.
	Hence, $g^N-g^{N'}\equiv0$, for all $N\in\cal N$.

\end{proof}

This theorem will allow us to finish our description of the horofunction boundary. At non-smooth points, horofunctions do not necessarily correspond to Pansu derivatives, but instead are piecewise defined by Pansu derivatives in each blow-up region.
Theorem~\ref{piecetogether} can also be used to recover results about the horofunction boundaries of normed spaces as in \cite{ji-schilling, walsh-norm}.



\section{Vertical sequences in the Heisenberg group $\HH$}\label{sec:disk}

In this section, we focus on the Heisenberg group, see Section~\ref{sec5f742960}.
We extend to sub-Finsler distances a result that Klein--Nicas  proved for the sub-Riemannian and the Korany distances in \cite{KN-cc, KN-koranyi}.
In particular, we 
show that, for any sub-Finsler metric in the Heisenberg group $\HH$, vertical sequences  induce a topological disk in the horoboundary.
The result is not true for all homogeneous distances in $\HH$, see Remark~\ref{rem5f742be1}.

\begin{theorem}\label{vertSeq}
	Let $d$ be the sub-Finsler distance on $\HH$ generated by norm $\|\cdot\|$ on the horizontal plane.
	Let $\{w_n\}_{n\in\N}\subset\R^2$ be a bounded sequence and $\{s_n\}_{n\in\N}\subset\R$ with $|s_n|\to\infty$, and set $p_n=(w_n,s_n)\in\HH$.
	Then for all $(v,t) \in \HH$
	\[
	\lim_{n\to \infty} d(p_n, (v,t)) - d(p_n, e) - (\| w_n\| - \| w_n-v\|) = 0.
	\]
	There is, therefore, a topological disk 
	$\{p\mapsto \|w\|-\|\pi(p)-w\|:w\in\R^2\}\subset C(\HH)$
	in the horofunction boundary.
\end{theorem}

We need a couple of lemmas before the proof of the theorem.
We start with a technical lemma concerning convex geometry. Fix $b > 0$ and an open bounded convex set $Q\subset\R^2$. Dilate $Q$  by $\lambda\ge 1$, and take two points $p,q\in\de(\lambda Q)$ so that $|p-q|\le b$. The line passing through $p$ and $q$ cuts $\lambda Q$ into two parts with areas $s$ and $t$ respectively, say $s\le t$.
Then the lemma says there is $M$ such that that $s<M$ for all $\lambda\ge 1$.  

\begin{lemma}\label{lem03251119}
	Let $Q\subset\R^2$ be an open bounded convex set.
	Fix $b>0$ and define for $\lambda\ge1$
	\begin{align*}
	\bar x(\lambda) &:= \inf\{x\in\R:L^1\{y:(x,y)\in\lambda Q\}\ge b\} ,\\
	Q^-_\lambda &:= \{(x,y)\in\lambda Q: x\le \bar x(\lambda) \} .
	\end{align*}
	Then
	\[
	\sup\{ L^2(Q^-_\lambda) : \lambda\ge1 \}<\infty ,
	\]
	where $L^1$ and $L^2$ denote the 1- and 2-dimensional Lebesgue measures, respectively.
\end{lemma}
\begin{proof}
	Since $L^2(Q^-_\lambda) \le \lambda^2 L^2(Q)$, 
	we must show that $L^2(Q^-_\lambda)$ remains bounded for $\lambda$ large.
	Taking $\lambda$ large enough, we can assume $\bar x(\lambda)<+\infty$.
	Define 
	\[
	V_\lambda(x) = L^1\{y:(x,y)\in\lambda Q\},
	\]
	and note that $V_\lambda(x) = \lambda V_1(x/\lambda)$.
	Up to translating $Q$, we can assume $V_1(x)=0$ for all $x\le 0$ and $V_1(x)>0$ for small $x>0$.
	Moreover, since $Q$ is convex, $V_1$ is a concave function.
	If $\lim_{x\to 0^+}V_1(x) >0$, then $\bar x(\lambda) =0$ and thus $L^2(Q^-_\lambda) =0 $ for $\lambda$ large.
	
	Now assuming that $\lim_{x\to 0^+}V_1(x)=0$, we have that $\bar x(\lambda)>0$ for all $\lambda$.
	By concavity, there are $\epsilon,m>0$ such that $V_1(x)\ge m x$ for all $x\in(0,\epsilon]$.
	By the definition of $\bar x$, if $x<\bar x(\lambda)$ then $V_\lambda(x) \le b$.
	For $\lambda$ large, $V_\lambda(\lambda\epsilon) = \lambda V_1(\epsilon) \ge \lambda m \epsilon \ge b$,  
	and so $\bar x(\lambda) <\lambda\epsilon$. 
	It follows that
	\[
	b > V_\lambda(\bar x(\lambda)/2) \ge m \bar x(\lambda)/2 ,
	\]
	that is, $\bar x(\lambda) \le 2b/m$.
	We conclude that 
	\[
	L^2(C^-_\lambda) = \int_0^{\bar x(\lambda)} V_\lambda(x) \dd x \le \frac{2b^2}{m} ,
	\]
	for $\lambda$ large enough.
\end{proof}

\begin{lemma}\label{lem03251220}
For any sub-Finsler metric $d$ on $\HH$ and any $v \in \R^2$,
\begin{equation}\label{eq03251818}
\lim_{t\to +\infty} \left[ d_e((v,t)) - d_e((0,t)) + d_e((v,0)) \right] = 0.
\end{equation}
Moreover, the convergence is uniform in $v$ on compact sets.
\end{lemma}
\begin{proof}
	By the triangle inequality, we have 
	\begin{equation}\label{eq03251121}
	d_e((v,t)) - d_e((0,t)) + d_e((v,0)) \ge 0
	\end{equation}
	for all $t$.
 Let $Q^*\subset\R^2$ be the convex set dual to the unit ball $Q$ of the norm $\norm\cdot$ on $\R^2$. Let $I$ be the rotation by $\frac\pi2$ of $Q^*$.

	Define $a = a(t) =d_e((v,t))$, $b=d_e((v,0))$ and $h=h(t)=d_e((0,t))$.
	For $t$ large enough, the projection $\gamma_1:[0,1]\to\R^2$ of a geodesic from $(0,0)$ to $(v,t)$ is a portion of the boundary of $\lambda I$, for some $\lambda$, with $\gamma_1(0)=(0,0)$ and $\gamma_1(1)=v$.
	Notice that $a$ is the length of $\gamma_1$, that $b=\norm v$ is the length of a chord of $\de(\lambda I)$ and that $t$ is the area one of the two parts of $\lambda I$ separated by the line passing through $0$ and $v$.
	Let $s$ be the area of the other part and $c$ the length of $\de(\lambda I)\setminus \gamma_1$.	
	If $A$ is the area of $I$ and $\ell$ is the length of $\de I$, we have
	$a+c = \lambda\ell$ and 
	$t+s = \lambda^2 A .$ See Figure~\ref{fig:vertSeq}.
	
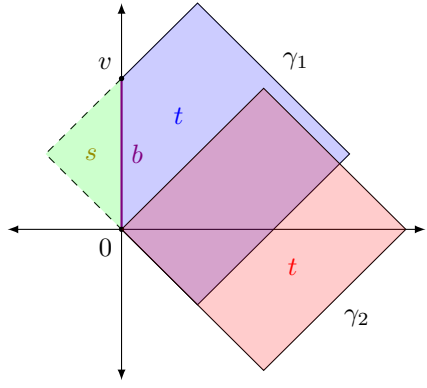
\begin{figure}[H]
\centering
\begin{tikzpicture}[>=latex,scale=.5]
\draw[->] (0,0) -- (0,6);
\draw[->] (0,0) -- (0,-4);
\draw[->] (0,0) -- (8,0);
\draw[->] (0,0) -- (-3,0);
\draw (0,0) -- (2, -2) -- (6, 2) to node[above right] {$\gamma_1$} (2, 6) to  (0, 4);
\fill[blue, opacity = .2] (0,0) -- (2, -2) -- (6, 2) -- (2, 6) -- (0, 4);
\draw[dashed] (0,4) -- (-2, 2) -- (0,0);
\fill[green, opacity = .2] (0,4) -- (-2, 2) -- (0,0);
\draw[violet, thick] (0,0) -- node[right] {$b$} (0,4);
\fill[black] (0,4) node[above left] {$v$} circle (.07cm);
\fill[black] (0,0) node[below left] {$0$} circle (.07cm);
\node[olive] at (-.8,2) {$s$};
\node[blue] at (1.5,3) {$t$};
\node[red] at (4.5,-1) {$t$};
\draw (0,0)--(3.74,-3.74) to node[below right] {$\gamma_2$} (7.48,0) -- (3.74, 3.74) -- cycle;
\fill[red, opacity = .2] (0,0)--(3.74,-3.74) -- (7.48,0) -- (3.74, 3.74) -- cycle;

\end{tikzpicture}
\caption{Convex geometry and vertical sequences}
\label{fig:vertSeq}
\end{figure}

	The projection $\gamma_2$ of a geodesic from $(0,0)$ to $(0,t)$ is the boundary of $\mu C$, for some $\mu$ so that $t=L^2(\mu C) = \mu^2 A$.
	Then $h = d_e((0,t))$ is the length of the boundary of $\mu C$.
	Therefore,
	\begin{equation}\label{eq-geq}
	h = \frac{\mu}{\lambda} (a+c) = \sqrt{\frac{t}{t+s}} (a+c) \ge \sqrt{\frac{t}{t+s}} (a+b) .
	\end{equation}
	
	By Lemma~\ref{lem03251119}, there is $M>0$ such that $s<M$ for all $t$ sufficiently large. Thus, by combining~\eqref{eq03251121} and~\eqref{eq-geq}, we see that $h(t)$ converges to $a(t) + b$, completing the first part of the proof. 
	
	For the uniform convergence, if we define 
	$f_t(v) = d_e((v,t)) - d_e((0,t)) + d_e((v,0))$,
	then by the reverse triangle inequality, $f_t:(\R^2\times\{0\},d)\to \R$ is Lipschitz, i.e.,
	\[
	|f_t(v)-f_t(w)|
	\le  d((v,t),(w,t)) + d((v,0),(w,0)) 
	= 2 d((v,0),(w,0)) ,
	\]
	and $f_t(e)=0$.
	Therefore, the pointwise convergence is uniform on compact sets.
\end{proof}

\begin{proof}[Proof of Theorem~\ref{vertSeq}]
	It suffices to consider the case when $s_n \to +\infty$. Notice that
	\begin{multline*}
	d(p_n, (v, t)) - d(p_n, e) - (\| w_n\| - \| w_n-v\|) \\
	= d_e((w_n-v,s_n-t-\omega(v,w_n)/2)) - d_e((0,s_n-t-\omega(v,w_n)/2)) + d_e((w_n-v,0)) \\
		- d_e((w_n,s_n)) + d_e((0,s_n)) - d_e((w_n,0)) \\
		+ d_e((0,s_n-t-\omega(v,w_n)/2)) - d_e((0,s_n)) ,
	\end{multline*}
	where, $\omega$ is the standard symplectic form on $\R^2$. Using Lemma~\ref{lem03251220} and the boundedness of $w_n$,
	\[
	\lim_{n\to\infty} d_e((w_n-v,s_n-t-\omega(v,w_n)/2)) - d_e((0,s_n-t-\omega(v,w_n)/2)) + d_e((w_n-v,0))
	= 0 ,
	\]
	and
	\[
	\lim_{n\to\infty} - d_e((w_n,s_n)) + d_e((0,s_n)) - d_e((w_n,0))  = 0.
	\]
	Finally, 
	\begin{multline*}
	\lim_{n\to\infty} d_e((0,s_n-t-\omega(v,w_n)/2)) - d_e((0,s_n))
	= d_e((0,1)) \lim_{n\to\infty} \left( \sqrt{s_n-t-\omega(v,w_n)/2} - \sqrt{s_n} \right) 
	=0 .
	\end{multline*}
	
	For the last statement, fix $w \in \R^2$ and set $p_n = (w, n) \in \HH$. Then $p_n \to f(v, t) = \norm w - \norm{w - v}$ in the horofunction boundary.
\end{proof}

\begin{remark}\label{rem5f742be1}
	For general homogeneous distances, Lemma~\ref{lem03251220} is not true.
	As an example, consider the function $f:\R\to\R$ defined by
	\[
	f(x) := (-1)^{k(x)} \left( 2|x| - 3^{1+k(x)} \right) ,
	\text{ where }
	k(x) := \left\lfloor \frac{\log(x)}{\log(3)} \right\rfloor ,
	\]
	which is piecewise linear with derivative $\pm2$ and satisfies $-|x|\le f(x)\le |x|$ for all $x$.
	\begin{center}
	\includegraphics[width=0.4\textwidth]{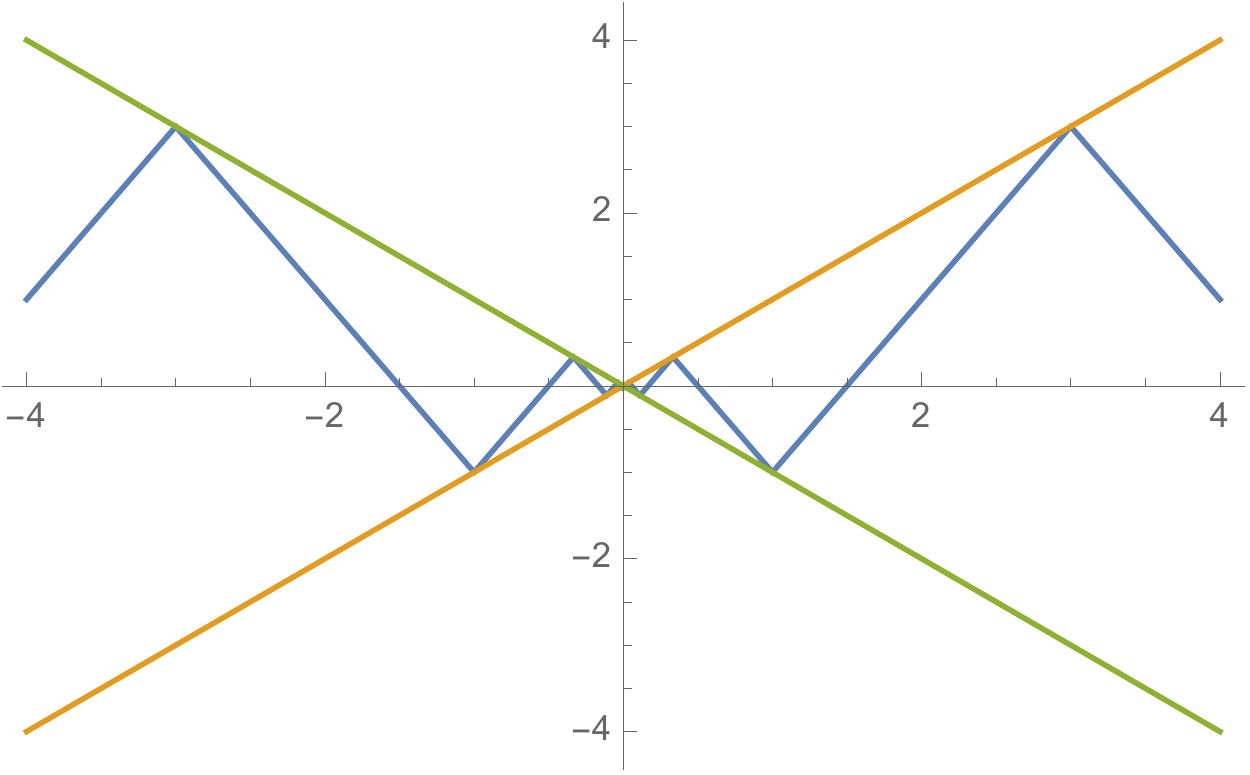}
	\end{center}
	Consider the function $\phi(v):=f(|v|)$ on the disk in $\R^2$.
	Since $\phi$ is Lipschitz, then, by \cite[Proposition 6.3]{MR3739202}, there is $M$ such that $\phi+M$ is the profile of the unit ball of a homogeneous distance $d$ in $\HH$.
	If $v\in\R^2\setminus\{0\}$, then there is a sequence $\{t_n\}_{n\in\N}$ with $t_n\to+\infty$ such that $f(\frac{|v|\sqrt{M}}{\sqrt{t_n}})=0$, i.e., $d_e((\frac{v\sqrt{M}}{\sqrt{t_n}},M))=1=d_e((0,M))$.
	Therefore,
	\[
	d_e((v,t_n)) - d_e((0,t_n)) + d_e((v,0))
	= \sqrt{\frac{t_n}{M}} \left( d_e((\tfrac{v\sqrt{M}}{\sqrt{t_n}}),M)) - d_e((0,M)) \right) + d_e((v,0)) 
	= d_e((v,0)),
	\]
	for all $n\in\N$.
	We conclude that~\eqref{eq03251818} cannot hold for such $d$.
\end{remark}


\section{Horofunctions in polygonal sub-Finsler metrics on $\HH$}\label{sec:subfinsler}
Before stating the main result of the section and the paper,
we introduce the necessary notation for the description of sub-Finsler distances in $\HH$.

\subsection{Geometry of polygonal sub-Finsler metrics}\label{coordinates}
On $\R^2$, we denote by $\langle \cdot,\cdot \rangle$ the standard scalar product,
and by $J$ the ``multiplication by $i$'', i.e., the anticlockwise rotation by $\frac\pi2$.
Notice that $\omega(\cdot,\cdot)=\langle J \cdot,\cdot \rangle$ is the standard symplectic form.
We will use the symplectic duality between $\R^2$ and $(\R^2)^*$ induced by $\omega$ 
via 
\[
\R^2\ni\alpha^\omega\leftrightarrow \alpha=\omega(\alpha^\omega,\cdot) \in(\R^2)^* .
\]

Let $Q$ be a centrally-symmetric polygon in $\R^2$ with $2N$ vertices, and let $\|\cdot\|$ be the norm on $\R^2$ with unit metric disk $Q$. 
Enumerate the vertices $\{\vv_k\}_{k}$ of $Q$ with $k\in\Z$ modulo $2N$, in an anticlockwise order.
Notice that $-\vv_{k}=\vv_{k+N}$.
Define the $k$-th edge to be the vector $e_k:=\vv_{k+1}-\vv_k$.
For each $k$, let $\alpha_k\in(\R^2)^*$ be the linear function such that $\alpha_k(\vv_k+t e_k)=1$ for all $t \in \R$, that is,
\[
\alpha_k = \frac{\omega(e_k,\cdot)}{\omega(e_k,\vv_k)} ,
\]
where $\omega(e_k,\vv_k) = \omega(\vv_{k+1},\vv_k) \neq 0$.

Let $Q^*\subset (\R^2)^*$ be the unit disk of the norm dual to $\|\cdot\|$, that is, the polar dual of $Q$. 
Note that $Q^*$ is the polygon with vertices $\{\alpha_k\}_k$.

A result of Busemann \cite{busemann-isoperimetric} tells us that the isoperimetric set $I$, or {\em isoperimetrix}, in $(\R^2,\|\cdot\|)$ is 
the image of $Q^*$ in $\R^2$ via the symplectic duality.%
\footnote{In other words, $I$ is $J^*Q^*$, seen as a subset of $\R^2$ via the equivalence $\R^2\simeq (\R^2)^*$ given by the scalar product.}
It follows that $I$ is the polygon with vertices 
\[
\alpha_k^\omega =  \frac{e_k}{\omega(e_k,\vv_k)} ,
\]
where $\omega(e_k,\vv_k)=\omega(\vv_{k+1},\vv_k)<0$,
and edges 
\[
\sigma_k := \alpha_k^\omega - \alpha_{k-1}^\omega .
\]
Figure~\ref{coords} describes the situation for a hexagonal $Q$. 
We note that $\sigma_k$ is a scalar multiple of $\vv_k$,
\begin{equation}\label{eq5f7879a4}
\sigma_k = \|\sigma_k\|\vv_k \text{ where } \|\sigma_k\| = \alpha_k(\sigma_k)>0 .
\end{equation}
Indeed, $\omega(\sigma_k,\vv_k) = \alpha_k(\vv_k) - \alpha_{k-1}(\vv_k) = 1-1=0$,
and thus $\sigma_k = \alpha_k(\sigma_k) \vv_k$.
Since $\alpha_k(\sigma_k) = \omega(\alpha_{k-1}^\omega,\alpha_k^\omega)>0$, we have $\|\sigma_k\| = \alpha_k(\sigma_k)$.

\begin{figure}[h]
\begin{tikzpicture}[scale=0.4, every node/.style={scale=0.8}]
\draw[->] (3,3) to (1.5,3);
\draw[->] (3,0) to (3,1.5);
\draw[->] (0,-3) to (1.5,-1.5);
\draw[->] (-3,-3) to (-1.5,-3);
\draw[->] (-3,0) to (-3,-1.5);
\draw[->] (0,3) to (-1.5,1.5);

\draw (3,0) -- (3,3) -- (0,3) -- (-3,0) -- (-3, -3) -- (0,-3) -- cycle;
\fill (3,0) node[right] {$\vv_1$} circle (4 pt);
\fill (3,3) node[above right] {$\vv_2$} circle (4 pt);
\fill (0,3) node[above] {$\vv_3$} circle (4 pt);
\fill (-3,0) node[left] {$\vv_4$} circle (4 pt);
\fill (-3,-3) node[below left] {$\vv_5$} circle (4 pt);
\fill (0, -3) node[below] {$\vv_6$} circle (4 pt);
\draw (3,1.5) node[right] {$ e_1$};
\draw (1.5,3) node[above] {$ e_2$};
\draw (-1.5,1.5) node[above left] {$ e_3$};
\draw (-3,-1.5) node[left] {$ e_4$};
\draw (-1.5,-3) node[below] {$ e_5$};
\draw (1.5, -1.5) node[below right] {$ e_6$};
\draw (0,0) node[scale =1.2] {$Q$};

\begin{scope}[xshift=10cm]
\draw[->] (-3,3) to (-3,1.5);
\draw[->] (0,3) to (-1.5,3);
\draw[->] (3,0) to (1.5,1.5);
\draw[->] (3,-3) to (3,-1.5);
\draw[->] (0,-3) to (1.5,-3);
\draw[->] (-3,0) to (-1.5,-1.5);

\draw (-3,3) -- (-3,0) -- (0,-3) -- (3,-3) -- (3, 0) -- (0,3) -- cycle;
\fill (3,0) node[right] {$\alpha_1$} circle (4 pt);
\fill (0,3) node[above ] {$\alpha_2$} circle (4 pt);
\fill (-3,3) node[above left] {$\alpha_3$} circle (4 pt); 
\fill (-3,0) node[left] {$\alpha_4$} circle (4 pt);
\fill (0,-3) node[below ] {$\alpha_5$} circle (4 pt);
\fill (3, -3) node[below right] {$\alpha_6$} circle (4 pt);
\draw (0,0) node[scale =1.2] {$Q^*$};
\end{scope}

\begin{scope}[xshift=20cm]
\draw[->] (3,3) to (1.5,3);
\draw[->] (3,0) to (3,1.5);
\draw[->] (0,-3) to (1.5,-1.5);
\draw[->] (-3,-3) to (-1.5,-3);
\draw[->] (-3,0) to (-3,-1.5);
\draw[->] (0,3) to (-1.5,1.5);

\draw (3,0) -- (3,3) -- (0,3) -- (-3,0) -- (-3, -3) -- (0,-3) -- cycle;
\draw (3,1.5) node[right] {$\sigma_3$};
\draw (1.5,3) node[above] {$\sigma_4$};
\draw (-1.5,1.5) node[above left] {$\sigma_5$};
\draw (-3,-1.5) node[left] {$\sigma_6$};
\draw (-1.5,-3) node[below] {$\sigma_1$};
\draw (1.5, -1.5) node[below right] {$\sigma_2$};
\draw (0,0) node[scale =1.2] {$I$};

\fill (3,0) node[right] {$\alpha^\omega_2$} circle (4 pt);
\fill (0,3) node[above left] {$\alpha^\omega_4$} circle (4 pt);
\fill (3,3) node[above right] {$\alpha^\omega_3$} circle (4 pt); 
\fill (-3,0) node[left] {$\alpha^\omega_5$} circle (4 pt);
\fill (0,-3) node[below right] {$\alpha^\omega_1$} circle (4 pt);
\fill (-3, -3) node[below left] {$\alpha^\omega_6$} circle (4 pt);

\end{scope}

\end{tikzpicture}
\caption{Example of a norm ball, dual ball, and isoperimetrix.}
\label{coords}
\end{figure}
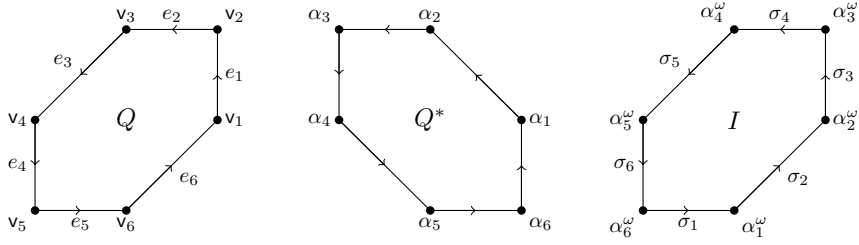

For the case of polygonal sub-Finsler metrics on $\HH$, Duchin--Mooney \cite{duchin-mooney} classify geodesics and describe the shape of the unit sphere. Here, we introduce some of their notation and summarize some key results.

Duchin--Mooney break geodesics into two categories: {\em beelines} and {\em trace paths}. 
Beeline geodesics are lifts of $(\R^2, \norm{\cdot})$-geodesics to admissible paths in $\HH$. 
Trace path geodesics, on the other hand, are lifts of paths in the plane which trace some portion of the boundary of rescaled versions of $I$.

As in Duchin--Mooney, we partition $Q$ into quadrilateral regions which are reached by trace paths which trace the same edges of $I$. 
That is, for $i < j < 2N+i$, define $Q_{ij}\subset \R^2$ to be the set of all endpoints of positively-oriented trace paths in the plane whose parametrizations start by tracing a portion of $\sigma_i$, trace all of $\sigma_{i+1},\ldots, \sigma_{j-1}$, and end by tracing a portion of $\sigma_j$, 
rescaled so that the total length is~1:
\begin{align*}
Q_{ij} 
	& := \left\{\frac1\mu(r\sigma_i + {\sigma_{i+1}} + \ldots + {\sigma_{j-1}} + s{\sigma_j}): r, s \in [0,1]\right\}, 
\end{align*}
where $\mu = \mu_{ij}(r,s) = r\norm{\sigma_i} + \norm{\sigma_{i+1}} + \ldots + \norm{\sigma_{j-1}} + s\norm{\sigma_j}$ normalizes the length of the path.
The case $i=j$ is similarly defined
\[
	Q_{ii}
	:= \left\{\frac1\mu(r\sigma_i + {\sigma_{i+1}} + \ldots + {\sigma_{i-1}} + s{\sigma_i}): r, s \in [0,1],\ r+s\le 1\right\}\bigcup \left\{\frac1\mu r\sigma_i:r\in[0,1]\right\},
\]
applying the convention that indices are modulo $2N$.
Note that $Q_{ii}=\{-s\vv_i:s\in[0,A(I_1)]\}\cup\{\vv_i\}$, where $A(I_1)$ is the area of the unit-perimeter scaled copy of $I$. Then $Q_{i,i+1}=\{s\vv_i+(1-s)\vv_{i+1}:s\in[0,1]\}$ is the $i$-th edge of $Q$ (see \cite[Theorem 7]{duchin-mooney}).
For $j\notin\{i,i+1\}$, 

the regions $Q_{ij}$ are non-degenerate quadrilaterals with disjoint interiors, and the set of all $Q_{ij}$ covers $Q$.

The unit sphere of a polygonal sub-Finsler distance
is the set of all endpoints of unit-length geodesics
and it can be described as a the region between the graphs of two functions $Q\to\R$, see Figure~\ref{fig:sphere}.
Endpoints of beeline geodesics make up vertical wall panels on the edges of $Q$: 
we denote by $\Panel_{i,i+1}$ the vertical wall panel which projects to edge $Q_{i,i+1}$, through vertices $\vv_i$ and $\vv_{i+1}$.

Endpoints of all unit-length, positively-oriented trace path geodesics make up the ceiling of the sphere:
we denote by $\Panel_{ij}^+$ the ceiling panel above $Q_{ij}$, that is the set of endpoints of lifts of all unit-length, positively-oriented trace paths whose endpoints lie in $Q_{ij}$.

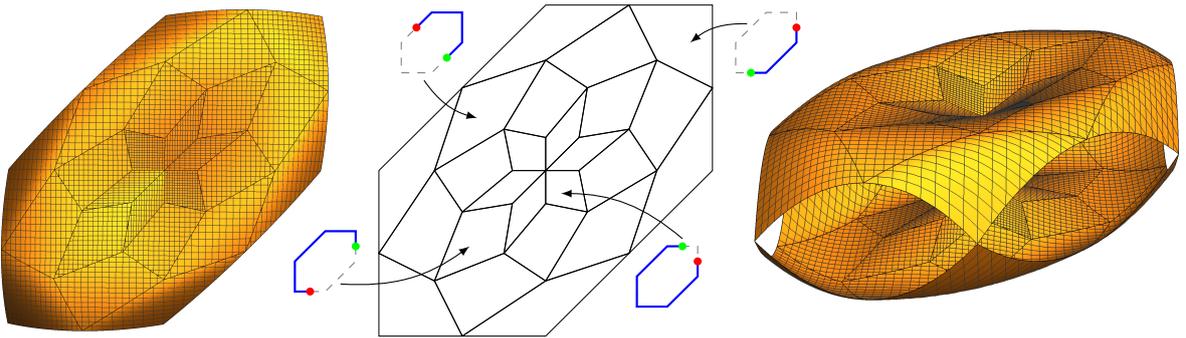
\begin{figure}[ht]
\centering
\begin{tikzpicture}[>=latex]

\node at (0,0) {\includegraphics[width=1.8in]{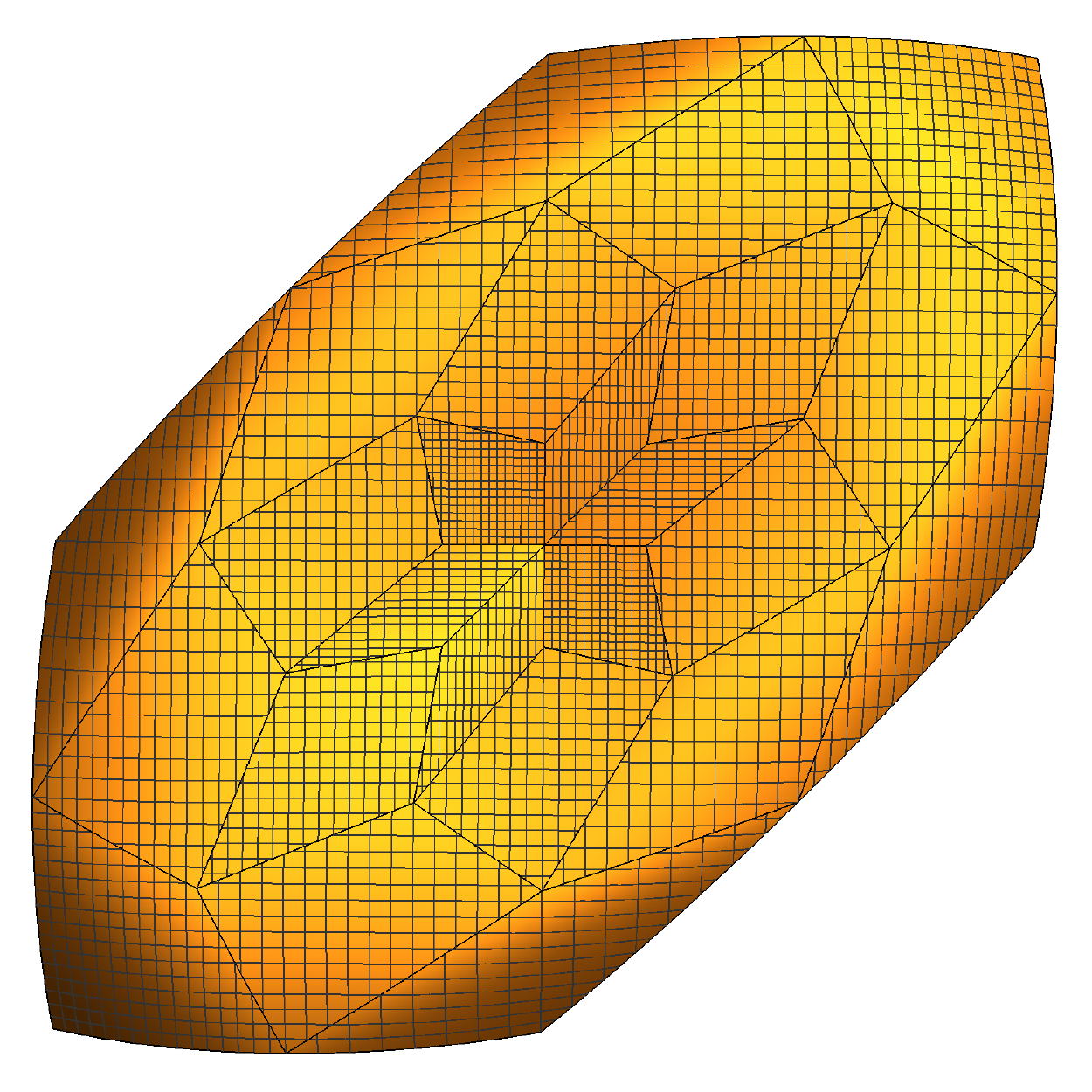}};

\node at (5,0) {\includegraphics[width=1.8in]{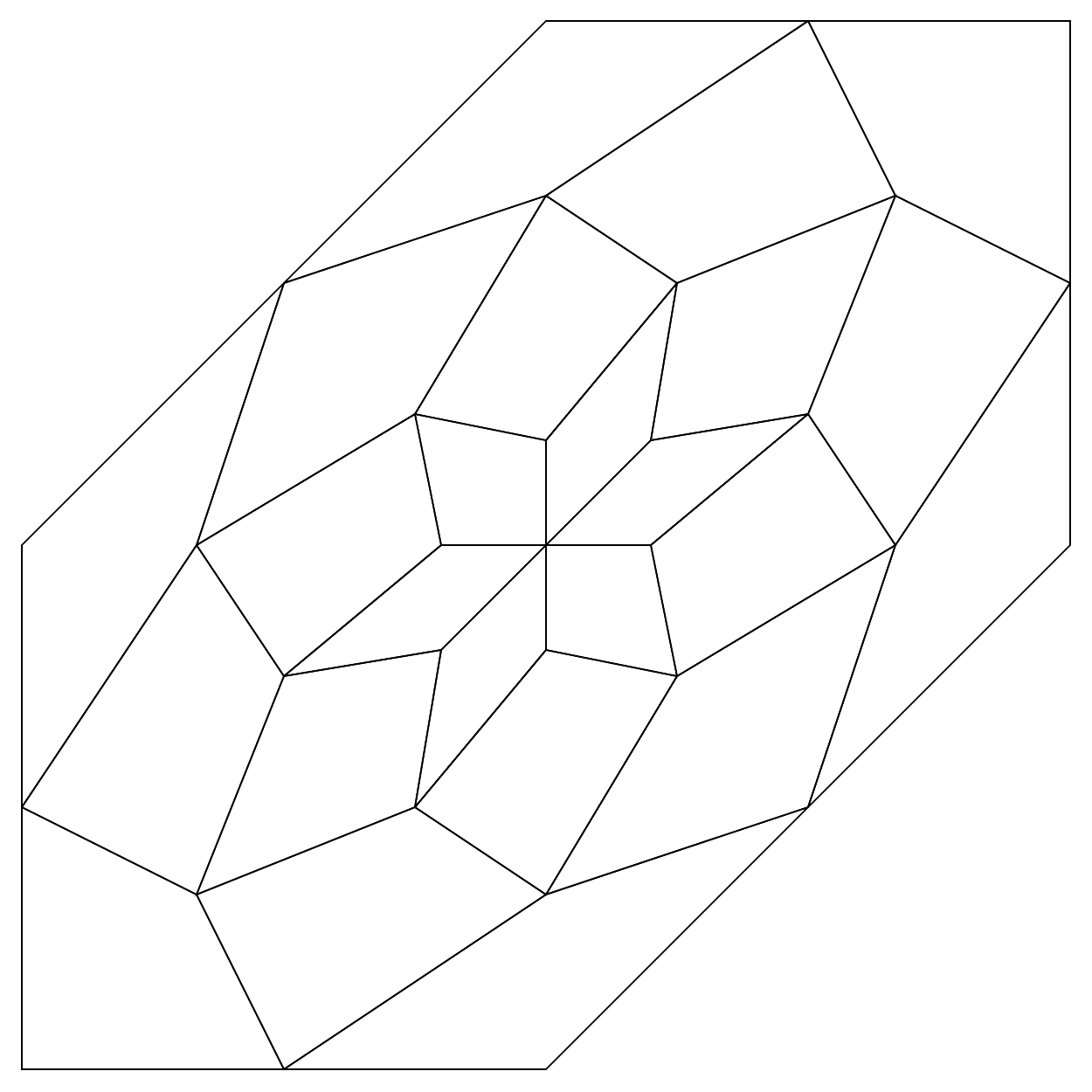}};

\node at (10.7,0) {\includegraphics[width=2.5in]{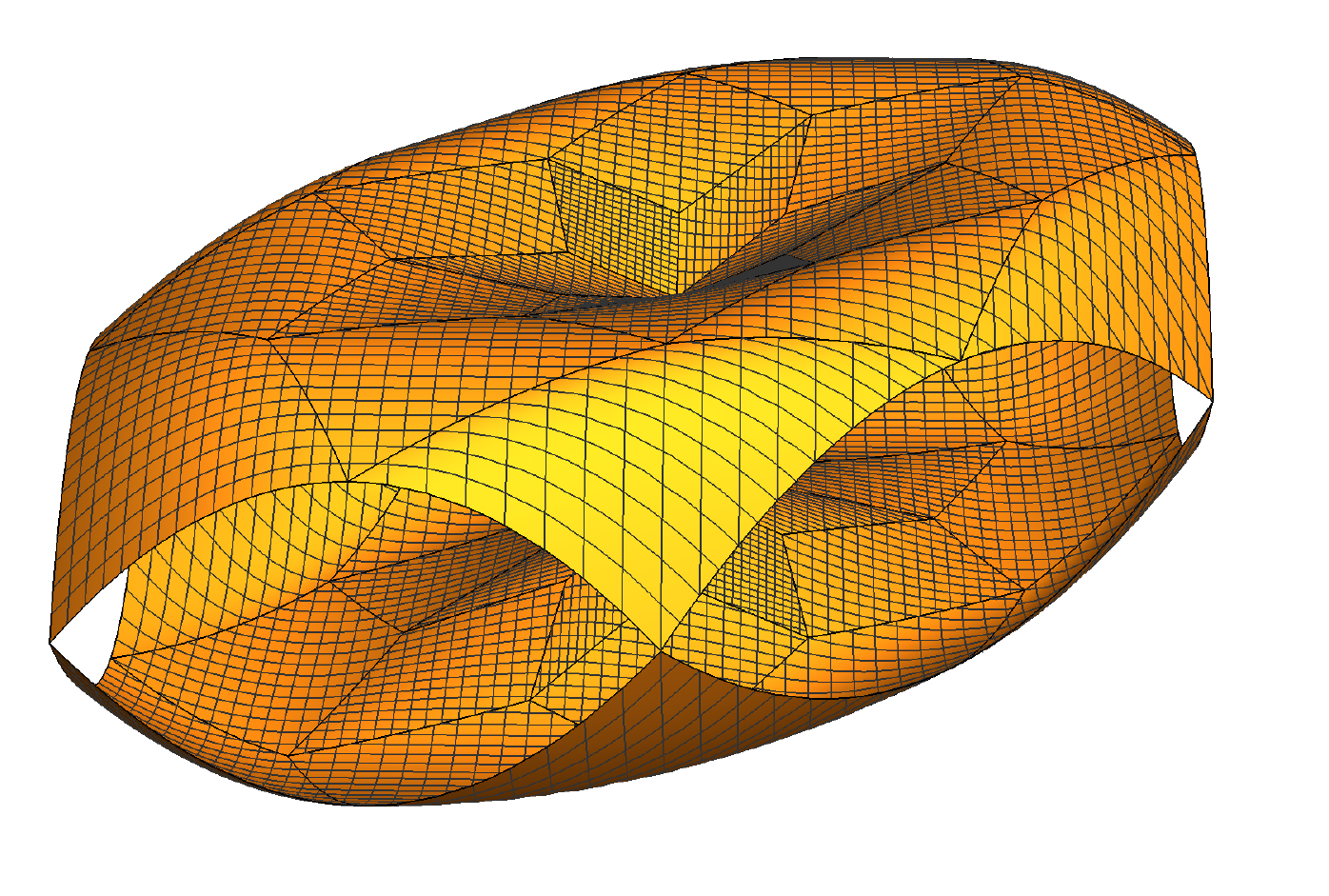}};

\begin{scope}[xshift = 7.5cm, yshift = 1.3cm]
\draw[dashed, gray] (0,0) -- (.4, 0) -- (.8, .4) -- (.8, .8) -- (.4, .8) -- (0,.4) -- cycle;
\draw[blue, thick] (.2, 0) -- (.4,0) -- (.8,.4) -- (.8,.6);
\fill[green] (.2,0) circle (1.5pt);
\fill[red] (.8,.6) circle (1.5pt);
\draw[->] (.15, .65) to[bend right = 20] (-.6, .4);
\end{scope}

\begin{scope}[xshift = 3.1cm, yshift = 1.3cm]
\draw[dashed, gray] (0,0) -- (.4, 0) -- (.8, .4) -- (.8, .8) -- (.4, .8) -- (0,.4) -- cycle;
\draw[blue, thick] (.6, .2) -- (.8,.4) -- (.8,.8) -- (.4, .8) -- (.2, .6);
\fill[green] (.6,.2) circle (1.5pt);
\fill[red] (.2,.6) circle (1.5pt);
\draw[->] (.3, -.1) to[bend right = 20] (1, -.6);
\end{scope}

\begin{scope}[xshift = 1.7cm, yshift = -1.6cm]
\draw[dashed, gray] (0,0) -- (.4, 0) -- (.8, .4) -- (.8, .8) -- (.4, .8) -- (0,.4) -- cycle;
\draw[blue, thick] (.8,.6) -- (.8,.8) -- (.4, .8) -- (0,.4) -- (0,0) -- (.2, .0);
\fill[green] (.8,.6) circle (1.5pt);
\fill[red] (.2,0) circle (1.5pt);
\draw[->] (.6, .1) to[bend right = 20] (2.3, .6);
\end{scope}

\begin{scope}[xshift = 6.2cm, yshift = -1.8cm]
\draw[dashed, gray] (0,0) -- (.4, 0) -- (.8, .4) -- (.8, .8) -- (.4, .8) -- (0,.4) -- cycle;
\draw[blue, thick] (.6,.8) -- (.4,.8) -- (0, .4) -- (0,0) -- (.4,0) -- (.8,.4) -- (.8,.6);
\fill[green] (.6,.8) circle (1.5pt);
\fill[red] (.8,.6) circle (1.5pt);
\draw[->] (.6, .9) to[bend right = 20] (-1, 1.5);
\end{scope}

\end{tikzpicture}
\caption{On the left and right, two views of a unit sphere with vertical walls missing. In the center, the norm ball $Q$ broken up into quadrilaterals $Q_{ij}$ of points reached by trace paths of the same shape. Figures adapted from \cite{duchin-mooney}.}
\label{fig:sphere}
\end{figure}

It will be useful to have an explicit description of these panels.
Fix a non-degenerate quadrilateral region $Q_{ij}$ and define $u=u_{ij}:[0,1]^2\to Q_{ij}$ by

\begin{equation}\label{u-rs}
u(r,s) 
= \frac{(r\sigma_i + \sigma_{ij} + s{\sigma_j})}{\mu} ,
\end{equation}
where $\sigma_{ij} := \sum_{i<k<j} \sigma_k$ and $\mu$ is again $r\norm{\sigma_i} + \norm{\sigma_{i+1}} +\ldots + \norm{\sigma_{j-1}} + s\norm{\sigma_j}$. Given $(r,s) \in[0,1]^2$, the point $u(r,s)$ is the endpoint of a trace path ending in $Q_{ij}$.
We then define $\phi=\phi_{ij}:[0,1]^2\to[0,+\infty)$ to be the height of the geodesic lift of the trace path implicitly described by $u(r,s)$; that is $\phi(r,s)$ is the balayage area spanned by the curve:
\footnote{
To help the reader, we remark that, if $\gamma:[0,1]\to\R^2$ is a piecewise affine curve tracing the vectors $u_0,u_1,\dots,u_n\in\R^2$, then the balayage area spanned by $\gamma$ is 
\[
\frac12 \sum_{0\le a<b\le n} \omega(u_a,u_b) .
\]
}
\begin{equation}\label{phi-rs}
\phi(r,s) = \frac{1}{2\mu^2}\left(\sum_{i<k<j}\omega(r\sigma_i, \sigma_k) + \omega(r\sigma_i, s\sigma_j) + \sum_{i<k_1<k_2<j}\omega(\sigma_{k_1},\sigma_{k_2})+ \sum_{i<k<j}\omega(\sigma_k,s\sigma_j)\right).
\end{equation}
The map $[0,1]^2\to\HH$, $(r,s)\mapsto(u_{ij}(r,s),\phi_{ij}(r,s))$ is then a parametrization of $\Panel^+_{ij}$.
Similarly, the map $[0,1]^2\to\HH$, $(r,s)\mapsto(u_{ij}(r,s),-\phi_{ij}(r,s))$ is a parametrization of $\Panel^-_{ij}$, the part of the basement of $\de B$ that projects onto $Q_{ij}$, see \cite[Theorem 10]{duchin-mooney}. 
Note that this parametrization of the basement panel does not encode the combinatorial information of the trace path geodesic that leads there.

\begin{remark}\label{rem5f784653}
	We will prove most of our results for ceiling points.
	The basement case is then derived via the involutive automorphism 
	$\flip:\HH\to\HH$, $\flip(v,t)=(\flip(v),-t)$,
	where $\flip:\R^2\to\R^2$ is the linear transformation that maps $\vv_{2N} $ to itself and 
	flips the line orthogonal to $\vv_{2N}$. 
	Notice that $\flip^2=\Id$ and $\flip^*\omega = -\omega$.
	
	The basement of the unit ball $B$ of $d$ is mapped to the ceiling of the unit ball $B^\flip$ of the new distance $d^\flip(x,y):=d(\flip(x),\flip(y))$.
	The distance $d^\flip$ is again sub-Finsler with unit disk $Q^\flip=\flip(Q)$. 
	The polygon $Q^\flip$ has vertices 
	$\vv^\flip_k:=\flip(\vv_{-k})$, in anti-clockwise order,
	and edges $e^\flip_k = \flip(-e_{-k-1})$.
	It follows that 
	\[
	\alpha^\flip_{k}\circ\flip = - \alpha_{-k-1} ,
	\]
	hence $\alpha^{\flip,\omega}_k = \flip(\alpha^\omega_{-k-1})$ 
	and  $\sigma^\flip_k = -\flip(\sigma_{-k})$ with $\|\sigma^\flip_k\|^\flip=\|\sigma_{-k}\|$.
	See Figure~\ref{fig5f8deb10}.
	Therefore, if $i<j$, then 
	\[
	u_{-j,-i}^\flip(s,r) = - \flip(u_{ij}(r,s)) .
	\]
	So, if $p\in\de B$ lies in the basement and $\pi(p)= u_{ij}(r,s)$, then $\flip(p)\in\de B^\flip$ lies in the ceiling and $\pi(\flip(p))=-u_{-j,-i}^\flip(s,r)=u^\flip_{-j+N,-i+N}(s,r)$. 
	See Figure~\ref{fig:tracepaths}.
	Finally, for all $p,v\in\HH$, we have
	\[
	\pD d^\flip_e|_{\flip(p)}[\flip(v)] = \pD d_e|_p[v] .
	\]
\end{remark}

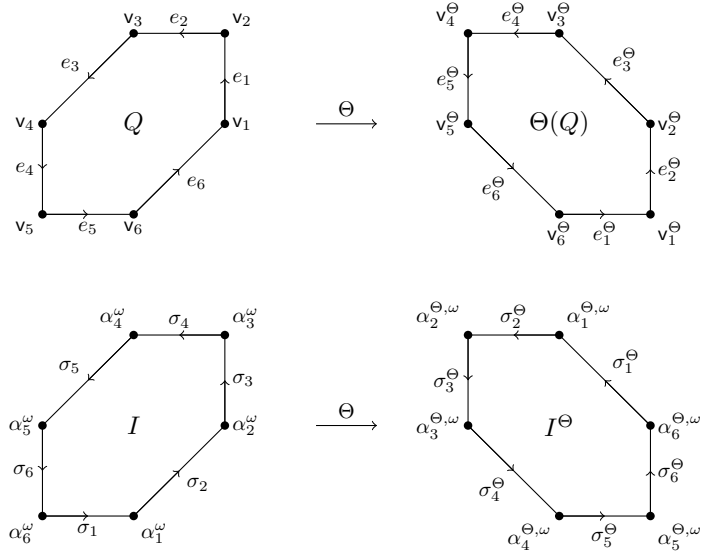
\begin{figure}[h]
\begin{tikzpicture}[scale=0.4, every node/.style={scale=0.8}]

\begin{scope}
\draw[->] (3,3) to (1.5,3);
\draw[->] (3,0) to (3,1.5);
\draw[->] (0,-3) to (1.5,-1.5);
\draw[->] (-3,-3) to (-1.5,-3);
\draw[->] (-3,0) to (-3,-1.5);
\draw[->] (0,3) to (-1.5,1.5);

\draw (3,0) -- (3,3) -- (0,3) -- (-3,0) -- (-3, -3) -- (0,-3) -- cycle;
\fill (3,0) node[right] {$\vv_1$} circle (4 pt);
\fill (3,3) node[above right] {$\vv_2$} circle (4 pt);
\fill (0,3) node[above] {$\vv_3$} circle (4 pt);
\fill (-3,0) node[left] {$\vv_4$} circle (4 pt);
\fill (-3,-3) node[below left] {$\vv_5$} circle (4 pt);
\fill (0, -3) node[below] {$\vv_6$} circle (4 pt);
\draw (3,1.5) node[right] {$ e_1$};
\draw (1.5,3) node[above] {$ e_2$};
\draw (-1.5,1.5) node[above left] {$ e_3$};
\draw (-3,-1.5) node[left] {$ e_4$};
\draw (-1.5,-3) node[below] {$ e_5$};
\draw (1.5, -1.5) node[below right] {$ e_6$};
\draw (0,0) node[scale =1.2] {$Q$};
\end{scope}

\begin{scope}[yshift=-10cm]
\draw[->] (3,3) to (1.5,3);
\draw[->] (3,0) to (3,1.5);
\draw[->] (0,-3) to (1.5,-1.5);
\draw[->] (-3,-3) to (-1.5,-3);
\draw[->] (-3,0) to (-3,-1.5);
\draw[->] (0,3) to (-1.5,1.5);

\draw (3,0) -- (3,3) -- (0,3) -- (-3,0) -- (-3, -3) -- (0,-3) -- cycle;
\draw (3,1.5) node[right] {$\sigma_3$};
\draw (1.5,3) node[above] {$\sigma_4$};
\draw (-1.5,1.5) node[above left] {$\sigma_5$};
\draw (-3,-1.5) node[left] {$\sigma_6$};
\draw (-1.5,-3) node[below] {$\sigma_1$};
\draw (1.5, -1.5) node[below right] {$\sigma_2$};
\draw (0,0) node[scale =1.2] {$I$};

\fill (3,0) node[right] {$\alpha^\omega_2$} circle (4 pt);
\fill (0,3) node[above left] {$\alpha^\omega_4$} circle (4 pt);
\fill (3,3) node[above right] {$\alpha^\omega_3$} circle (4 pt); 
\fill (-3,0) node[left] {$\alpha^\omega_5$} circle (4 pt);
\fill (0,-3) node[below right] {$\alpha^\omega_1$} circle (4 pt);
\fill (-3, -3) node[below left] {$\alpha^\omega_6$} circle (4 pt);

\end{scope}

\begin{scope}[xshift=7cm,yshift=0cm]
\draw[->] (-1,0) to (1,0) ;
\draw (0,0) node[above] {$\flip$};
\end{scope}
\begin{scope}[xshift=7cm,yshift=-10cm]
\draw[->] (-1,0) to (1,0) ;
\draw (0,0) node[above] {$\flip$};
\end{scope}

\begin{scope}[xshift=14cm]
\draw[->] (-3,3) to (-3,1.5);     
\draw[->] (-3,0) to (-1.5,-1.5);
\draw[->] (0,-3) to (1.5,-3);
\draw[->] (3,-3) to (3,-1.5);
\draw[->] (3,0) to (1.5,1.5);
\draw[->] (0,3) to (-1.5,3);

\draw (-3,0) -- (-3,3) -- (0,3) -- (3,0) -- (3, -3) -- (0,-3) -- cycle;
\fill (-3,0) node[left] {$\vv^\flip_5$} circle (4 pt);
\fill (-3,3) node[above left] {$\vv^\flip_4$} circle (4 pt);
\fill (0,3) node[above] {$\vv^\flip_3$} circle (4 pt);
\fill (3,0) node[right] {$\vv^\flip_2$} circle (4 pt);
\fill (3,-3) node[below right] {$\vv^\flip_1$} circle (4 pt);
\fill (0, -3) node[below] {$\vv^\flip_6$} circle (4 pt);
\draw (-3,1.5) node[left] {$ e^\flip_5$};
\draw (-1.5,3) node[above] {$ e^\flip_4$};
\draw (1.5,1.5) node[above right] {$ e^\flip_3$};
\draw (3,-1.5) node[right] {$ e^\flip_2$};
\draw (1.5,-3) node[below] {$ e^\flip_1$};
\draw (-1.5, -1.5) node[below left] {$ e^\flip_6$};
\draw (0,0) node[scale =1.2] {$\flip(Q)$};
\end{scope}

\begin{scope}[xshift=14cm,yshift=-10cm]
\draw[->] (-3,3) to (-3,1.5);  
\draw[->] (-3,0) to (-1.5,-1.5);
\draw[->] (0,-3) to (1.5,-3);
\draw[->] (3,-3) to (3,-1.5);
\draw[->] (3,0) to (1.5,1.5);
\draw[->] (0,3) to (-1.5,3);

\draw (-3,0) -- (-3,3) -- (0,3) -- (3,0) -- (3, -3) -- (0,-3) -- cycle;
\draw (-3,1.5) node[left] {$\sigma^\flip_3$};
\draw (-1.5,3) node[above] {$\sigma^\flip_2$};
\draw (1.5,1.5) node[above right] {$\sigma^\flip_1$};
\draw (3,-1.5) node[right] {$\sigma^\flip_6$};
\draw (1.5,-3) node[below] {$\sigma^\flip_5$};
\draw (-1.5, -1.5) node[below left] {$\sigma^\flip_4$};
\draw (0,0) node[scale =1.2] {$I^\flip$};

\fill (-3,0) node[left] {$\alpha^{\flip,\omega}_3$} circle (4 pt);
\fill (-3,3) node[above left] {$\alpha^{\flip,\omega}_2$} circle (4 pt); 
\fill (0,3) node[above right] {$\alpha^{\flip,\omega}_1$} circle (4 pt);
\fill (3,0) node[right] {$\alpha^{\flip,\omega}_6$} circle (4 pt);
\fill (3, -3) node[below right] {$\alpha^{\flip,\omega}_5$} circle (4 pt);
\fill (0,-3) node[below left] {$\alpha^{\flip,\omega}_4$} circle (4 pt);

\end{scope}

\end{tikzpicture}
\caption{Example of a norm ball and isoperimetrix, with their transformations under $\flip$.}
\label{fig5f8deb10}
\end{figure}

\begin{figure}[H]
\begin{tikzpicture}[scale=0.4, every node/.style={scale=0.8}]

\begin{scope}
\draw[gray, ->] (-1,-3) -- (6,-3);
\draw[gray,-] (-1,-3) -- (-4,-3);
\draw[gray,->] (-1,-3) -- (-1,5);
\draw[gray,-] (-1,-3) -- (-1,-5);
\fill[blue, opacity=.3] (-1, -3) -- (1,-3) -- (4,0) -- (4,3) -- (1,3) -- (0,2) -- cycle;
\draw[dashed] (0,2) -- (-2, 0) -- (-2,-3) -- (-1,-3);
\draw (-1, -3) -- node[below] {$r\sigma_1$}  (1,-3) --node[below right] {$\sigma_2$}  (4,0) -- node[right] {$\sigma_3$} (4,3) --node[above] {$\sigma_4$}  (1,3) -- node[above left] {$s\sigma_5$}  (0,2);
\draw[->] (-1,-3) -- (0,-3);
\draw[->] (1,-3) -- (2.5,-1.5);
\draw[->]  (4,0) -- (4, 1.5);
\draw[->]  (4,3) -- (2.5, 3);
\draw[->]  (1,3) -- (.5,2.5);
\draw (-1,-3) -- (4,0);
\draw (-1,-3) -- (4,3);
\draw (-1,-3) -- (1,3);
\draw (-1,-3) -- (0,2);
\fill (4,0) circle (4 pt);
\fill (4,3) circle (4 pt);
\fill (1, -3) circle (4 pt);
\fill (1, 3) circle (4 pt);
\fill[green] (-1,-3) circle (4 pt);
\fill[red] (0,2)  circle (4 pt);
\end{scope}

\begin{scope}[xshift = 14cm, yshift = 2cm]
\draw[gray,->] (0,0) -- (6,0);
\draw[gray,-] (0,0) -- (-4,0);
\draw[gray,-] (0,0) -- (0,-7);
\draw[gray,->] (0,0) -- (0,3);
\fill[red, opacity=.3] (0,0) -- (1,1) -- (4,1) -- (4, -2) -- (1, -5) -- (-1, -5);
\draw[dashed] (-1,-5) -- (-2,-5) -- (-2,-2) -- (0,0);
\draw (0,0) -- node[above left,fill=white] {$ -s\sigma_5$} (1,1) -- node[above] {$-\sigma_4$} (4,1) -- node[right] {$-\sigma_3$} (4, -2) -- node[below right] {$-\sigma_2$} (1, -5) -- node[below, fill=white] {$ -r\sigma_1$} (-1, -5);
\draw[->] (0,0) -- (.5,.5);
\draw[->] (1,1) -- (2.5,1);
\draw[->]  (4,1) -- (4, -.5);
\draw[->]  (4,-2) -- (2.5, -3.5);
\draw[->]  (1,-5) -- (0,-5);
\draw (0,0) -- (4,1);
\draw (0,0) -- (4,-2);
\draw (0,0) -- (1,-5);
\draw (0,0) -- (-1,-5);
\fill (1,1) circle (4 pt);
\fill (4,1) circle (4 pt);
\fill (4, -2) circle (4 pt);
\fill (1, -5) circle (4 pt);
\fill[green] (0,0) circle (4 pt);
\fill[red] (-1,-5)  circle (4 pt);
\end{scope}

\begin{scope}[xshift = 28cm]
\draw[gray, -] (1,-3) -- (-6,-3);
\draw[gray,->] (1,-3) -- (4,-3);
\draw[gray,->] (1,-3) -- (1,5);
\draw[gray,-] (1,-3) -- (1,-5);
\fill[red, opacity=.3] (1, -3) -- (-1,-3) -- (-4,0) -- (-4,3) -- (-1,3) -- (0,2) -- cycle;
\draw[dashed] (0,2) -- (2, 0) -- (2,-3) -- (1,-3);
\draw (1, -3) -- node[below] {$r\flip(\sigma_1)=r\sigma^\flip_{-1}$}  
	(-1,-3) --node[below left] {$\sigma^\flip_{-2}$}  
	(-4,0) -- node[left] {$\sigma^\flip_{-3}$} 
	(-4,3) --node[above] {$\sigma^\flip_{-4}$}  
	(-1,3) -- node[above right] {$s\sigma^\flip_{-5}$}  
	(0,2);
\draw[->] (1,-3) -- (0,-3);
\draw[->] (-1,-3) -- (-2.5,-1.5);
\draw[->]  (-4,0) -- (-4, 1.5);
\draw[->]  (-4,3) -- (-2.5, 3);
\draw[->]  (-1,3) -- (-.5,2.5);
\draw (1,-3) -- (-4,0);
\draw (1,-3) -- (-4,3);
\draw (1,-3) -- (-1,3);
\draw (1,-3) -- (0,2);
\fill (-4,0) circle (4 pt);
\fill (-4,3) circle (4 pt);
\fill (-1, -3) circle (4 pt);
\fill (-1, 3) circle (4 pt);
\fill[green] (1,-3) circle (4 pt);
\fill[red] (0,2)  circle (4 pt);
\end{scope}

\end{tikzpicture}
\caption{Three trace paths with similar combinatorics whose lifts end at ceiling point $p = (u, \phi(u))$, basement point $p\inv = (-u, -\phi(u))$, and the image $\flip(p)$ of $p$ under the involution $\flip$, respectively. Note that the trace path of $p\inv$ has the reverse parametrization as that of $p$.}
\label{fig:tracepaths}
\end{figure}
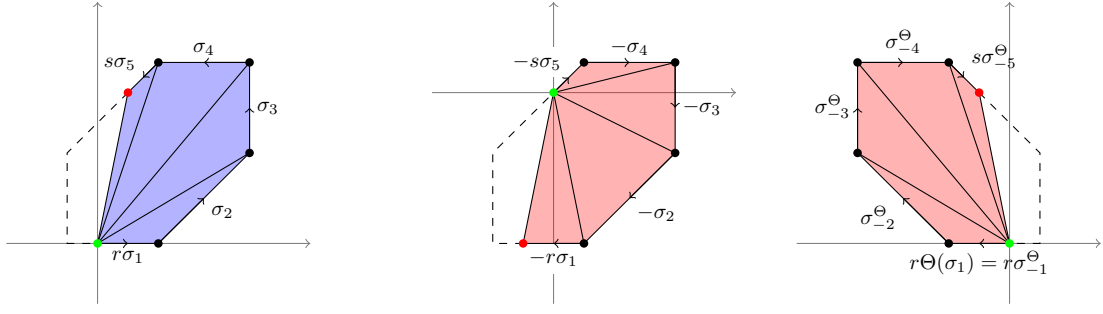

\subsection{The theorem}
Let $d$ be a polygonal sub-Finsler metric on the Heisenberg group $\HH$. 
The fundamental lemma identifying horofunctions with Pansu derivatives (Lemma~\ref{lem:pansuderiv}) applies in this case, 
but we need to take care in describing all possible blow-ups of the distance function at points on the sphere.

These blow-ups take two forms.
As we will explain below, $\HH$ is partitioned so that the function $d_e$ is $C^\infty$ in the interior of each region.
So, on the one hand, we have the points of the unit sphere where $d_e$ is smooth, and thus the only blow-up is the Pansu derivative of $d_e$. 
On the other hand, on the non-smooth part of the unit sphere, which we call the {\em seam}, the blow-up of $d_e$ is defined piecewise as in Theorem~\ref{piecetogether}.

The unit sphere of $d$ is made of smooth and non-smooth points.
Smooth points are the interior points of the panels on ceiling, basement, and walls.
Non-smooth points are on the seams between those panels, that is: 
north and south poles,
star-like seams near the north and south poles,
and seams between ceiling or basement and wall panels,
vertices of $Q$.
See Figure \ref{fig:seams} for the seams along a hexagonal unit sphere: each type of seam point intersects different combinations of panel dilation cones and hence provides a different kind of blow-up function.
We will study the blow-ups of the distance function $d_e$ in each case separately.
The results are summarized in the following theorem.

\begin{figure}[ht]
\centering
\begin{tikzpicture}[>=latex,scale=.7]

\node at (0,0) {\includegraphics[width=1.75in]{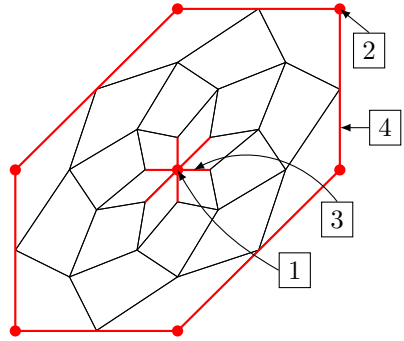}};

\draw[red, thick] (3.05, 3.05) -- (3.05, 0) -- (0, -3.05) -- (-3.05, -3.05) -- (-3.05, 0) -- (0, 3.05) -- cycle;

\foreach \x in {(3.05, 3.05), (3.05, 0), (0, -3.05), (-3.05, -3.05), (-3.05, 0),(0, 3.05),(0,0)}
	\fill[red] \x circle (3 pt);

\draw[red, thick] (-.61,0) -- (.61,0)
   (0,.61) -- (0, -.61)
   (.61, .61) -- (-.61,-.61); 

\draw[->] (1.9,-1.9) to[bend left = 15] (0,0);
\node [draw, right] at (1.9,-1.9)  {$1$};

\draw[->] (3,-.6) to[bend right = 40] (.3,0);
\node [draw, below] at (3,-.6)  {$3$};

\draw[->] (3.6, .8) to[bend left = 0] (3.05,.8);
\node [draw, right] at (3.6,.8)  {$4$};

\draw[->] (3.6,2.6) to[bend right = 0] (3.05, 3.05);
\node [draw, below] at (3.6,2.6)  {$2$};

\end{tikzpicture}
\caption{A top-down view of the four types of seams, in red, on a hexagonal unit sphere; 1)~north and south pole; 2)~vertices; 3)~star-like seams near poles; 4)~wall seams.}
\label{fig:seams}
\end{figure}

\begin{theorem}[Blow-ups of $d_e$]\label{thm5f7476ff}
	Using the notation of Section~\ref{coordinates},
	the blow-ups of $d_e$ at a point $p$ on the sphere of $d$ fall in one of the following cases.

	\vspace{4pt}
	\textbf{Smooth points:}
	\begin{enumerate}[leftmargin=*,label=($S$\arabic*)]
	\item If $p$ is in the interior of $\Panel^+_{ij}$ such that the $\pi(p) = u_{ij}(r,s)$, then
		the Pansu derivative of $d_e$ exists at $p$ and 
		\[
		\pD d_e\vert_p(v,t) = ((1-s)\alpha_{j-1}+ s\alpha_j)(v).
		\]
	\item If $p$ is in the interior of $\Panel^-_{ij}$ such that the $\pi(p) = u_{ij}(r,s)$, then
		the Pansu derivative of $d_e$ exists at $p$ and
		\[
		\pD d_e\vert_p(v,t) =  ((1-r)\alpha_{i}+ r\alpha_{i-1})(v).
		\]
	\item If $p$ is in the interior of $\Panel_{i,i+1}$, then
		\[
		\pD d_e\vert_p(v,t) = \alpha_i(v).
		\]
	\end{enumerate}
	\textbf{Non-smooth points:}
	\begin{enumerate}[leftmargin=*,label=($\not{}\! S$\arabic*)]
	\item North and south poles
		\begin{enumerate}
		\item For $w\in\R^2$,
			\[f(v,t) = \norm{w} - \norm{w-v} ; \]
		\item For $C \in \R$ and $i \in \{1,\ldots, 2N\}$
			\[ 
			f(v,t) = \begin{cases}
				\alpha_i(v) +c_1& \omega(\vv_i,v) \leq C\\
				\alpha_{i-1}(v) + c_2 & \omega(\vv_i,v) > C
				\end{cases}
				;
			\]
		\item For $i \in \{1,\ldots, 2N\}$ (corresponding to $C\in \{-\infty,+\infty\}$)
			\[
			f(v,t) = \alpha_i(v) ;
			\]
		\end{enumerate}
	\item $i$-th vertex of $Q$, $i \in \{1,\ldots, 2N\}$, for $C \in \R\cup \{-\infty,+\infty\}$:
			\[ 
			f(v,t) = \begin{cases}
				\alpha_i(v) +c_1& \omega(\vv_i,v) \geq C\\
				\alpha_{i-1}(v) + c_2 & \omega(\vv_i,v) < C
				\end{cases}
				;
			\]
	\item Star-like seams
		\begin{enumerate} 
		\item Near the north pole, for $C \in \R\cup \{-\infty,+\infty\}$ and $s \in (0,1]$:
		\[
		f(v,t) = \begin{cases}
		\alpha_{i-1}(v)+ c_1 & \omega(\vv_i,v) \geq C\\
		((1-s)\alpha_{i-1} + s\alpha_{i})(v) + c_2& \omega(\vv_i,v) < C
		\end{cases}
		\]
		\item Near the south pole, for $C \in \R\cup \{-\infty,+\infty\}$ and $s \in (0,1]$:
		\[
		f(v,t) = \begin{cases}
		\alpha_{i}(v) + c_1& \omega(\vv_i,v) \leq C\\
		((1-s)\alpha_{i} + s\alpha_{i-1})(v) + c_2& \omega(\vv_i,v) > C
		\end{cases}
		\]
		\end{enumerate}
	\item Wall seams
		\begin{enumerate}
		\item Between wall and ceiling panels, for $C \in \R\cup \{-\infty,+\infty\}$ and $s \in (0,1]$:
		\[
		f(v,t) = \begin{cases}
		\alpha_{i-1}(v)+ c_1 & \omega( \vv_i,v) \leq C \\
		((1-s)\alpha_{i-1} + s\alpha_{i})(v) + c_2& \omega(\vv_i,v) > C
		\end{cases}
		\]
		\item Between wall and basement panels, for $C \in \R\cup \{-\infty,+\infty\}$ and $s \in (0,1]$:
		\[
		f(v,t) = \begin{cases}
		\alpha_i(v)+ c_1 & \omega(\vv_i,v) \geq C \\
		((1-s)\alpha_{i} + s \alpha_{i-1})(v)+ c_2 & \omega(\vv_i,v) < C
		\end{cases}
		\]
		\end{enumerate}
	\end{enumerate} 
\end{theorem}

We remark that in each of the non-smooth cases, the constants $c_1$ and $c_2$ are uniquely determined by the value of $C$ since by definition $f(0,0) = 0$ and $f$ is continuous.
The proof of this theorem is the content of the rest of the section. 
Before diving into it, we present several consequences, in particular the description of the horoboundary of $(\HH,d)$.

Theorem~\ref{thm5f7476ff} has corollaries concerning the regularity of $d_e$ on the sphere.
Indeed, since the Pansu derivative of $d_e$ on the ceiling depends only on the endpoints of trace paths, 
it follows that $\pD d_e$ is continuous on the ceiling and the basement of $\de B$,
except to the star-like seams near the poles, in red in Figure~\ref{fig:samePD}.
We could draw a similar figure for the basement, where the families would spiral in anticlockwise, instead of clockwise.

\begin{corollary}\label{cor5f72adee}
Except for star-like sets near the north and south poles, 
the function $d_e$ has continuous Pansu derivative 

in the interior of the ceiling and the basement of $\de B$.

\end{corollary}

Corollary~\ref{cor5f72adee} with Proposition~\ref{prop5f72ae06} implies that $\de B$ is a $C^1_H$ submanifold of $\HH$ (in the sense of~\cite{2020arXiv200402520J}) at every interior point of the ceiling or basement, outside the star-like sets.

\begin{figure}[H]

\centering
\begin{tikzpicture}[>=latex,scale=.8]

\node at (0,0) {\includegraphics[width=2in]{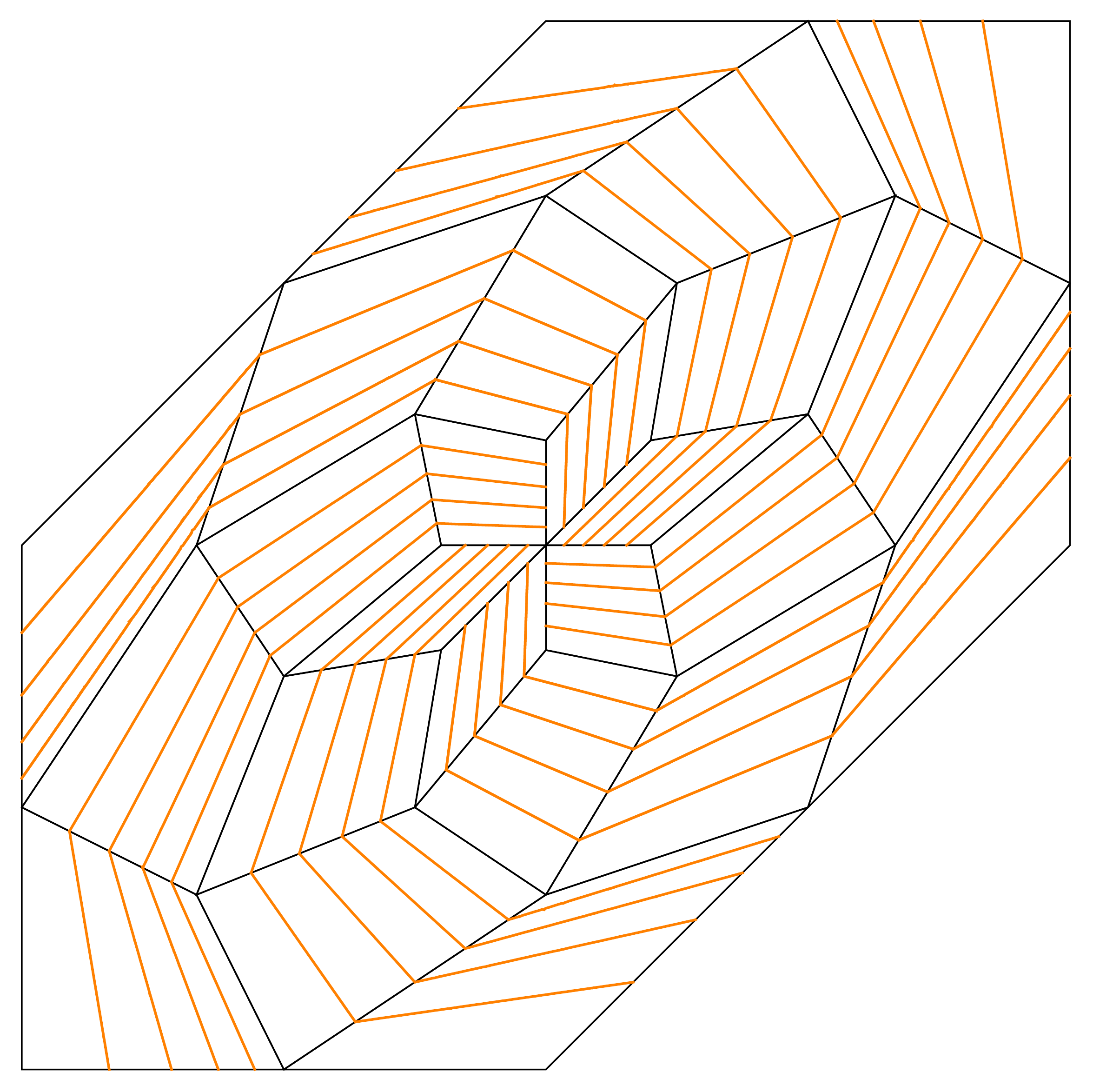}};
\draw[red, very thick] (-.61,0) -- (.61,0)
   (0,.61) -- (0, -.61)
   (.61, .61) -- (-.61,-.61); 

\node at (10,0) {\scalebox{1}[-1]{\includegraphics[width=2in,angle=90]{ceil-samePD.pdf}}};
\draw[red, very thick] (-.61+10,0) -- (.61+10,0)
   (0+10,.61) -- (0+10, -.61)
   (.61+10, .61) -- (-.61+10,-.61); 

\end{tikzpicture}

\caption{A top-down view of families of ceiling points (left) and a bottom-up view of families of basement points (right) with the same Pansu derivatives.}
\label{fig:samePD}
\end{figure}

\begin{theorem}[The horofunction boundary of $(\HH,d)$]\label{thm:main}
	The horofunction boundary of $(\HH,d)$ is the union of the image of the following embeddings in $C(\HH)$:
	First, a disk given by $K:\R^2\to C(\HH)$,
	$w\mapsto f_w$, where $f_w(v,s)=\|w\|-\|v-w\|$.
	The boundary of $K(\R^2)$ in $C(\HH)$ is $-\de_h(\R^2,\norm{\cdot})$.
	
	Second, for each $i$, we have the following maps	
	 $[0,1]\times[-\infty,\infty]\to C(\HH)$:
	\begin{align*}
	\psi_i^\vee(s,a)
		&= \left( \alpha_{i-1}-(a\vee0) \right) \vee \left( (1-s)\alpha_{i-1}+s\alpha_i+(a\wedge0) \right) \\
	\psi_i^\wedge(s,a)
		&= \left( \alpha_{i-1}-(a\wedge0) \right) \wedge \left( (1-s)\alpha_{i-1}+s\alpha_i+(a\vee0) \right) \\
	\xi_i^\vee(s,a) 
		&= \left( \alpha_{i}-(a\vee0) \right) \vee \left( (1-s)\alpha_{i-1}+s\alpha_i+(a\wedge0) \right) \\
	\xi_i^\wedge(s,a)
		&= \left( \alpha_{i}-(a\wedge0) \right) \wedge \left( (1-s)\alpha_{i-1}+s\alpha_i+(a\vee0) \right) .
	\end{align*}
	For each $i$, the image of these four maps is two spheres glued together along a meridian.
	The second meridian of each of the two spheres is 
	the segment between $\alpha_{i-1}$ and $\alpha_i$ in $\de_h(\R^2,\|\cdot\|)$ and $-\de_h(\R^2,\|\cdot\|)$, respectively.
\end{theorem}

Theorem~\ref{thm:main} is proven by inspection of the functions listed in Theorem~\ref{thm5f7476ff}.

It turns out that all of the smooth points on the ceiling, basement, and vertical walls of $\de B$ contribute a circle's worth of functions to the horoboundary. Indeed, they all have Pansu derivatives which lie in $L^*$, the boundary of the dual ball $Q^*$. This is analogous to results in the sub-Riemannian case; Klein--Nicas in \cite{KN-cc} showed that the smooth points contribute a circle's worth of functions to the boundary, while the rest of the boundary comes from vertical sequences, analogous to our Theorem~\ref{vertSeq}.

See Figures~\ref{colors} and~\ref{fig5f766085}.
In Figure \ref{colors}, we introduce a sense of directionality to the horofunction boundary. Recall that to any sequence $\{q_n\} \subset \HH$ converging to a horofunction, we can associate sequences $\{p_n\}_n \subset \de B$ and $\{\epsilon_n\}_n \subset \R$, where $\delta_{\epsilon_n} q_n = p_n$. For each horofunction $f \in \de_h\HH$, there exist sequences $\{q_n\}_n \leftrightarrow (\{p_n\}_n,\{\epsilon_n\}_n)$ such that $q_n \to f$ and $p_n \to p \in \de B$. This assigns directions to horofunctions in the boundary. This correspondence between the boundary and the unit sphere is far from bijective. There exist families of directions, such as each blue vertical wall panel, which collapse to single points in the boundary. On the other hand, there are directions, such as the purple north and south poles, which blow-up to 1- or 2- dimensional families in the boundary. In these cases, which boundary point you converge to will depend on how exactly $q_n$ goes off to infinity. The colors in the figures allow us to see which directions on the sphere converge to which families horofunctions.

\begin{corollary}\label{busemann}
Let $d$ be a polygonal sub-Finsler metric on $\HH$. Then the set of Busemann functions is homeomorphic to a circle.
\end{corollary}
\begin{proof}
	The only infinite geodesic rays based at the origin are the beeline geodesics, the lifts of $L$-norm geodesics in the plane to admissible paths in $\HH$, as described in Section \ref{coordinates}. Thus, the set of Busemann functions comes from blow-ups of points in the vertical walls and vertices of the unit sphere and is isomorphic to $\de_h(\R^2, \|\cdot\|_Q)~\cong~S^1$.
\end{proof}

%
%

\newcommand{\ml}{3}
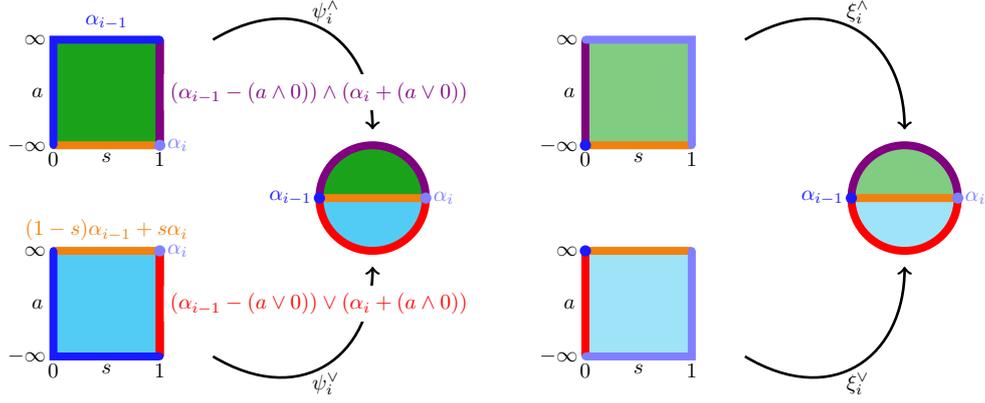
\begin{figure}
\begin{tikzpicture}[scale=0.7, every node/.style={scale=0.8}]

\colorlet{giallo}{blue!90!white}
\colorlet{blu}{cyan}
\colorlet{verde}{orange!90!gray}
\colorlet{rosso}{blue!50!white}
\colorlet{viola}{violet}

\draw[->,line width=1pt] (5-\ml,1) .. controls (6.66-\ml,2) and (8-\ml,1) ..node[above]{$\psi^\wedge_i$} (8-\ml,-0.7);

\draw[->,line width=1pt] (5-\ml,-5) .. controls (6.66-\ml,-6) and (8-\ml,-5) ..node[below]{$\psi^\vee_i$} (8-\ml,-3.3);

\begin{scope}
\coordinate (A) at (-1,-1);
\coordinate (B) at (1,-1);
\coordinate (C) at (1,1);
\coordinate (D) at (-1,1);

\fill[black!40!green, opacity=.9] (A) -- (B) -- (C) -- (D) -- cycle;
\draw 
	(A) node[below]{$0$} node[left]{$-\infty$} --node[below]{$s$} 
	(B) node[below]{$1$} -- 
	(C) -- 
	(D) node[left]{$\infty$} --
	cycle;
\draw[line width=3pt, verde, opacity=1] (A) -- (B);
\draw[line width=3pt, viola, opacity=1] (B) --node[right,fill=white]{$(\alpha_{i-1}-(a\wedge0))\wedge(\alpha_i+(a\vee0))$} (C);
\draw[line width=3pt, viola, opacity=1] (B) -- (C);
\draw[line width=3pt, giallo, opacity=1] (C) --node[above]{$\alpha_{i-1}$} (D) -- node[black,left]{$a$}  (A);
\fill[rosso] (B)  circle (3 pt) node[above,right]{$\alpha_i$};
\fill[giallo] (A)  circle (2 pt);
\fill[giallo] (C)  circle (2 pt);
\end{scope}

\begin{scope}[yshift=-4cm]
\coordinate (A) at (-1,-1);
\coordinate (B) at (1,-1);
\coordinate (C) at (1,1);
\coordinate (D) at (-1,1);

\fill[white!40!cyan, opacity=.9] (A) -- (B) -- (C) -- (D) -- cycle;
\draw 
	(A) node[below]{$0$} node[left]{$-\infty$} --node[below]{$s$} 
	(B) node[below]{$1$} -- 
	(C) -- 
	(D) node[left]{$\infty$} --
	cycle;
\draw[line width=3pt, red, opacity=1] (B) --node[right,fill=white]{$(\alpha_{i-1}-(a\vee0))\vee(\alpha_i+(a\wedge0))$} (C);
\draw[line width=3pt, red, opacity=1] (B) -- (C);
\draw[line width=3pt, verde, opacity=1] (C) --node[above]{$(1-s)\alpha_{i-1}+s\alpha_i$} (D);
\draw[line width=3pt, giallo, opacity=1] (D) --node[black,left]{$a$} (A) -- (B);
\fill[rosso] (C)  circle (3 pt) node[above,right]{$\alpha_i$};
\fill[giallo] (B)  circle (2 pt);
\fill[giallo] (D)  circle (2 pt);
\end{scope}

\begin{scope}[xshift=5cm,yshift=-2cm]
\coordinate (O) at (0,0) ;
\coordinate (A) at (-1,0) ;
\coordinate (C) at (1,0) ;

\fill[line width=3pt, black!40!green, opacity=.9] (C) arc[start angle=0, end angle=180, radius =1] -- cycle;
\draw[line width=3pt, viola, opacity=1] (C) arc[start angle=0, end angle=180, radius =1];
\fill[line width=3pt, white!40!cyan, opacity=.9] (C) arc[start angle=0, end angle=-180, radius =1] -- cycle;
\draw[line width=3pt, red, opacity=1] (C) arc[start angle=0, end angle=-180, radius =1];
\draw[line width=3pt, verde, opacity=1] (A) -- (C) ;

\fill[rosso] (C)  circle (3 pt) node[right]{$\alpha_i$};
\fill[giallo] (A)  circle (3 pt) node[left]{$\alpha_{i-1}$};
\end{scope}


\begin{scope}[xshift = 10cm]

\draw[->,line width=1pt] (5-\ml,1) .. controls (6.66-\ml,2) and (8-\ml,1) ..node[above]{$\xi^\wedge_i$} (8-\ml,-0.7);

\draw[->,line width=1pt] (5-\ml,-5) .. controls (6.66-\ml,-6) and (8-\ml,-5) ..node[below]{$\xi^\vee_i$} (8-\ml,-3.3);

\begin{scope}
\coordinate (A) at (-1,-1);
\coordinate (B) at (1,-1);
\coordinate (C) at (1,1);
\coordinate (D) at (-1,1);

\fill[black!40!green, opacity=.5] (A) -- (B) -- (C) -- (D) -- cycle;
\draw 
	(A) node[below]{$0$} node[left]{$-\infty$} --node[below]{$s$} 
	(B) node[below]{$1$} -- 
	(C) -- 
	(D) node[left]{$\infty$} --
	cycle;
\draw[line width=3pt, verde, opacity=1] (A) -- (B);
\draw[line width=3pt, viola, opacity=1] (D) --node[black,left]{$a$} (A);
\draw[line width=3pt, rosso, opacity=1] (B) -- (C) -- (D);
\fill[giallo] (A)  circle (3 pt) ;
\fill[rosso] (B)  circle (2 pt);
\fill[rosso] (D)  circle (2 pt);
\end{scope}

\begin{scope}[yshift=-4cm]
\coordinate (A) at (-1,-1);
\coordinate (B) at (1,-1);
\coordinate (C) at (1,1);
\coordinate (D) at (-1,1);

\fill[white!40!cyan, opacity=.5] (A) -- (B) -- (C) -- (D) -- cycle;
\draw 
	(A) node[below]{$0$} node[left]{$-\infty$} --node[below]{$s$} 
	(B) node[below]{$1$} -- 
	(C) -- 
	(D) node[left]{$\infty$} --
	cycle;

\draw[line width=3pt, verde, opacity=1] (C) -- (D);
\draw[line width=3pt, red, opacity=1] (D) --node[black,left]{$a$} (A);
\draw[line width=3pt, rosso, opacity=1] (A) -- (B) -- (C) ;
\fill[giallo] (D)  circle (3 pt);
\fill[rosso] (A)  circle (2 pt);
\fill[rosso] (C)  circle (2 pt);
\end{scope}

\begin{scope}[xshift=5cm,yshift=-2cm]
\coordinate (O) at (0,0) ;
\coordinate (A) at (-1,0) ;
\coordinate (C) at (1,0) ;

\fill[line width=3pt, black!40!green, opacity=.5] (C) arc[start angle=0, end angle=180, radius =1] -- cycle;
\draw[line width=3pt, viola, opacity=1] (C) arc[start angle=0, end angle=180, radius =1];
\fill[line width=3pt, white!40!cyan, opacity=.5] (C) arc[start angle=0, end angle=-180, radius =1] -- cycle;
\draw[line width=3pt, red, opacity=1] (C) arc[start angle=0, end angle=-180, radius =1];
\draw[line width=3pt, verde, opacity=1] (A) -- (C) ;

\fill[rosso] (C)  circle (3 pt) node[right]{$\alpha_i$};
\fill[giallo] (A)  circle (3 pt) node[left]{$\alpha_{i-1}$};
\end{scope}
\end{scope}

\end{tikzpicture}

\caption{A schematic description of the maps $\psi_i^\vee$, $\psi_i^\wedge$, $\xi_i^\vee$ and $\xi_i^\wedge$.}
\label{fig5f766085}
\end{figure}

\begin{figure}[H]
\begin{minipage}{.4\textwidth}
\centering
\includegraphics[width=.75\textwidth]{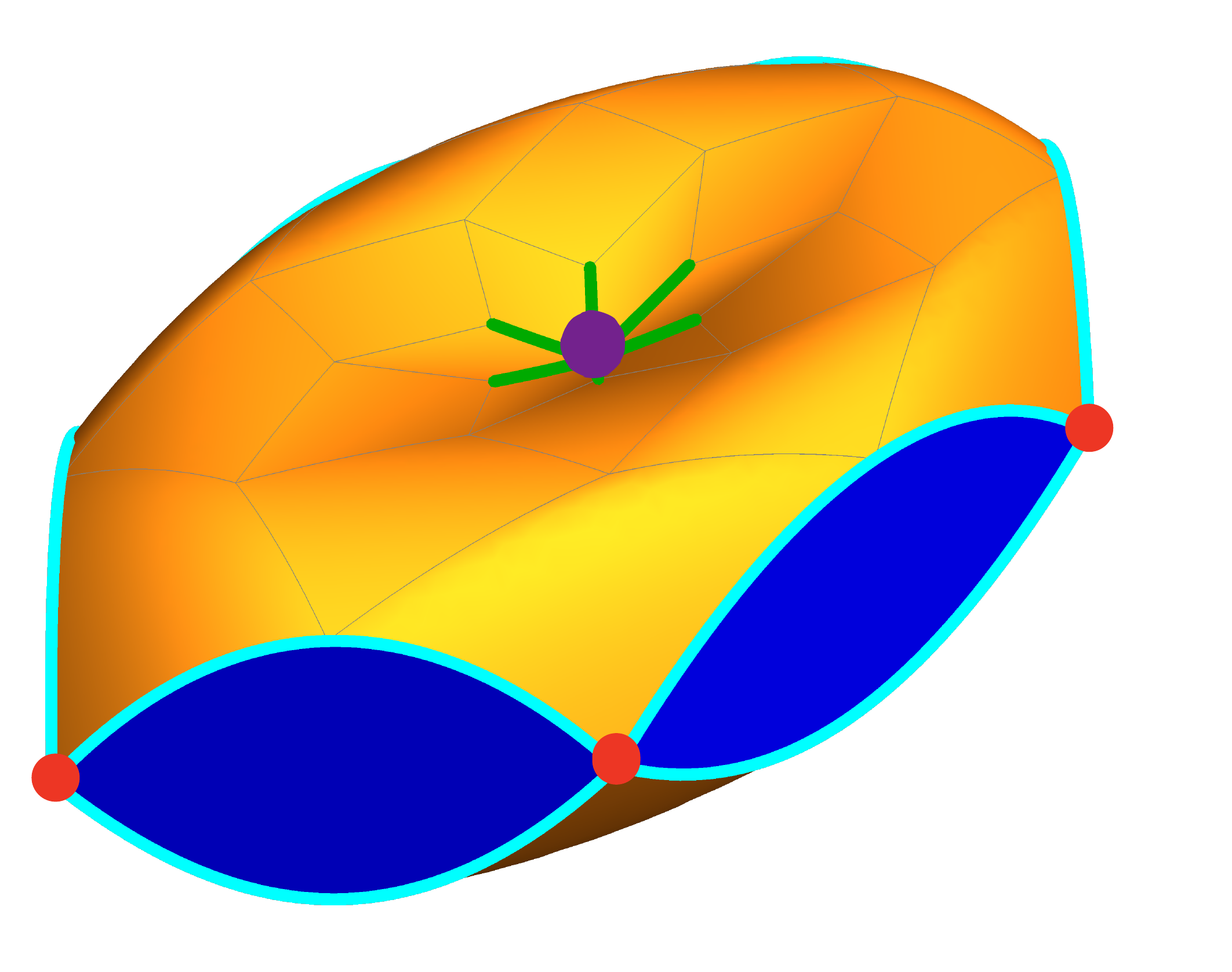}
\end{minipage}
\begin{minipage}{.4\textwidth}
\centering
\includegraphics[width=.75\textwidth]{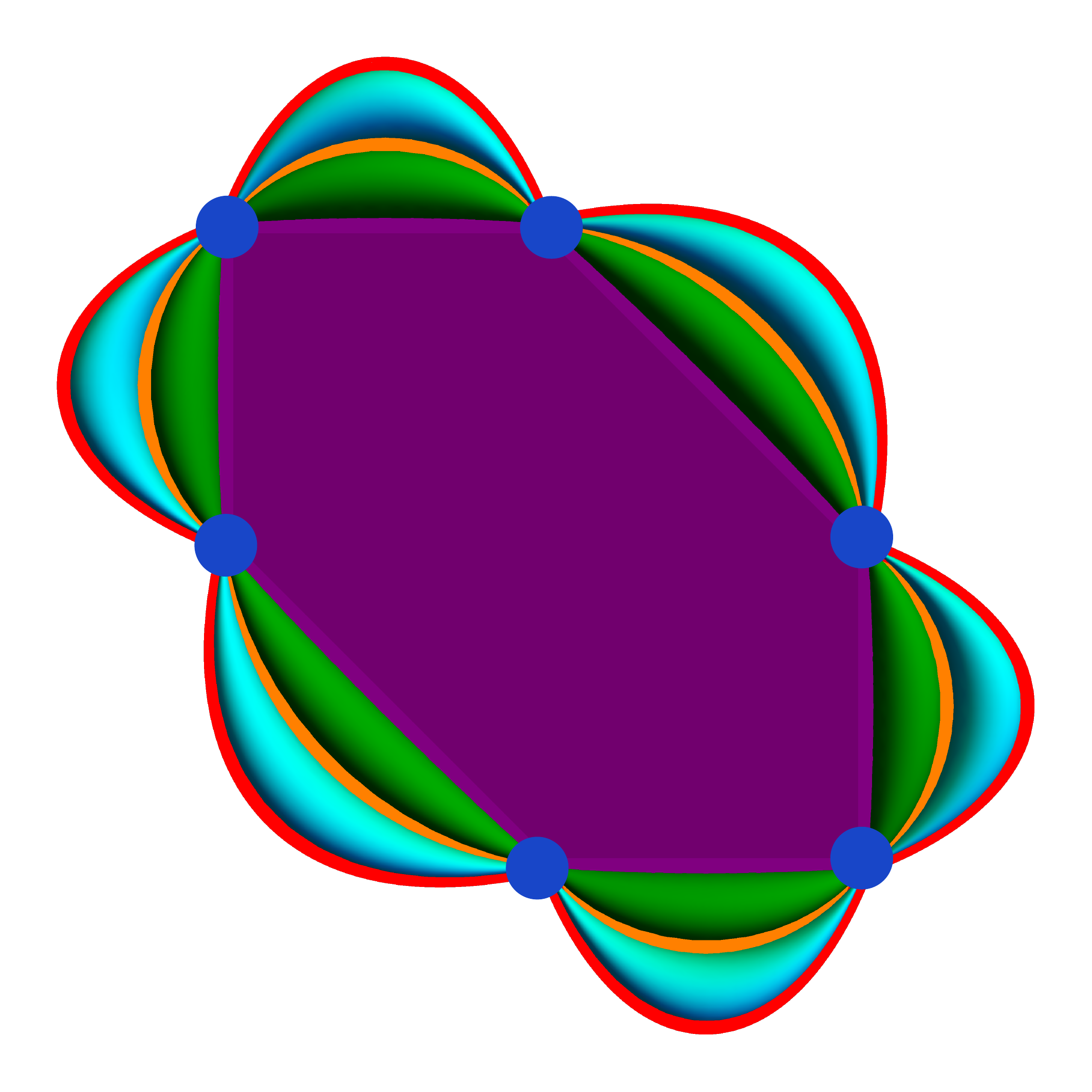}
\end{minipage}
\caption{The unit sphere coming from the hexagonal norm on the left, and the corresponding sub-Finsler boundary on the right.}
\label{colors}
\end{figure}

\subsection{Blow-ups of $d_e$ at smooth points}
First we consider blow-ups of $d_e$ at smooth points on $\de B$, such as in the interior of each ceiling, basement, or wall panel making up the unit sphere. Since $d_e$ is smooth in the interior of each of these panels, it is strictly Pansu differentiable.

We know from above that ceiling and basement points are reached by geodesics which are lifts of trace paths. It turns out that the Pansu derivative of $d_e$ on the ceiling or basement depends only on where in the isoperimetrix $I$ the trace path ends and is independent of the rest of the shape of the trace path.

\begin{proposition}[Ceiling and basement Pansu derivatives]\label{ceilingPD}
If $p\in\de B$ is a ceiling point with $\pi(p)=u_{ij}(r,s)$, $j\notin\{i,i+1\}$,

then the Pansu derivative of $d_e$ exists at $p$, and
\[
\pD d_e\vert_p(v,t) = ((1-s)\alpha_{j-1}+ s\alpha_j)(v).
\]
Similarly, if $p\in\de B$ is a basement point with $\pi(p)=-u_{ij}(r,s)$, $i< j$,
then the Pansu derivative of $d_e$ exists at $p$, and
\[
\pD d_e\vert_p(v,t) =  ((1-r)\alpha_{i}+ r\alpha_{i-1})(v).
\]
\end{proposition}

\begin{proof}
Given that $p$ is in the interior of $\Panel^+_{ij}$, $d_e$ is smooth at $p$, and hence the Pansu derivative exists. Pansu derivatives are linear and invariant on vertical fibers, so we are looking for a linear functional $A\in(\R^2)^*$ such that $\pD d_e(p)(v,t)=A[v]$.

Let $\gamma_1:[0,1+\epsilon]\to \HH$ be the unit-speed trace path geodesic from the origin to $\gamma_1(1)=p$. 
Since $p$ is in the interior of $\Panel_{ij}^+$, 

for sufficiently small $h$ we have $\gamma_1(1 + h) = p\delta_h\vv_j$, and so
\[
1 = \lim_{h\to 0} \frac{d_e(\gamma_1(1+h)) - d_e (\gamma_1(1))}{h} = \lim_{h\to0}\frac{d_e(p\delta_h\vv_j) - d_e (p)}{h} = \pD d_e\vert_p(\vv_j).
\]

Next, let $\gamma_2:(-\epsilon, \epsilon) \to \HH$ be a $C^2$ path along the unit sphere $\de B$ which is horizontal at $p$, with $\gamma_2(0) = p$ and $\gamma_2'(0) = w \in \Delta_p$. 
Since $\gamma_2$ is on the unit sphere, $d_e(\gamma_2) \equiv 1$.  A consequence of the horizontality of $\gamma_2$ at $p$ is the limit $\lim_{h\to 0} \delta_{1/h}(p\inv\gamma_2(h)) = \gamma_2'(0) = w$, and so by the Pansu differentiability of $d_e$ at $p$,
\[
0 = \lim_{h\to 0} \frac{d_e(\gamma_2(h)) - d_e (\gamma_2(0))}{h}  = \pD d_e\vert_p(w).
\]
Now, we seek an expression for $w$. 

Since $\Panel_{ij}^+ = \{(u(r,s), \phi(r,s)) : r, s \in [0,1]\}$,
where $u$ and $\phi$ are defined in~\eqref{u-rs} and in~\eqref{phi-rs},
the tangent bundle to the unit sphere has a frame given by 
$\left(\begin{smallmatrix}
\de_r u\\
\de_r \phi
\end{smallmatrix}\right)$ and  $\left(\begin{smallmatrix}
\de_s u\\
\de_s \phi
\end{smallmatrix}\right)$.
We take the partial derivatives of $u$ and $\phi$ 
\begin{equation}\label{rs_deriv}
\begin{aligned}
\de_ru(r,s) &= \frac{\norm{\sigma_i}}\mu(\vv_i - u(r,s)), 
&& \de_r\phi(r,s) = \frac{\norm{\sigma_i}}{\mu}\left(\frac12 \omega(\vv_i, u(r,s) -r\frac{\sigma_i}{\mu}) - 2\phi(r,s)\right),\\
\de_su(r,s) & =\frac{\norm{\sigma_j}}\mu (\vv_j - u(r,s)), 
&& \de_s\phi(r,s) = \frac{\norm{\sigma_j}}{\mu}\left(\frac12 \omega(u(r,s) - s\frac{\sigma_j}{\mu}, \vv_j) - 2\phi(r,s)\right),
\end{aligned}
\end{equation}
where again $\mu = r\norm{\sigma_i} + \norm{\sigma_{i+1}} +\ldots + \norm{\sigma_{j-1}} + s\norm{\sigma_j}$.
If $p = (u,\phi)$ has trace coordinates $(r,s)$ in $\Panel_{ij}^+$, then after rescaling and simplifying, we have
\[
\
T_p\de B = \text{ Span} \left \{ \begin{pmatrix}
\vv_i - u(r,s)\\
\frac1{2}\omega(\vv_i, u(r,s)) - 2\phi(r,s)
\end{pmatrix}, \begin{pmatrix}
\vv_j - u(r,s)\\
\frac1{2}\omega(u(r,s), \vv_j) - 2\phi(r,s)
\end{pmatrix}\right\}.
\]
Meanwhile, the horizontal subspace at $p$ is spanned 
by the left translations from the origin to $p$ of $(\vv_i - u(r,s),0)$ and $(\vv_j - u(r,s),0)$.

This gives
\[
\Delta_p= \text{ Span}\left \{\begin{pmatrix}
\vv_i - u(r,s)\\
\frac12\omega(u(r,s),\vv_i)
\end{pmatrix}, \begin{pmatrix}
\vv_j - u(r,s)\\
\frac12\omega(u(r,s),\vv_j)
\end{pmatrix}\right\}.
\]
These two bases for $T_p\de B$ and $\Delta_p$ allow us to find $w$ in the intersection as 
\begin{equation}\label{simplify}
w = \begin{pmatrix}
2\phi(r,s)(\vv_j - \vv_i) + \omega(u,\vv_i)(\vv_j- u)\\
\phi(r,s)\omega(u, \vv_j-\vv_i) + \frac12\omega(u,\vv_i)\omega(u,\vv_j)
\end{pmatrix}.
\end{equation}
The vector $w$ is the left-translation from the origin to $p$ of the horizontal vector 
$(\hat w,0)$, where
$\hat w:=2\phi(r,s)(\vv_j - \vv_i) + \omega(u,\vv_i)(\vv_j- u)$.
Notice that, if we set $A=(1-s)\alpha_{j-1}+s\alpha_j$, then in order to prove $\pD d_e|_p=A$ we only need to show that 
\begin{equation}\label{eq5f776f11}
A[\hat w]=0 ,
\end{equation}
because $Av_j=1$.
The proof of~\eqref{eq5f776f11} is a long computation, of which we describe the main steps.
The strategy is to write $A[\hat w]$ in terms of the symplectic duals $\alpha^\omega_i$ of the covectors $\alpha_i$.
One can easily show that
\[
\|\sigma_j\| = \omega(\alpha^\alpha_{j-1},\alpha^\omega_j) ,
\quad\text{ and }\quad
\vv_j = \frac{ \alpha_j-\alpha_{j-1} }{ \omega(\alpha_{j-1},\alpha_j) } . 
\]

First of all, we have
\begin{multline}\label{eq5f777be7}
A[\hat w]
= \frac1{\mu_{ij}^2} \Bigg( 
	\Bigg( \omega(r\sigma_i,\sigma_{ij}) + \omega(\sigma_{ij,s\sigma_j}) + \omega(r\sigma_i,s\sigma_j) + \sum_{i<a<b<j}\omega(\sigma_a,\sigma_b) \Bigg)
		A[\vv_j-\vv_i] 
	 \\ +
	\omega( r\sigma_i + \sigma_{ij} + s\sigma_j ,\vv_i)
	A[\mu \vv_j - r\sigma_i - \sigma_{ij} - s\sigma_j]
\Bigg) . 
\end{multline}
Secondly, one can check the following equalities: 
\begin{align*}
A[\vv_j-\vv_i] 
	&= 1-\frac{A[\alpha^\omega_i-\alpha^\omega_{i-1}]}{\omega(\alpha_{i-1},\alpha_i)} ,\\
A[\mu \vv_j - r\sigma_i - \sigma_{ij} - s\sigma_j]
	&= A[ r(\alpha^\omega_{i-1} - \alpha^\omega_i) + \alpha^\omega_i ]+ \mu ,\\
\omega( r\sigma_i + \sigma_{ij} + s\sigma_j ,\vv_i	)
	&= \frac{ A[ \alpha^\omega_i-\alpha^\omega_{i-1} ] }{ \omega(\alpha^\omega_{i-1},\alpha^\omega_i) } - 1 ,\\
\mu_{ij}
	&= r \omega(\alpha^\omega_{i-1},\alpha^\omega_i) + \sum_{i<k<j} \omega(\alpha^\omega_{k-1},\alpha^\omega_k) + s\omega(\alpha^\omega_{j-1},\alpha^\omega_j) ,\\
\sum_{i<a<b<j} \omega(\sigma_a,\sigma_b) 
	&= \sum_{i<k<j-1} \omega(\alpha^\omega_{k-1},\alpha^\omega_k) + \omega(\alpha^\omega_{j-2},\alpha^\omega_{j-1} ) - \omega(\alpha^\omega_i,\alpha^\omega_{j-1}) ,
\end{align*}
\begin{multline*}
\omega(r\sigma_i,\sigma_{ij}) + \omega(\sigma_{ij},s\sigma_j) + \omega(r\sigma_i,s\sigma_j) 
	= r\omega(\alpha^\omega_i-\alpha^\omega_{i-1},\alpha^\omega_{j-1}-\alpha^\omega_i) 
		+ s(\omega(\alpha^\omega_{j-1},\alpha^\omega_j) \\
		- \omega(\alpha^\omega_i,\alpha^\omega_j) 
		+ \omega(\alpha^\omega_i,\alpha^\omega_j) ) 
		+ rs (\omega(\alpha^\omega_i,\alpha^\omega_j)
			- \omega(\alpha^\omega_i,\alpha^\omega_{j-1})
			- \omega(\alpha^\omega_{i-1},\alpha^\omega_j) 
			+ \omega(\alpha^\omega_{i-1},\alpha^\omega_{j-1}) ) .
\end{multline*}
Finally, using these formulas to rewrite~\eqref{eq5f777be7}, one easily finds that 
$A[\hat w]$ factorizes into 
$\frac1{\mu_{ij}^2}\left(1-\frac{A[\alpha^\omega_i-\alpha^\omega_{i-1}]}{\omega(\alpha_{i-1},\alpha_i)}\right)$ 
and a polynomial of order two in $r$ and $s$.
Each coefficient of this polynomial is easily shown to be zero, completing the proof of~\ref{eq5f776f11}.

To show the result for basement points, we will use Remark~\ref{rem5f784653} and the result just proved for ceiling points.
Let $p\in\de B$ be in the basement
with $\pi(p)=u_{ij}(r,s)$, $i< j$, so that $\flip(p)$ lies in the ceiling of $\de B^\flip$ and $\pi(\flip(p))=-u_{-j,-i}^\flip(s,r) = u_{-j+N,-i+N}^\flip(s,r) $.
Then, for all $v\in\HH$,
\begin{align*}
	\pD d_e|_p[v]
	&= \pD d^\flip_e|_{\flip(p)}[\flip(v)] \\
	&= ((1-r) \alpha^\flip_{-i-1+N} + r \alpha^\flip_{-i+N})[\flip(v)] \\
	&= -((1-r) \alpha^\flip_{-i-1} + r \alpha^\flip_{-i})[\flip(v)] \\
	&= ( (1-r) \alpha_{i} + r \alpha_{i-1}) (v) .
\end{align*}
This completes the proof.

\end{proof}

\begin{proposition}[Wall Pansu derivatives]\label{wallPD}
If $p$ is in the interior of the wall panel $\Panel_{i,i+1}$, then
\[
 \pD d_e\vert_p(v,t) = \alpha_i(v).
\]
\end{proposition}
\begin{proof}
Let $p = (u,t')$ be in the interior of $\Panel_{i,i+1}$, and let $q = (v,t) \in \HH$. For sufficiently small $\epsilon >0$, the point $p \delta_\epsilon q$ is inside the dilation cone of $\Panel_{i,i+1}$. In this dilation cone, $d_e = \alpha_i\circ\pi$. Thus, by definition of the Pansu derivative and the linearity of $\alpha_i$,
\[
\pD d_e|_p(q) = \lim_{\epsilon \to 0} \frac{d_e(p\delta_\epsilon q) - d_e(p)}{\epsilon} = \lim_{\epsilon \to 0} \frac{\alpha_i(u + \epsilon v) - 1}{\epsilon} = \alpha_i(v).
\]
\end{proof}

\subsection{Blow-ups of $d_e$ at non-smooth points}
We now consider blow-ups of the function $d_e$ at points on the unit sphere which are not smooth, i.e., along the seams of the sphere. 

\subsubsection{Blow-ups near north and south poles}

For each $i$, define the cones
\[
C_i^- = - [0,+\infty) \vv_{i} - [0,+\infty)\vv_{i+1} 
	= \{ \vv_i^\omega\le0\}\cap \{\vv_{i+1}^\omega\ge 0\} 
\]
in $\R^2$.
If $v\in C^-_i$, then $\|v\|=-\alpha_i(v)$.
Recall from~\cite{duchin-mooney} that the non-degenerate $Q_{ij}$ containing $(0,0)$ are $Q_{i+1,i}$ for $i=1,\dots,2N$.
For each $i$ we also define the dilation cones $U_i:=\delta_{(0,+\infty)}\Panel^+_{i+1,i}$.
Notice that
\begin{align*}
	u_{i+1,i}(r,s) 
	&= \frac{1}{\mu} (r\sigma_{i+1} + \sigma_{i+2} + \dots + \sigma_{i-1} + s \sigma_{i}) \\
	&= \frac{1}{\mu} (r\sigma_{i+1} + \sigma_{i+2} + \dots + \sigma_{i-1} + \sigma_{i} + \sigma_{i+1} + s \sigma_{i} - \sigma_{i} - \sigma_{i+1}) \\
	&= - (1-r) \frac{\|\sigma_{i+1}\|}{\mu} v_{i+1} - (1-s) \frac{\|\sigma_{i}\|}{\mu} v_{i} .
\end{align*}
Therefore,  $\pi(\Panel^+_{i+1,i}) = Q_{i+1,i} \subset C_i^-$, 
and thus $U_i\subset \pi^{-1}(C^-_i)$.

\begin{proposition}[Blow-ups at north and south poles]\label{normHorofns}
Let $p$ be the north or south pole of the unit sphere $\de B$. 
Then, all blow-ups of $d_e$ at $p$ are:

		\begin{enumerate}
		\item For $w\in\R^2$,
			\[f(v,t) = \norm{w} - \norm{w-v} ; \]
		\item For $C \in \R\cup\{-\infty,+\infty\}$ and $i \in \{1,\ldots, 2N\}$
			\[ 
			f(v,t) = \begin{cases}
				\alpha_i(v) +c_1& \omega(\vv_i,v) \leq C\\
				\alpha_{i-1}(v) + c_2 & \omega(\vv_i,v) > C
				\end{cases}
				;
			\]

		\end{enumerate}

\end{proposition}
Proposition~\ref{normHorofns} gives a second proof of Theorem~\ref{vertSeq} in the case of polygonal sub-Finsler distances.
\begin{proof}
	Suppose $p$ is the north pole.
	A sufficiently small neighborhood $\Omega$ of $p$ is covered by the dilation cones $U_i$.
	Moreover, up to shrinking $\Omega$, we can suppose $U_i\cap \Omega = \pi^{-1}(C^-_i)\cap \Omega$.
		
	From Proposition~\ref{prop5e8c4f94}, we conclude that all blow-ups of $U_i$ at $p$ are $\HH$, $\emptyset$, left translations of $\pi^{-1}(C^-_i)$, and the half spaces $\pi^{-1}(\{ v_i^\omega\le0\})$ and $\pi^{-1}(\{v_{i+1}^\omega\ge 0\})$.
	
	Next, we see from~\eqref{u-rs}, \eqref{phi-rs} and \eqref{rs_deriv} that $(r,s)\mapsto(u_{i+1,i}(r,s),\phi_{i+1,i}(r,s))$ is well defined in a neighborhood of $(1,1)$ and the image is not tangent to $[\HH,\HH]$ at $(1,1)$.
	Notice that $u_{i+1,i}(1,1)=(0,0)$.
	It follows that the map $\psi_i:(r,s,t)\mapsto \delta_t (u_{i+1,i}(r,s),\phi_{i+1,i}(r,s))$ is a diffeomorphism near to $(1,1,1)$ and $\psi_i(1,1,1)=p$.
	Therefore, we can extend $d_e|_{U_i}$ to a homogeneous smooth function $f_i$ defined in a neighborhood of $p$ by $f_i(q) = t(\psi_i^{-1}(q))$.
	Using Proposition~\ref{ceilingPD} and the smoothness of $f_i$, we deduce that
	\[
	\pD f_i|_p(v,t) = \alpha_{i}(v) .
	\]
	
	We are now in the position to conclude the proof.
	On the one hand, if $\{p_n\}_{n\in\N}\subset\HH$ and $\{\epsilon_n\}_{n\in\N}\subset(0,+\infty)$ are sequences with $p_n\to p$ and $\epsilon_n\to 0$,
	then,  up to passing to a subsequence, we can assume that $\BU(U_i,\{p_n\}_n,\{\epsilon_n\}_n)$ exist for each $i$, by Theorem~\ref{thm5f74be88}.
	Therefore, by Theorem~\ref{piecetogether}, we obtain that $\BU((\HH,d_e),\{p_n\}_n,\{\epsilon_n\}_n)$ is one of the functions listed in the statement.
	
	On the other hand, if $f$ is one of the functions listed in the statement,
	then there are sequences $\{p_n\}_{n\in\N}\subset\HH$ and $\{\epsilon_n\}_{n\in\N}\subset(0,+\infty)$ with $p_n\to p$ and $\epsilon_n\to 0$
	so that the blow-ups $\BU(U_i,\{p_n\}_n,\{\epsilon_n\}_n)$ make the partition of $\HH$ given by $f$ and thus, by Proposition~\ref{prop5e8c4f94}, we obtain
	$(\HH,f) = \BU((\HH,d_e),\{p_n\}_n,\{\epsilon_n\}_n)$.
	
	For the south pole the proof is the same.
\end{proof}

\begin{proposition}[Blow-ups at the north star seam]
	Let $p=(u_{ii}(r,0),\phi_{ii}(r,0))$, $r\in(0,1)$, be a ceiling point above the star of $Q$ in the degenerate panel $\Panel_{ii}^+$.
	All the blow-ups of $d_e$ at $p$ are
		\[
		f(v,t) = \begin{cases}
		\alpha_{i-1}(v)+ c_1 & \omega(\vv_i,v) \geq C\\
		((1-r)\alpha_{i-1} + r\alpha_{i})(v) + c_2& \omega(\vv_i,v) < C
		\end{cases}
		\]
	for $C \in \R\cup \{-\infty,+\infty\}$.
\end{proposition}
\begin{proof}
	Notice that $u_{i,i}(r,0) = u_{i,i-1}(r,1) = u_{i+1,i}(1,r)$, 
	hence $Q_{i,i}\subset Q_{i,i-1}\cap Q_{i+1,i}$.
	A sufficiently small neighborhood $\Omega$ of $p$ is covered by the two cones
	$U_{i-1}$ and $U_i$.
	Up to shrinking $\Omega$, we have
	$U_{i-1}\cap\Omega = \{\vv_i^\omega\ge0\}\cap\Omega$
	and $U_{i}\cap\Omega = \{\vv_i^\omega\le0\}\cap\Omega$.
	
	Thus, arguing like in the proof of Proposition~\ref{normHorofns}, 
	we can smoothly extend both $d_e|_{U_i}$ and $d_e|_{U_{i+1}}$ to $\Omega$ 
	and show that all the blow-ups of $d_e$ at $p$ are those listed in the statement.
\end{proof}

A similar analysis of points in star line segments in the basement of the unit sphere yields the following proposition.

\begin{proposition}[Blow-ups at the south star seam]
	Let $p=(u_{ii}(r,0),-\phi_{ii}(r,0))$, $r\in(0,1)$, be a basement point below the star of $Q$ in the degenerate panel $\Panel_{ii}^-$.
	All the blow-ups of $d_e$ at $p$ are
	\[
		f(v,t) =  \begin{cases}
		\alpha_{i}(v)+c_1 & \omega(\vv_i,v) \leq C\\
		((1-r)\alpha_{i} + r\alpha_{i-1})(v) + c_2  & \omega(\vv_i,v) > C
		\end{cases}.
	\]
	for $C \in \R\cup \{-\infty,+\infty\}$.
\end{proposition}

\begin{proposition}[Blow-ups at the tips of the star seam]
	Let $p\in\de B$ be such that $\pi(p)=u_{ii}(1,0)$.
	Then $d_e$ is Pansu differentiable at $p$.
\end{proposition}
\begin{proof}
	The point $p$ lies at the end of the star line segment and in the intersection of a third panel $\Panel_{i+1,i-1}^\pm$. 
	Checking the $(r,s)$ coordinates of $p$ in the three panels, one sees that the three pieces of the blow-up function are all equal to $\alpha_{i-1}$. 
	Thus the Pansu derivative of $d_2$ at $p$ exists.
\end{proof}

\subsubsection{Blow-ups along wall seams} 

For each $i$, define the cones
\[
C_i^+ = [0,+\infty) \vv_{i} + [0,+\infty)\vv_{i+1} 
	= \{ \vv_i^\omega\ge0\}\cap \{\vv_{i+1}^\omega\le 0\} 
\]
in $\R^2$.
If $v\in C_i^+$, then $\|v\|=\alpha_i(v)$.
For each $i$ we also define the dilation cones $W_i:=\delta_{(0,+\infty)}\Panel^+_{i,i+1}$, where $\Panel^+_{i,i+1}$ is the vertical wall of $\de B$ containing the edge of $Q$ between $v_i$ and $v_{i+1}$.
We recall that 
\[
u_{i,i+1}(r,s) = \frac{r\|\sigma_i\|}{r\|\sigma_i\|+s\|\sigma_{i+1}\|} v_i 
	+ \frac{s\|\sigma_{i+1}\|}{r\|\sigma_i\|+s\|\sigma_{i+1}\|} v_{i+1}
\]
is a convex combination of $v_i$ and $v_{i+1}$.

The boundary of $W_i$ is made up of a top and a bottom piece, each of which is smooth, which we denote by $\de W_i^+$ and $\de W_i^-$, respectively. 
There exists a function $\hat F: C_i^+ \to \R$ whose graph is $\de W_i^+$
Indeed, $\de W_i$ is parametrized by
	\[
	\de W_i^\pm = \{(\epsilon((1-\lambda)\vv_i +\lambda\vv_{i+1}), \pm\frac{\epsilon^2}2\omega(\vv_i,\vv_{i+1})(\lambda-\lambda^2)) : \epsilon \in (0,\infty), \lambda \in [0,1]\}.
	\]
	Using this parametrization, we solve for the height function,
	\[
	\hat F_i(v) = \frac{\omega(\vv_i,v)\omega(v, \vv_{i+1})}{2\omega(\vv_i,\vv_{i+1})}.
	\]
	Thus, $W_i = \{F_i\leq0\}$, where $F_i(v,t) = |t| - \hat F_i(v)$, which is smooth except  in the $\{t=0\}$ plane.
	Notice that
	\begin{equation}\label{eq5f7620fb}
		\pD F_i|_{(w,s)}(v,t) =
		\begin{cases}
		\frac{\omega(w,v)}{2} - \frac{ \omega(\vv_i,v)\omega(w,\vv_{i+1}) + \omega(\vv_i,w) \omega(v,\vv_{i+1}) }{ 2\omega(\vv_i,\vv_{i+1}) }
			 &\text{ if }s>0 ,\\
		- \frac{\omega(w,v)}{2} - \frac{ \omega(\vv_i,v)\omega(w,\vv_{i+1}) + \omega(\vv_i,w) \omega(v,\vv_{i+1}) }{ 2\omega(\vv_i,\vv_{i+1}) }
			 &\text{ if }s<0 .
		\end{cases}
	\end{equation}
	If $w=\lambda\vv_i+(1-\lambda)\vv_{i+1}$, then
	\begin{equation}\label{eq5f7627e9}
		\pD F_i|_{(w,s)}(v,t) =
		\begin{cases}
		(1-\lambda)\omega(\vv_{i+1},v)
			 &\text{ if }s>0 , \\
		\lambda\omega(v,\vv_i)
			 &\text{ if }s<0 .
		\end{cases}
	\end{equation}

\begin{proposition}[Blow-ups along wall seams: ceiling]
	Let $p$ be a point on $\de B$ which lies on the seam between the vertical side $\Panel_{i,i+1}$ and the ceiling 
	such that $\pi(p)=u_{i,i+1}(r,s)$ with $r,s \in (0,1]$, one of them equal to 1.
	Then all the blow-ups of $d_e$ at $p$ are 
	\[
	f(v,t) = \begin{cases}
	\alpha_i(v)+c_1 & \omega(\vv_{i+1},v) \leq C \\
	((1-s)\alpha_i + s\alpha_{i+1})(v)+c_2 & \omega(\vv_{i+1},v) > C
	\end{cases} ,
	\]
	for $C\in\R\cup\{+\infty,-\infty\}$.
\end{proposition}
\begin{proof}
	We consider three cases.
	First, if $r=1$ and $s<1$, then $\pi(p)=u_{i,i+1}(1,s)= u_{i-1,i+1}(0,s) \in Q_{i-1,i+1}$.
	Therefore, a neighborhood $\Omega$ of $p$ is decomposed into two regions,
	$\Omega\cap W_i = \Omega\cap\{F_i\le0\}$ and $\Omega\cap \delta_{(0,+\infty)}\Panel^+_{i-1,i+1}=\Omega\cap\{F_i\ge0\}$.
	Therefore, Propositions~\ref{prop5e8c49ee} and~\ref{prop5e8c4f94} with~\eqref{eq5f7627e9} show what are all the blow-ups of this decomposition, 
	while Propositions~\ref{ceilingPD} and~\ref{wallPD} give us the Pansu differentials of $d_e$ near to $p$,
	so that we can conclude using Theorem~\ref{piecetogether}.
	
	Second, if $r<1$ and $s=1$, then $\pi(p)=u_{i,i+1}(r,s)= u_{i,i+2}(r,0) \in Q_{i,i+2}$.
	Then we can proceed like before.
	
	Third, when $r=s=1$, then the point $p$ lies in the boundary of four regions: 
	$W_i$, $\delta_{(0,+\infty)}\Panel^+_{i-1,i+1}$, $\delta_{(0,+\infty)}\Panel^+_{i,i+2}$ and $\delta_{(0,+\infty)}\Panel^+_{i-1,i+2}$.
	However, the function $d_e$ is $C^1_H$ on the union of the latter three and $\pD d_e$ has a continuous extension to $p$.
	It follows that the blow-ups of $d_e$ at $p$ are again the ones listed above.
	This completes the proof.
\end{proof}

A similar result holds for the basement.

\begin{proposition}[Blow-ups along wall seams: basement] 
	Let $p$ be a point on $\de B$ which lies on the seam between the vertical side $\Panel_{i,i+1}$ and the basement 
	such that $\pi(p)=u_{i,i+1}(r,s)$ with $r,s \in (0,1]$, one of them equal to 1.
	Then all the blow-ups of $d_e$ at $p$ are 
	\[
	f(v,t) = \begin{cases}
	\alpha_i(v)+c_1 & \omega(\vv_{i},v) \leq C \\
	((1-r)\alpha_i + r\alpha_{i-1})(v)+c_2 & \omega(\vv_{i},v) > C
	\end{cases} ,
	\]
	for $C\in\R\cup\{+\infty,-\infty\}$.
\end{proposition}
%


\begin{proposition}[Blow-ups along wall seams: vertices]
	Let $p$ be the vertex $\vv_i$ on the unit sphere. 
	Then all the blow-ups of $d_e$ at $p$ are
	\[
	f(v,t) = \begin{cases}
	\alpha_{i}(v) + c_1 & \omega(\vv_i,v) \geq C \\
	\alpha_{i-1}(v) + c_2 & \omega(\vv_i, v) < C
	\end{cases},
	\]
	for $C\in\R\cup\{+\infty,-\infty\}$.
\end{proposition}
\begin{proof}
	The point $p$ belongs to four regions: the two wall cones $W_i$ and $W_{i-1}$, 
	and the cones $P_i^+:=\delta_{(0,+\infty)}\Panel^+_{i-1,i+1}$ and $P_i^-:=\delta_{(0,+\infty)}\Panel^-_{i-1,i+1}$
	Using the formulas for $F_i$ written above, one readily sees that
	\begin{align*}
	\Klim_{\epsilon\to0} \delta_{1/\epsilon}((\vv_i,0)^{-1}{-1}P_i^+)
	&= \{(v,t):\omega(\vv_i,v)\ge0\} \cap \{(v,t):t-\hat F_i(v)\ge0\} , \\
	\Klim_{\epsilon\to0} \delta_{1/\epsilon}((\vv_i,0)^{-1}{-1}W_i)
	&= \{(v,t):\omega(\vv_i,v)\ge0\} \cap \{(v,t):t-\hat F_i(v)\le0\} , 
	\end{align*}
	The union of these two limit cones is the half-space $\{(v,t):\omega(\vv_i,v)\ge0\}$.
	Similarly, $P_i^-$ and $W_{i-1}$ blow-up to 	$\{(v,t):\omega(\vv_i,v)\le0\}$.

	Meanwhile, the function $d_e$ blows up to $\alpha_i$ on both in $P_i^+$ and $W_i$,
	while it blows up to $\alpha_{i-1}$ on both $P_i^-$ and $W_{i-1}$.
	We then conclude.
\end{proof}

\section{Dynamics of the action of $\HH$ on the boundary}\label{sec:dynamics}

One of the main motivations for studying the boundary of a metric space is to then examine how the group of isometries acts on the boundary. Ideally this action on the boundary is simpler than the action on the space itself, and one can hope to glean information about the space or the group through this action. In any Lie group with a left-invariant metric, the group acts isometrically on itself via left translation. In this section, we explore how $\HH$ with a polygonal sub-Finsler metric acts on its horofunction boundary and its reduced horofunction boundary, generalizing results on finitely generated nilpotent groups by Walsh and Bader-Finkelshtein \cite{walsh-orbits, bader-finkel}. Given that our polygonal sub-Finsler metrics are the asymptotic cones of the discrete word metrics, the fact that the results generalize is not overly surprising.

\subsection{Action of the group on the boundary}
Let $d$ be any left-invariant homogeneous metric on $\HH$. To understand how the group acts on the boundary, it suffices to understand how the group acts on sequences. Suppose $\{q_n\}_n$ is a sequence in $\HH$ which converges to a horofunction $f \in \de_h(\HH, d)$. By definition $f(x) = \lim_{n\to\infty} d(q_n, x) - d(q_n, e)$. For a group element $g \in \HH$, the image $g.f(x)$ is the limit of the translated sequence $\{gq_n\}_n$. We have
\begin{align*}
g.f(x) & = \lim_{n\to \infty} d(gq_n, x) - d(gq_n, e)\\
 & = \lim_{n\to\infty} d(q_n, g\inv x) - d(q_n, e) - d(q_n, g\inv) + d(q_n, e) \\
 &= f(g\inv x) - f(g\inv).
\end{align*}

In Lemma~\ref{lem:pansuderiv}, we observed how horofunctions are related to Pansu derivatives and blow-ups of the distance function at points on the unit sphere. In particular, we have shown any horofunction $f$ in the boundary can be realized as a limit
\[
f(x) = \lim_{n\to\infty} \frac{d_e(p_n\delta_{\epsilon_n}x)-d_e(p_n)}{\epsilon_n},
\]
where $p_n \to p \in \de B$ and $\epsilon_n \to 0$. The following lemma shows that $g.f$ similarly is a directional derivative of $d_e$ at the same point $p$.

\begin{lemma}\label{action}
Suppose $f \in \de_h(\HH, d)$ is a blow-up of $d_e$ at a point $p$ on the unit sphere $\de B$. Then for any $g \in \HH$, the boundary point $g.f$ is also a blow-up of $d_e$ at $p$.
\end{lemma}
\begin{proof}
Let $\{p_n\}_n$, $p_n \to p$, and $\{\epsilon_n\}_n$, $\epsilon_n \to 0$, be such that
\[
f(x) = \lim_{n\to\infty} \frac{d_e(p_n\delta_{\epsilon_n}x)-d_e(p_n)}{\epsilon_n}.
\]
The corresponding sequence in $\HH$ which converges to $f$ is $\{q_n\}_n = \{\delta_{1/\epsilon_n}p_n\inv\}_n$. Since $\delta_{\epsilon_n}q_n = p_n\inv \to p\inv$, we say that $q_n$ {\em converges in direction} to $p\inv$. 
This observation implies that $f$ is a blow-up of $d_e$ along a sequence $p_n \to p$ if and only if the sequence $\{q_n\}_n$ which converges to $f$ converges in direction to $p\inv$. If we translate $\{q_n\}_n$ by an element $g\in \HH$, we observe that
\[
\delta_{\epsilon_n}(gq_n) = (\delta_{\epsilon_n}g)(\delta_{\epsilon_n}q_n) \to p\inv.
\]
Thus $gq_n$ also converges in direction to $p\inv$, and $g.f$ is the blow-up of $d_e$ at $p$ along the sequence $\{p_n\delta_{\epsilon_n}g\inv\}_n$ with the same $\{\epsilon_n\}_n$.
\end{proof}

\begin{remark}\label{trivsmooth}
We recall from Proposition~\ref{prop5e8c979c} that $d_e$ is strictly Pansu differentiable at a point $p$, then the blow-up of $d_e$ along any sequences $\{p_n\}_n$ and $\{\epsilon_n\}_n$ satisfying $p_n \to p$ and $\epsilon_n \to 0$ is equal to the Pansu derivative of $d_e$ at $p$. Along with the previous lemma, this implies that if $q_n \to f \in \de_h(\HH, d)$ and $q_n$ converges in direction to a point $p$ where the distance function $d_e$ is strictly Pansu differentiable, then $g.f = f$ for any $g \in \HH$.
\end{remark}

\subsection{Busemann functions} 
Recall that Busemann functions are points of the horofunction boundary which can be realized as limits of geodesic rays. In Corollary~\ref{busemann}, we observe that the set of Busemann functions in the boundary of a polygonal sub-Finsler metric on $\HH$ is homeomorphic to a circle. Indeed, Busemann functions come in two flavors depending on whether they are the blow-ups of vertical wall points or vertices of the unit sphere:
\[
f(q) = f(v, t) = \alpha_i(v) \quad \text{ or }\quad f(q) = f(v,t) = \begin{cases}
\alpha_{i}(v) + c_1& \omega(\vv_i,v) \geq C \\
\alpha_{i-1}(v) + c_2& \omega(\vv_i,v) < C
\end{cases},
\]
where $i \in \{1, \ldots, 2N\}$, $C \in \R$, and $c_1, c_2$ are functions of $C$, determined uniquely by the criteria $f(e) = 0$ and $f$ is continuous.

In \cite{walsh-orbits}, Walsh proves that for any finitely generated nilpotent group, there is a one-to-one correspondence between finite orbits of Busemann functions under the action of the group and facets of a polyhedron defined by the generators of the group. The following proposition generalizes this result to the real Heisenberg group for any polygonal sub-Finsler metric.

\begin{proposition}
In the boundary of a polygonal sub-Finsler metric on $\HH$, there is a one-to-one correspondence between finite orbits of Busemann functions and edges of the metric-inducing polygon $Q$.
\end{proposition}

\begin{proof}
By Remark~\ref{trivsmooth} and also by direct calculation, the action of the group on horofunctions of the form $(v,t) \mapsto \alpha_i(v)$ is trivial. Since the $\alpha_i$ are the blow-ups of $d_e$ on vertical walls of the unit sphere, we get a correspondence between the facets of $Q$ and finite orbits of the action.

It remains to show that no other Busemann functions are fixed globally by the action of the group. For each vertex $\vv_i$ we have a family of blow-ups, in this case Busemann functions,
\[
\mathcal{F}_i = \left\{f(v, t) = \begin{cases}
\alpha_{i}(v) + c_1& \omega(\vv_i,v) \geq C \\
\alpha_{i-1}(v) + c_2& \omega(\vv_i,v) < C
\end{cases} : C \in \R\right\}.
\]
A direct calculation shows that if $g = (w, s)$, $f \in \mathcal{F}_i$, and $\omega(\vv_i, w) \neq 0$, then $g.f \in \mathcal{F}_i$, but $g.f \neq f$.
\end{proof}

\subsection{Trivial action on reduced horofunction boundary}
When defining the horofunction boundary of a metric space, we defined the maps $\iota:X\into\Co(X)$ and $\hat\iota:X\into\Co(X)/\R$. To define the {\em reduced horofunction boundary} we consider the image of $\de_h(X,d)$ in $\Co(X)/\Co_b(X)$, where $\Co_b(X)$ is the space of all continuous bounded functions. It is worth noting that the reduced horofunction is not necessarily Hausdorff, but as we show below, it has value in its strong relationship with the action of the group on $\de_h(X,d)$.

In \cite{bader-finkel}, Bader--Finkelshtein show that the for any finitely generated abelian group and discrete Heisenberg group with any finite generating set, the action of the group on its reduced horofunction boundary is trivial. They further conjecture that this result should hold for any finitely generated nilpotent group. We are able to extend this result to the real Heisenberg group with a polygonal sub-Finsler metric.

\begin{proposition}
Let $d$ be a polygonal sub-Finsler metric on $\HH$. Then the reduced horofunction boundary is in bijection with the quotient of $\de_h(\HH,d)$ by the action of the group. That is
\[
\de_h^r(\HH, d) \leftrightarrow \de_h(\HH,d)/\HH,
\]
and so $\HH$ acts trivially on its reduced horofunction boundary.
\end{proposition}

\begin{proof}
To prove this proposition, it will suffice to look at each of the families of functions described in the Theorem~\ref{thm5f7476ff}.

We start by considering the three smooth families of horofunctions, which compose a circle in the boundary. These boundary points are all Pansu derivatives, and hence are linear. It is clear that two linear functions stay bounded distance from one another if and only if they are identical, and so each Pansu derivative remains distinct in the reduced horofunction boundary.  By the definition of action on the boundary, it is clear that if $f$ is linear, then $g.f = f$ for all $g \in \HH$, and so the action on these points in $\de_h^r(\HH,d)$ is trivial.

Next we consider the piecewise-linear horofunctions coming from the blow-ups of non-smooth points. Any (nontrivially) piecewise linear function cannot have bounded difference from a linear function, and so they cannot be equivalent in the reduced horofunction boundary to the smooth families mentioned above. Our goal is to show that two horofunctions $f_1$ and $f_2$ differ by a bounded function if and only if $f_1$ and $f_2$ belong to the same orbit.
Let $f_1$ and $f_2$ be distinct functions coming from the same family of functions in Theorem~\ref{thm5f7476ff}.

\underline{Case 1:} Suppose that for $j=1,2$, we have $f_j = \norm{w_j} - \norm{w_j-v}$, $w_1 \neq w_2$. Then
\begin{align*}
|f_2(v,t) - f_1(v,t)| &= \mathlarger|\norm{w_2} - \norm{w_2 - v} -(\norm{w_1} - \norm{w_1 - v})\mathlarger|\\
	& \leq\mathlarger|\norm{w_2} - \norm{w_1}\mathlarger| + \mathlarger| \norm{w_1 - v}- \norm{w_2 - v}\mathlarger|\\
	& \leq 2\norm{w_2 -w_1},
\end{align*}
and so $f_1$ and $f_2$ are identified in the reduced boundary. It remains to show that they lie in the same orbit. Let $g = (w_2 -w_1, 0)$. Then
\begin{align*}
g.f_1(v,t) &= \norm{w_1} - \norm{w_1 - (v - (w_2 - w_1))} - (\norm{w_1} - \norm{w_1 - (w_2 - w_1)})\\
	& = \norm{w_2} - \norm{w_2 - v} = f_2(v,t).
\end{align*}
This calculation also shows that for any $g \in \HH$, $g.f_1$ will lie in this same family of blow-ups.

\underline{Case 2:} The remaining families of functions 
are the images of $(0,1] \times \R$ under the maps $\psi_i^\vee$, $\psi_i^\wedge$, $\xi_i^\vee$, and $\xi_i^\wedge$, $i \in \{1,\ldots,2N\}$. It is clear that no function from these families can have bounded difference with a function from Case 1. Indeed, the norm-like functions of Case 1 are piecewise linear, where $2N$ distinct functions are defined on $2N$ regions, for $N>1$. Meanwhile, the images of $\psi_i^\vee$, $\psi_i^\wedge$, $\xi_i^\vee$, and $\xi_i^\wedge$ are defined by two functions defined on two halfspaces. There is, therefore, an unbounded region in the plane where the two functions have unbounded difference.

Consider two functions $f_1$ and $f_2$ from these families, where $f_j$, $j=1,2,$ is the image of $(s_j, a_j)$,  $s_j\in(0,1]$, $a_j \in \R$, under a map in $\{\psi_{i_j}^\vee, \psi_{i_j}^\wedge, \xi_{i_j}^\vee, \xi_{i_j}^\wedge\}$, for index $i_j \in \{1, \ldots, 2N\}$. We omit the cases where $s_j = 0$ or $a_j \in \{-\infty,\infty\}$, as these cases result in linear functions, already discussed above. We claim that $f_1$ and $f_2$ have bounded difference if and only 1) $s_1 = s_2$; 2) $i_1 = i_2$; and 3) they both are the image under the same map. Indeed, if any of these three conditions is not met, a direct inspection of the functions $\psi_i^\vee$, $\psi_i^\wedge$, $\xi_i^\vee$, and $\xi_i^\wedge$ makes clear that there is an unbounded region on which $f_1$ and $f_2$ are defined as distinct linear functions and have unbounded difference. On the other hand, if the three conditions are met, we must show that $f_1$ and $f_2$ have bounded difference. For $j=1,2$, let $f_j$ be linear on the two regions $U_j= \{(v,t): \omega(\vv_i,v) \geq C_j\}$ and $L_j= \{(v,t): \omega(\vv_i,v) \geq C_j\}$. To analyze the difference between $f_1$ and $f_2$, we assume $C_1 > C_2$ and consider three regions:
\begin{align*}
\Omega_1 &= U_1\cap U_2 =  \{\omega(\vv_{i},v) \geq C_1\} && \Omega_3 = L_1 \cap L_2 = \{\omega(\vv_{i},v) \leq C_2\}\\
 \Omega_2 &=  L_1 \cap U_2 = \{C_2 < \omega(\vv_{i},v) < C_1\}
\end{align*}
Since $s_1 = s_2$, $i_1 = i_2$, and $f_1$, $f_2$ are both images under the same map, $(f_1 - f_2)|_{\omega_1}$ and $(f_1 - f_2)|_{\omega_3}$ are constant. Meanwhile $f_1$ and $f_2$ are distinct on the unbounded strip $\Omega_2$. Since $f_1$ and $f_2$ are continuous, the level sets of the $f_j$ in $\Omega_2$ must be transverse (not parallel) to the strip $\Omega_2$, guaranteeing that the functions have bounded difference.

Finally, by choosing an element $g \in \HH$ such that $\omega(\vv_i, \pi(g)) = C_1-C_2$, one can confirm that $f_1$ and $f_2$ are in the same orbit.
\end{proof}

\begin{figure}[H]
\begin{tikzpicture}[scale=0.53, every node/.style={scale=0.8}]

\draw[very thick, black!40!green] (-3,3) to[bend left = 20] (0,3);
\draw[very thick, orange] (-3,3) to[bend left = 40] (0,3);
\draw[very thick, white!40!cyan] (-3,3) to[bend left = 60] (0,3);

\draw[very thick, black!40!green] (3,-3) to[bend left = 20] (0,-3);
\draw[very thick, orange] (3,-3) to[bend left = 40] (0,-3);
\draw[very thick, white!40!cyan] (3,-3) to[bend left = 60] (0,-3);

\draw[very thick, black!40!green] (3,0) to[bend left = 20] (3,-3);
\draw[very thick, orange] (3,0) to[bend left = 40] (3,-3);
\draw[very thick, white!40!cyan] (3,0) to[bend left = 60] (3,-3);

\draw[very thick, black!40!green] (-3,0) to[bend left = 20] (-3,3);
\draw[very thick, orange] (-3,0) to[bend left = 40] (-3,3);
\draw[very thick, white!40!cyan] (-3,0) to[bend left = 60] (-3,3);

\draw[very thick, black!40!green] (-3,0) to[bend left = 20] (-3,3);
\draw[very thick, orange] (-3,0) to[bend left = 40] (-3,3);
\draw[very thick, white!40!cyan] (-3,0) to[bend left = 60] (-3,3);

\draw[very thick, black!40!green] (0,-3) to[bend left = 20] (-3,0);
\draw[very thick, orange] (0,-3) to[bend left = 40] (-3,0);
\draw[very thick, white!40!cyan] (0,-3) to[bend left = 60] (-3,0);

\draw[very thick, black!40!green] (0,3) to[bend left = 20] (3,0);
\draw[very thick, orange] (0,3) to[bend left = 40] (3,0);
\draw[very thick, white!40!cyan] (0,3) to[bend left = 60] (3,0);

\fill[violet] (3,-1.5) circle (4 pt);
\fill[violet] (-1.5,3) circle (4 pt);
\fill[violet] (1.5,1.5) circle (4 pt); 
\fill[violet] (-3,1.5) circle (4 pt);
\fill[violet] (-1.5,-1.5) circle (4 pt);
\fill[violet] (1.5, -3) circle (4 pt);
\fill[violet] (0,0) circle (4 pt);

\fill[blue] (3,0) circle (4 pt);
\fill[blue] (0,3) circle (4 pt);
\fill[blue] (-3,3) circle (4 pt); 
\fill[blue] (-3,0) circle (4 pt);
\fill[blue] (0,-3) circle (4 pt);
\fill[blue] (3, -3) circle (4 pt);

\fill[red] (4,-1.5) circle (4 pt);
\fill[red] (-1.5,4) circle (4 pt);
\fill[red] (2.6,2.6) circle (4 pt); 
\fill[red] (-4,1.5) circle (4 pt);
\fill[red] (-2.6,-2.6) circle (4 pt);
\fill[red] (1.5, -4) circle (4 pt);

\begin{scope}[xshift=-11cm]
\node at (0,0){\includegraphics[width=6cm]{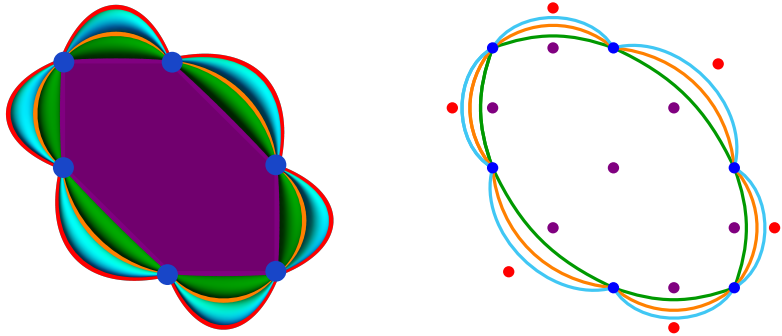}};
\end{scope}
\end{tikzpicture}
\caption{The standard and reduced horofunction boundaries for a hexagonal sub-Finsler metric.}
\label{reduced}
\end{figure}


\bibliography{BIB-Horoboundary.bib}{}
\bibliographystyle{plain}

\end{document}